\documentclass{article}
\usepackage[journal=JEMS,lang=british]{ems-journal-jems}

\usepackage{color}
\usepackage{mathtools}
\mathtoolsset{showonlyrefs}
\usepackage{verbatim} % adds environment for commenting out blocks of text & for better verbatim
\usepackage{subfig} % make it possible to include more than one captioned figure/table in a single float
\usepackage{mathtools}
\newcommand{\mel}{\MoveEqLeft}
\usepackage{amsthm}

\newtheorem{theorem}{Theorem}[section]

\newtheorem{remark}[theorem]{Remark} 

\newtheorem{example*}{Example\textsuperscript{*}}
\newtheorem{proposition*}{Proposition\textsuperscript{*}}
\newtheorem{corollary}[theorem]{Corollary}
\newtheorem{corollary*}{Corollary\textsuperscript{*}}
\newtheorem{proposition}[theorem]{Proposition}
\newtheorem{lemma}[theorem]{Lemma}

\numberwithin{equation}{section}

\usepackage{bbm}
\newcommand{\1}{\mathbbm{1}}
\newcommand{\la}{\langle}
\newcommand{\ra}{\rangle}

\renewcommand{\Re}{\operatorname{Re}}

\newcommand{\psiin}{\psi_{\mathrm{int}}}
\newcommand{\psiout}{\psi_{\mathrm{ext}}}

\newcommand\eps{\epsilon}

\newcommand\R{\mathbb{R}}
\newcommand\HH{\mathbb{H}}

\newcommand\T{\mathbb{T}}
\newcommand\Z{\mathbb{Z}}
\newcommand\F{\mathcal{F}}
\renewcommand{\S}{\mathcal{S}}

\newcommand{\K}{\mathcal{K}}

\newcommand{\cL}{\mathcal{L}}

\newcommand\M{\mathcal{M}}
\newcommand{\Chi}{\mathcal{X}}
\newcommand\sym{\text{sym}}
\newcommand\N{\mathbb{N}}

\newcommand\grad{\nabla}

\renewcommand{\div}{\grad\cdot}

\DeclareMathOperator{\const}{const}
\DeclareMathOperator{\curl}{curl}

\DeclareMathOperator{\spann}{Span}

\DeclareMathOperator{\tr}{tr}

\newcommand\norm[1]{\left\lVert #1 \right\rVert}
\newcommand\scalar[2]{\langle #1,#2 \rangle}
\newcommand\de{\partial}
\renewcommand{\div}{\operatorname{div}}

\newcommand\dd{\,\mathrm{d}}
\newcommand\dx{\,\mathrm{d}x}

\newcommand\dy{\,\mathrm{d}y}
\newcommand\dt{\,\mathrm{d}t}
\newcommand\ds{\,\mathrm{d}s}

\newcommand\da{\,\mathrm{d}\alpha}

\newcommand{\mres}{\mathbin{\vrule height 1.6ex depth 0pt width
0.13ex\vrule height 0.13ex depth 0pt width 1.3ex}}

\newcommand\Span{\mathrm{span}}
\newcommand{\vin}{\varphi_{\mathrm{int}}}
\newcommand{\vout}{\varphi_{\mathrm{ext}}}
\newcommand{\rin}{\rho_{\mathrm{int}}}
\newcommand{\rout}{\rho_{\mathrm{ext}}}
\newcommand{\Ha}{\mathcal{H}}
\newcommand{\D}{\mathcal{D}}
\newcommand{\Din}{\mathcal{D}_{\mathrm{int}}}
\newcommand{\Dout}{\mathcal{D}_{\mathrm{ext}}}
\newcommand{\bin}{b_{\mathrm{int}}}
\newcommand{\bout}{b_{\mathrm{ext}}}
\newcommand{\We}{\mathrm{We}}

\usepackage{hyperref}

\emsauthor{1}{
	\givenname{David}
	\surname{Meyer}
	\mrid{1611141}
	\zblid{meyer.david-tobias}
	\orcid{0000-0002-1174-7519}}{D.~Meyer}

\emsauthor{2}{
	\givenname{Christian}
	\surname{Seis}
	\mrid{924781}
	\zblid{seis.christian}
	\orcid{0000-0003-4203-8341}}{C.~Seis}

\Emsaffil{1}{
	\pretext{}
	\department{Instituto de Ciencias Matem\'aticas}
	\organisation{Consejo Superior de Investigaciones Cient\'ificas}
	\rorid{05e9bn444}
	\address{Calle de Nicolas Cabrera 13-15}
	\zip{28049}
	\city{Madrid}
	\country{Spain}
	\posttext{}
\affemail{david.meyer@icmat.es}}
	
\Emsaffil{2}{
	\pretext{}
	\department{Institut für Analysis und Numerik}
	\organisation{Universität Münster}
	\rorid{00pd74e08}
	\address{Einsteinstr. 62}
	\zip{48149}
	\city{Münster}
	\country{Germany}
	\posttext{}
\affemail{seis@uni-muenst.der}}

\begin{document}

\title{Steady ring-shaped vortex sheets}

\classification[35C07, 35J25]{76B47}


\begin{abstract}
In this work, we construct traveling wave solutions to the two-phase Euler equations, featuring a vortex sheet at the interface between the two phases. The interior phase exhibits a uniform vorticity distribution and may represent a vacuum, forming what is known as a \emph{hollow vortex}. These traveling waves take the form of ring-shaped vortices with a small cross-sectional radius, referred to as \emph{thin rings}.  Our construction is based on the implicit function theorem, which also guarantees local uniqueness of the solutions. Additionally, we derive asymptotics for the speed of the ring, generalizing the well-known Kelvin--Hicks formula to cases that include surface tension.\end{abstract}

\maketitle

\section{Introduction}

We focus on free boundary problems for the incompressible Euler equations in $\R^3$, with particular attention to constructing steady vortex ring solutions. Vortex rings are toroidal regions where the vorticity is concentrated. Within the ring, the fluid rapidly rotates around a central curve, while the ring itself moves through the surrounding fluid in a direction orthogonal to the plane containing the central curve.

When the ring consists of a different phase than the surrounding fluid, such as an air bubble ring in water, a discontinuity in shear velocity occurs along the ring's surface. This discontinuity creates a singularity in the vorticity distribution, turning the surface into a vortex sheet.

The presence of vortex sheets presents a significant challenge in the mathematical treatment of the Euler equations. Vortex sheets are considered ill-posed in the Hadamard sense because the Kelvin--Helmholtz instabilities cause small fluctuations around a steady state to grow exponentially over time \cite{Ebin88,Wu06}. However, this situation can be mitigated if surface tension is present, as it exerts a regularizing effect. The Cauchy problem for the Euler equations with surface tension is locally in time well-posed \cite{Ambrose03,
ambrose2007well,ChengCoutandShkoller08,ShatahZeng08,ShatahZeng11}. Yet,  singularities may occur in finite time \cite{castro2012finite,castro2013finite,Coutand19}.

The aim of this work is to construct steady, ring-shaped vortex sheets with a small cross-sectional radius $ \bar \eps$, considering cases both with and without surface tension. Inside the ring, we examine either a uniform vorticity distribution or an irrotational or vacuum setting, which leads to the formation of hollow vortex rings. When surface tension is present, the strength of inertia forces compared to surface tension will be measured in terms of the non-dimensional Weber number \eqref{501}. This study is the first to rigorously establish vortex rings with surface tension as well as hollow vortex rings.

The first explicit example of a steady vortex ring solution to the Euler equations in $\R^3$ was provided by Hill in 1894 \cite{Hill1894}, who actually constructed a \emph{spherical} vortex. The first \emph{toroidal} examples, though non-explicit, were found by Fraenkel in 1970 \cite{fraenkel1970steady} and Norbury in 1972 \cite{Norbury72}. Fraenkel’s work considers rings with a small cross-sectional radius, similar to ours, while Norbury’s rings are thicker and represent toroidal perturbations of Hill’s spherical vortex. Numerous subsequent constructions have been developed for vortex rings, including those with arbitrary cross-sectional radii; see, for instance, \cite{fraenkel1974global,BergerFraenkel80,friedman1981vortex,ambrosetti1989existence,Yang95,BadianiBurton01,Burton03,MR4404610,CaoQinYuZhanZou22+}. Hollow vortices in 2D have been constructed e.g.\ in \cite{chen2023desingularization}. In a companion paper with Niebel \cite{NiebelMeyerSeis24}, we establish near-spherical (non-toroidal) perturbations of Hill's vortex, featuring a vortex sheet interface similar to the one in the present work. Additionally, there are many examples of traveling wave solutions in various fluid contexts, including water waves and viscous flows, see e.g.\ \cite{leoni2023traveling,haziot2022traveling,constantin2004exact,berti2023time,hassainia2022multipole,constantin2016global} and the references therein. 

In the regime of small cross-sectional radius $\bar \eps$, it is possible to derive asymptotic expressions for various quantities of interest. Most notably,  when surface tension is absent, the (constant) speed $\bar W$ of such rings  is well-known to obey the asymptotic Kelvin--Hicks formula
\begin{equation}
    \label{401}
\bar W= \frac{\bar{b}_{\mathrm{ext}}}{4\pi R} \left(\log \frac{8R}{\bar \eps} -c\right) +o(1),
\end{equation}
if $\eps=\bar \eps/R\ll1$. Here,  $\bar{b}_{\mathrm{ext}}$ is the Kelvin circulation (see \eqref{402} below), $R$ the inner radius of the ring, and $c$ is a number dependent on the vorticity distribution in the interior. The expression was first discovered by Helmholtz in 1858  \cite{helmholtz1858integrale}. In 1867, Kelvin determined the value $c=1/4$ for a uniform vortex  \cite{Kelvin1867} and in 1883,  Hicks found the value $c=1/2$ for a hollow vortex   \cite{Hicks1883}. This formula was later rigorously validated by steady solutions in \cite{fraenkel1970steady}. For non-steady solutions, the leading order term in \eqref{401} was verified in \cite{benedetto2000motion,jerrard2017vortex}. In our work, we will recover the Kelvin--Hicks formula within our framework and extend it to include the case with surface tension. For a detailed survey of the history of vortex theory since  Helmholtz's time, we refer to \cite{meleshko2012vortex}.

Vortex rings with a small cross-sectional radius occasionally occur in nature, such as the smoke rings recently emitted by Sicily’s Mount Etna, which quickly captured global attention, or the air bubble rings created by dolphins and beluga whales. These fascinating structures can also be produced by tobacco smokers, scuba divers, or experimentalists in laboratories. Our analysis provides a mathematical construction of such bubble rings.

\subsection{Original mathematical description}

The model that comprises all these situations is the two-phase Euler equation. We denote by $\Din(t)$ the open domain of the interior  phase with the surface $\S(t)=\partial \Din(t)$ and by   $\Dout(t) = (\overline{\Din(t)})^c$ the open domain of the exterior phase, so that 
\[ 
\R^3 = \Din(t) \cup \S(t)\cup \Dout(t).
\]

Inside of the two phases, the evolution is described by the incompressible Euler equations,
\begin{align}
\bar \rho(\partial_t u + u\cdot \grad_x u) + \grad_x p & =0\quad \mbox{in }\R^3\setminus\S(t),\label{1}\\
\div_x u&=0\quad \mbox{in }\R^3\setminus \S(t),\label{2}
\end{align}
describing the conservation of momentum and mass. Here, $u=u(t,x)\in\R^3$ is the velocity of the fluid, $p=p(t,x)\in\R$ is the pressure, and $\bar \rho=\bar \rho(t,x)\in\R_{\ge0}$ is the mass density.  The latter is assumed to be constant in the respective phases so that $\Din(t) = \{\bar \rho(t)= \rho_\mathrm{int}\}$ and $\Dout(t)=\{\bar \rho(t)= \rho_{\mathrm{ext}}\}$. We will assume that $\rout\in\R_{>0}$ with $\rout\ge \rin$.

Across the interface, the tangential velocity component may experience a discontinuity (turning $\S(t)$ into a vortex sheet), though its normal component must be continuous,
\begin{equation}\label{4}
    [u\cdot n]  = 0\quad \mbox{on }\S(t).
\end{equation}
Here $[f]= f_{\mathrm{int}} - f_{\mathrm{ext}}$ measures the jump across the interface  of a discontinuous quantity $f= f_{\mathrm{ext}} \1_{\Dout} + f_{\mathrm{int}}\1_{\Din}$, and $n$ is the unit normal vector pointing from $\Din (t)$ into $\Dout(t)$.
The Young--Laplace equation 
\begin{equation}
[p] = \bar \sigma H \quad \mbox{on }\S(t),\label{3}
\end{equation}
 in which $\bar \sigma$ is the surface tension and $H$ is the mean curvature, relates the pressure difference to the geometry of the surface. 
We are using the convention that the mean curvature is the sum of the two principal radii $H=(\kappa_1+\kappa_2)$, chosen to be positive for convex $\Din(t)$.

Finally, the two phases itself are advected by the flow, 
\begin{equation}\label{5}
    \partial_t\bar \rho +u\cdot \grad\bar\rho = 0\quad \mbox{in }\R^3.
\end{equation}

 The model reduces to the free boundary Euler equation if the interior phase is a vacuum, $ \rho_\mathrm{int}=0$, in which case the interior pressure is a constant, and the jump condition in \eqref{3} can be rewritten as $-p = \bar \sigma H$, because the pressure is defined only up to an additive constant. Allowing for $ \rho_\mathrm{int}=0$ in our choice of densities, it is enough to consider the general two-phase setting with jump condition \eqref{3} in the following.

Steady vortex rings are traveling wave solutions, moving at a constant speed $\bar W$ in a fixed direction, say, $e_z$, i.e.,
\[
u(t,x)  = U(x-t\bar W e_z),\quad p(t,x) = P(x-t\bar W e_z).
\]
In the new variables, the two phases and their interface $\S$ are frozen in time and the momentum equation becomes
\[
\bar \rho (U-\bar W e_z)\cdot \grad_x U + \grad_x P = 0\quad \mbox{in }\R^3\setminus \S.
\]
In this formulation, $\bar \rho$ is a two-valued constant. The relative velocity   $U- \bar We_z$  is now tangential to the ring's surface,
\begin{equation}
    \label{400}
(U-\bar W e_z)\cdot n=0\quad \mbox{on }\S.
\end{equation}

  It is not restrictive to assume a cylindrical symmetry about the $e_z$ axis for our vortex ring, so that
\[
U = U(r,z),\quad  U = U^re_r + U^z e_z,\quad P = P(r,z).
\]
There is no swirl velocity pointing in the angular direction, and for symmetry reasons, the velocity is tangential at the $z$-axis, $U^r=0$.
We caution the reader that in these \emph{axi-symmetric} variables, the velocity is not divergence-free with respect to $r$ and $z$. Instead, it holds that
\begin{equation}\label{408}
\partial_r (rU^r) + \partial_z(rU^z)=0.
\end{equation}
We will denote by $\Omega$ the cross-sectional slice of the ring $\D_\mathrm{int}$ that lies in the meriodial half plane $\HH = \{(r,z): r>0\}$.

In two-dimensional steady flows, a well-known advantage is that the pressure and the kinetic energy density are related by the Bernoulli equation (see, e.g., Chapter 1.9 in \cite{Saffman92}) along every streamline. Specifically, this relationship implies that
\begin{equation}
    \label{412}
[ {\bar \rho} (U\cdot \tau -\bar W \tau \cdot e_z)^2 + 2P ] =\bar \nu\quad \mbox{on }\partial \Omega
\end{equation}
for some $\bar \nu\in\R$, which we will refer to as the \emph{Bernoulli constant}. Here and in the following, $\tau$ is the tangent on $\partial \Omega$ obtained from the normal $n$ via a counter-clockwise rotation by 90 degrees, so that $\tau=n^{\perp}$. The above identity will be used to eliminate the pressure term in the  Young--Laplace equation \eqref{3}, see \eqref{26} below.

We remark for later reference that in the axisymmetric setting,  the mean curvature $H$ of the two-dimensional cylindrical surface  $\S$ reduces  to the curvature $\kappa$ of the one-dimensional surface $\partial \Omega$ via 
\begin{equation}\label{413}
H   =\frac{\div(rn)}r = \div n + \frac{n\cdot e_r}r = \kappa + \frac{n\cdot e_r}r.
\end{equation}

 It remains to introduce two quantities that are of considerable importance in vortex theory. The first is the  \emph{Kelvin circulation} $\bar{b}_{\mathrm{ext}}$ around the vortex ring, which is computed along the boundary curve $\partial \Omega$ of the vortex cross-section,
\begin{equation}
    \label{402}
\bar{b}_{\mathrm{ext}} = - \int_{\partial \Omega} U_{\mathrm{ext}}\cdot \tau\dd\Ha^1.
\end{equation}
It describes the cumulative shear velocity around the ring and is proportional to the ring's speed given by the Kelvin--Hicks formula \eqref{401}. By Kelvin's circulation theorem, the circulation along any (unsteady) evolution is an integral of the Euler equations.
The second quantity is the  \emph{interior circulation},
\begin{equation}\label{402a}
\bar{b}_{\mathrm{int}}  = \int_{\Omega} \curl U\dd (r,z) = \int_{\Omega} (\partial_z U^r-\partial_r U^z)\dd x = - \int_{\partial \Omega} U_{\mathrm{int}}\cdot \tau\dd\Ha^1,
\end{equation}
 where the last identity can be easily verified by an application of Stokes' theorem. 
It is evident that in the regular setting, in which the tangential velocity does not exhibit a jump discontinuity, both quantities $\bar{b}_{\mathrm{int}}$ and $\bar{b}_{\mathrm{ext}}$ agree. In our construction, we will consider the situation in which the fluid is irrotational outside of the ring and the vorticity is uniform in its interior. Because it is the potential vorticity $r^{-1}\curl U$ that is transported by the axisymmetric Euler equations rather than the vorticity itself (see, for instance, Chapter 1.5 in \cite{Saffman92}), it is natural to assume that $\curl U = \bar \xi r$ for some $\bar \xi\in\R$. In the case where the  cross-sectional radius $\bar \eps$ of the vortex ring is small compared to the inner radius $R$, so that $\Omega$ is comparable to $ B_{\bar \eps}(R,0)$, it must then hold that
\begin{equation}
    \label{411}
    \bar{b}_{\mathrm{int}} \approx \pi R\bar \eps^2\bar \xi.
\end{equation}

The problem under consideration is controlled by two physical non-dimensional  parameters. One is the circulation-based inertia ratio
\begin{equation}
    \label{500}
\rho = \frac{1}{(4\pi)^2}\frac{\bar{b}_{\mathrm{int}}^2}{\bar{b}_{\mathrm{ext}}^2}\frac{\rin}{\rout},
\end{equation}
which measures the relative strength of the interior vs. exterior vortex motion. The numerical prefactor is included for notational convenience in our later analysis. The other parameter is the circulation-based  Weber number (which may depend on $\eps$)
\begin{equation}
    \label{501}
\mathrm{We} =\frac{\rout \bar{b}_{\mathrm{ext}}^2}{\bar \eps\bar \sigma},
\end{equation}
which measures the inertial force of the ambient fluid relative to the surface tension.

 The goal of this work is to find an axisymmetric traveling wave vortex ring solution to the free-boundary Euler equation \eqref{1}, \eqref{2}, \eqref{4}, \eqref{3}, \eqref{5}, with cross-sectional domain  $\Omega$ (in $\HH$) that is close to a disc $B_{\bar \eps}(R,0)$. We restrict to small cross-sectional radii, $\eps=\bar \eps/R\ll1$.

We present an initial version of our main result, where we intentionally remain somewhat non-specific. For a mathematically precise statement, please refer to Theorem \ref{T2}.

\begin{theorem}[First non-specific formulation]
    \label{Thm1}
    Let $\bar{b}_{\mathrm{int}}\in\R$,   $\bar{b}_{\mathrm{ext}}\in\R_{>0}$, $\rin\in \R_{\ge0} $, $\rout\in \R_{>0}$, $R\in\R_{>0}$, $\bar \eps\in\R_{>0}$, and $\bar \sigma\in\R_{\ge 0}$ be given in such a way that $\rho$ defined in \eqref{500} is fixed and that the Weber number $\We\in(0,\infty]$ defined in \eqref{501} is a $C^1$ function of the aspect ratio $ \eps=\bar \eps/R$ that is  either identical infinity, $\We = \infty$, or it is strictly finite, $\We_{\eps}\in(0,\infty)$,  with
    \begin{gather}
            \lim_{ \eps\to 0}  \frac{\eps}{\We_{ \eps}}=0,\label{404}\\
            \lim_{ \eps\to 0} \We_{\eps}\in[0,\infty)\setminus \left(\frac1{4\pi^2} +4\rho\right)^{-1} \N_{\ge 3},\label{405}\\
             \eps |\frac{\dd}{\dd  \eps} \We_{ \eps}|\lesssim  \We_{ \eps} .\label{406}
    \end{gather}
    Then there exists an $ \eps_0\in\R_{>0}$ such that for all $ \eps\in(0, \eps_0)$, there is a unique axisymmetric traveling wave vortex ring solution with cross-sectional domain $\Omega$ close to $B_{\bar \eps}(R,0)$. The speed of the ring is approximately given by
    \begin{equation}
        \label{407}
        \bar W =  \frac{\bar{b}_{\mathrm{ext}}}{4\pi R}\left(\log\left(\frac{8R}{\bar \eps}\right)-\frac{1}{2} + \frac14\frac{\bar{b}_{\mathrm{int}}^2}{\bar{b}_{\mathrm{ext}}^2} \frac{ \rho_{\mathrm{int}}}{ \rho_{\mathrm{ext}}}\right) +\frac{\pi R}{\bar{b}_{\mathrm{ext}}\rho_{\mathrm{ext}}}\bar \eps \bar\sigma_{\eps} +o(1),
    \end{equation}
    as $\eps=\bar\eps/R\ll1$.
 \end{theorem}

The theorem provides the first rigorous construction of vortex rings with surface tension, as well as hollow vortex rings. The asymptotic formula for the speed in \eqref{407} generalizes the Kelvin--Hicks formula. Specifically, as mentioned earlier, in the one-fluid case where $\rho_{\mathrm{int}} =\rho_{\mathrm{ext}}$  and $\bar \sigma = 0$, the Kelvin circulation \eqref{402} and the interior circulation \eqref{402a} agree, $\bar{b}_{\mathrm{int}} \approx \bar{b}_{\mathrm{ext}}$. This allows us to recover the Kelvin--Hicks formula \eqref{401} with the constant $c = 1/4$. Similarly, if the interior is irrotational ($\bar{b}_{\mathrm{int}} = 0$) or a vacuum ($\rho_{\mathrm{int}} = 0$), we find that $c = 1/2$. The formula including surface tension was derived (non-rigorously) in Section 4 of \cite{MooreSaffman72} for a hollow vortex.

The uniqueness statement made in the theorem is intentionally kept vague. We kindly refer the impatient reader to Theorem \ref{T2} for more details on the solution class.

The condition on the (finite) Weber number in \eqref{404}   ensures that $\We_{  \eps}$ decays not faster than $\eps$ for $ \eps\ll1$. This lower bound is necessary to make sure that the linearized operator (see \eqref{Fourier sym}) does not degenerate. For technical reasons, we have to exclude in \eqref{405} that the Weber number blows up. For uniformly finite Weber numbers, a discrete set of values is excluded if $\We_{\eps}$. For these values, the linear operator associated with the jump condition \eqref{3} will degenerate, which we suspect results in a saddle point bifurcation. For a thorough explanation of this exclusion,  we kindly refer the reader again to Theorem \ref{T2} and the subsequent discussion.
The growth condition in \eqref{406} is rather mild and is satisfied by any reasonable monotone function in the specified range.

We consider it as important to notice that the choice $ \We_{\eps}\sim 1/|\log { \eps} |$ is admissible under the conditions \eqref{404}, \eqref{405}, and \eqref{406}. This scaling is critical for rings with small cross-sectional radius, $ \eps=\bar \eps/R\ll1$, as it balances both kinetic  and surface energy in the two-phase Hamiltonian 
\[
 \frac12\int_{\R^3}\bar \rho |u|^2\dd x + \bar \sigma\Ha^2(\S)   
\]
associated with  the Euler system \eqref{1}, \eqref{2}, \eqref{4}, \eqref{3}, \eqref{5}. Indeed, the kinetic energy of a thin ring grows asymptotically as $0.5\bar{\rho}_{\mathrm{ext}}\bar{b}_{\mathrm{ext}}^2R|\log  \eps|$, see, e.g., Section 10.2 in \cite{Saffman92}, while the surface area decreases proportionally to  $4\pi^2 R\bar \eps$.
Another relevant scaling is $
\We\sim 1$, which is the critical scaling under which the limiting problem in 2D is scaling-invariant. We will elaborate on this later.

\subsection{An overdetermined boundary problem}

To simplify the notation in the following, we will interpret all differention symbols with respect to the cylindrical variables $(r,z)$, so that, for instance, $\grad = (\partial_r,\partial_z)^T$, $\div = \grad\cdot$, and $\curl = -\grad^{\perp}\cdot$.

It is a standard approach to rewrite the steady Euler system as an elliptic problem. Since $rU(r,z)$ is divergence-free by \eqref{408}, there exists a function $\psi$ such that $U = r^{-1}\grad^{\perp}\psi$ in $\R^3$. This function is commonly referred to as the \emph{stream function}. It is determined by the following elliptic Dirichlet problem:
\begin{align}
 - \div\left(\frac1r\grad\psi\right) &= \curl U\quad \mbox{in }\HH,\label{409}\\
\psi&=0\quad \mbox{on }\partial \HH, \label{30}\\
\frac1r|\grad\psi|&\to 0\quad\mbox{as }|(r,z)|\to \infty.\label{31}\\
\lim_{r\rightarrow 0}\frac{1}{r}|\nabla\psi|&<\infty.\label{32b}
\end{align}
Here the last condition encodes the fact that $u$ should be non-singular at the $z$-axis.
We remark that the Dirichlet boundary condition at the $z$-axis is consistent with the requirement that the velocity must be tangential there so that $\partial_z \psi=0$.
As mentioned in the introduction, the vorticity $\curl U$ exhibits a singularity at the surface $\partial \Omega$ of the vortex ring's cross-section, characterized by the (negative) jump in the tangential velocity component.  The magnitude of the shear velocities along this surface is not known a priori, and determining them (implicitly) will be part of our analysis.

As outlined above, to simplify the problem, we focus on uniform vorticity distributions inside of the ring, leading to the assumption $\curl U = r\bar \xi$ in $\Omega$ for some constant $\bar \xi \in \R$, and we suppose that the fluid outside of the ring is irrotational, so that $\curl U=0$ in $\overline{\Omega}^c$. We thus relax \eqref{409} to 
\begin{align}\label{27}
    -\div\left(\frac1r\grad \psi\right) &=\begin{cases} \bar \xi r&\quad\mbox{in }\Omega,\\
  0&\quad\mbox{in }\overline \Omega^c,\end{cases}
\end{align}
where  the  complement is taken with respect to the half plane $\HH$. 
From the tangential flow condition in \eqref{400} we also  derive a Dirichlet boundary condition at the surface. Since $\psi - \bar W/2 r^2$ must be constant along $\partial \Omega$, there exists a  $\bar \gamma\in\R_{\ge 0}$ so that
\begin{equation}\label{29}
    \psi  - \frac{\bar W}2r^2 -\bar \gamma=0 \quad\mbox{on }\partial \Omega.
\end{equation}
In the literature, $\bar \gamma$ is referred to as the \emph{flux constant}. The non-negativity of $\bar \gamma$ is a consequence of the maximum principle.

Next, we reformulate the Young--Laplace equation using the stream function. By replacing the pressure jump in \eqref{3} with the jump in the squared shear velocity, as given by the Bernoulli-type equation \eqref{412}, and expressing the velocity in terms of the stream function, we derive the following updated jump condition:
\begin{equation}\label{26}
\left[  \bar \rho\left( \frac{1}{r}\partial_n \psi -  \bar W n\cdot e_r\right)^2\right] + 2\bar \sigma H =\bar \nu \quad \mbox{on }\partial \Omega.
\end{equation}
Notice, that we used the identity $\tau\cdot e_z = n\cdot e_r$ to rewrite the speed term.

Taking into account equations \eqref{30}, \eqref{31}, \eqref{27}, \eqref{29}, and \eqref{26}, the elliptic problem for the stream function becomes an overdetermined boundary value problem. For a given domain $\Omega$ and fixed velocities $\bar W$ and flux constants $\bar\gamma$, the interior domain Dirichlet problem \eqref{27} and \eqref{29}, together with the exterior domain Dirichlet problem \eqref{27}, \eqref{30}, and \eqref{31}  would already uniquely determine $\psi$. Therefore, the requirement that the jump condition \eqref{26} is additionally satisfied imposes a condition on the shape of $\Omega$.

Our analysis shows that this problem is uniquely solvable for a class of domains $\Omega$ close to $B_{\bar \eps}(R,0)$. It also determines the unknowns $\bar W, \bar \gamma$, and $\bar \nu$.  

We now present a second still non-specific version of our main theorem, which leads to Theorem \ref{Thm1} by choosing $U=r^{-1} \grad^\perp\psi$.

\begin{theorem}[Second non-specific formulation]
    \label{Thm2}
    Let $\bar{b}_{\mathrm{int}}, \bar{b}_{\mathrm{ext}}, \rin,\rout$,  $R$, $\bar \eps$ and $\bar \sigma_{ \eps}$ be given as in Theorem \ref{Thm1}. Then there exists an $ \eps_0\in\R_{>0}$ such that for all $ \eps \in (0, \eps_0)$, there exists a unique solution $(\Omega, \bar W, \bar \gamma, \bar \nu,\psi)$ to the overdetermined boundary problem with $\Omega$ close to $B_{\bar \eps}(0,R)$. The speed $\bar W$ is approximately given by formula \eqref{407}. The leading order asymptotics of the flux constant and the Bernoulli constant are known.
 
\end{theorem}

Analyzing overdetermined boundary problems is a classical problem in mathematical physics, particularly in potential theory, fluid dynamics, and differential geometry. A well-known example is the Schiffer problem, where the eigenvalue problem for the Neumann Laplacian is studied for geometries that permit constant Dirichlet data (see page 688 of \cite{Yau82} or \cite{Shklover00}). Closely related is the overdetermined Dirichlet eigenvalue problem proposed by Sicbaldi \cite{Sicbaldi10}, or in a more general nonlinear form by Berestycki, Caffarelli, and Nirenberg \cite{BerestyckiCaffarelliNirenberg97}.
See also the work by Dom\'inguez-V\'azquez,  Enciso,  and Peralta-Salas \cite{DominguezVazquezEncisoPeraltaSalas23}.

In the context of (one-phase) vortex rings, overdetermined boundary problems similar to ours, but involving a regular stream function, have been studied since Fraenkel’s work \cite{fraenkel1970steady}. Alt and Caffarelli also explored free fluid problems using a variational method \cite{AltCaffarelli81}. While many related contributions exist, we cannot cite them all here; however, a good, though still non-comprehensive, overview is provided in \cite{DominguezVazquezEncisoPeraltaSalas23}. Our problem differs from those mentioned above in that it involves a jump condition for the normal derivative, rather than a (nonlinear) Neumann condition. To the best of our knowledge, such conditions have not been investigated in other settings.

We solve the overdetermined boundary problem in a perturbative setting. The key observation is that, in the limit of vanishing cross-sectional radii, $ \eps \to 0$, the vortex ring locally resembles a cylinder, and the cross-section $\Omega$ approximates the disk $B_{\bar \eps}(R,0)$. After an appropriate change of variables (i.e., using \emph{blow-up} variables), the limiting problem becomes thus a circular vortex patch in the plane, where the jump condition \eqref{26} is trivially satisfied.

In the next subsection, we will perform this change of variables and then formulate our precise result.

\subsection{The perturbative setting}

For a ring of small cross-sectional radius, we make the ansatz
\[
\Omega = \begin{pmatrix}
R\\0
\end{pmatrix} + \bar\eps\Omega_{\theta},
\]
where $ \Omega_{\theta} $ is the domain enclosed by the image $\chi_{\theta}(\partial B_1(0))$ of the unit circle under the map \begin{align}
\chi_{\theta}(y) := y(1+\theta(y)).
\end{align}
By construction, $\theta$ describes  the radial signed distance from the boundary $\partial \Omega_{\theta}$ to the unit circle $\partial B_1(0)$. It will be regular and small so that the cross-section is nearly spherical. See Figure \ref{fig} for an illustration. \begin{figure}
  \centering
  \includegraphics[scale=0.2]{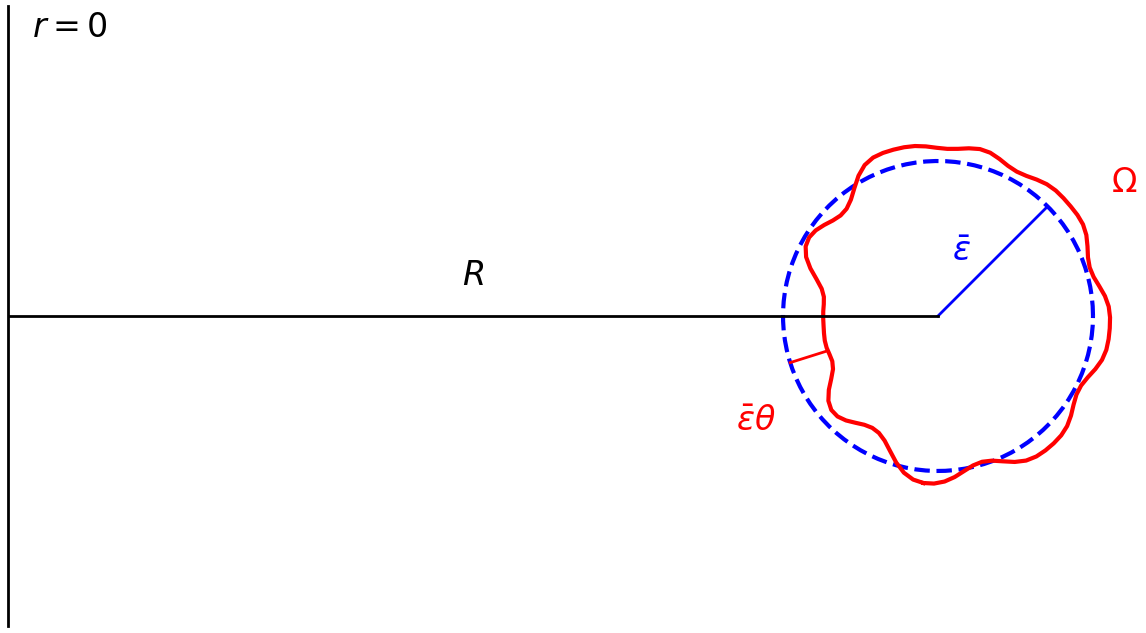}
  \caption{$\Omega$ in the original variables}\label{fig}
\end{figure}
 The non-dimensional parameter $\eps=\frac{\bar\eps}{R}$ measures the ratio of the interior ring radius $R$ to the cross-section radius. In order to simplify the notation slightly, we will often write $B$ for $B_1(0)$. Moreover, whenever it is convenient, we will parametrize the circle $\partial B$ by the torus  $\T=\R/2\pi\Z$.

We will work with the geometric conditions
\begin{equation}\label{62}
\theta(y_1,y_2) = \theta(y_1,-y_2),\quad |\Omega_{\theta}| =  \pi,\quad \int_{ \Omega_{\theta}} y_1\, \dy = 0,\quad \norm{\theta}_{L^\infty}\leq\frac{1}{2}.
\end{equation} 
The first prescribes symmetry across the radial axis (which is natural as the problem is symmetric in $z$), the second fixes the volume of the perturbed ball, and the third condition sets the center of the domain to the origin (with respect to the new variables). The second and third conditions are justified since $\bar\eps$ and $R$ are given.
 The last condition is there to make sure $\Omega_\theta$ is well-defined.

We  will choose our perturbation variables in the Banach manifold with boundary
\begin{equation}\label{def M}
\M  :=\left\{ (\eps,\theta)\in [0,\eps_0]\times H^k(\partial B):\:  \eqref{62}\mbox{ holds and }\|\theta\|_{H^k} \le \eps^{\ell}\right\}, 
\end{equation}
where $k\in \N$ is an arbitrary number with $k\geq 5$ and $\ell \in(\frac{1}{2},1)$ is arbitrary. The restriction $k\geq 5$ is most likely not optimal, but difficult to improve and mostly comes from Lemma \ref{L4}. In particular, it implies that $\de \Omega_\theta$ is (at least) $C^4$ because of the Sobolev embedding $H^5(\partial B)\hookrightarrow C^4(\partial B)$.
The purpose of the last condition in \eqref{def M} is to estimate broken terms such as e.g.\ $ \theta^2\eps^{-1}$ near $0$.

We will also consider the sets \begin{align}
V^k:=\{\theta\in H^k(\partial B):\: \exists \eps \text{ such that } (\eps,\theta)\in \mathcal{M}\}
\end{align}

\noindent of all $\theta$ such that \eqref{62} holds, which is a Banach manifold too (and in fact smooth around $0$, as we will see in Subsection \ref{tspace M}).
For notational convenience, we will furthermore write \begin{align}
H_{\sym}^k(\de B):=\{\theta\in H^k(\de B):\:\theta(y_1,y_2)=\theta(y_1,-y_2)\}.
\end{align}

 It is convenient to split the elliptic problem \eqref{30}, \eqref{31}, \eqref{32b}, \eqref{27}, \eqref{29} into an inhomogeneous interior contribution and the remaining homogeneous problem, which are independent of each other. More specifically, slightly abusing our notation convention, we consider $\psi_{\mathrm{int}} = \psi-\frac{\bar W}2r^2-\bar \gamma  $, which solves  the elliptic Dirichlet problem inside of the domain,
\begin{equation}
\label{39}
-\div\left(\frac1{r}\grad \psiin \right) =\bar{\xi} r \quad\mbox{in }\Omega,\quad 
\psiin =0\quad \mbox{on }\partial\Omega.
\end{equation}   
The exterior part $\psiout=\psi-\psiin\1_{\Omega}$, defined on all of $\HH$, then fulfills
\begin{equation}
\label{40}
-\div\left(\frac1{r}\grad \psiout  \right) =0 \quad\mbox{in } \HH\backslash \de\Omega,\quad 
\psiout    =\frac{\bar W}2r^2+\bar\gamma\quad \mbox{on }\partial\Omega,
\end{equation}
together with the boundary conditions
\begin{equation}
    \label{40a}
    \psiout=0\quad\mbox{on }\partial \HH,\quad \frac1r|\grad\psiout|\to 0\quad \mbox{as }|(r,z)|\to \infty, \quad\lim_{r\rightarrow 0} \frac{1}{r}|\nabla\psiout|<\infty.
\end{equation}
In these new variables, the circulation condition \eqref{402} reads
\begin{align}
\bar{b}_{\mathrm{ext}}= - \int_{\de\Omega}\frac{1}{r}\de_n\psiout\dx.\label{circ cond}
\end{align}

It is readily checked that    $\psiout=\frac{\bar{W}}{2}r^2+\bar \gamma$ also inside of $\Omega$.
Using this decomposition, the jump condition \eqref{26} simply reads as  \begin{align}
\rin\left(\frac{1}{r}\de_n\psiin\right)^2-\rout\left[\frac{1}{r}\de_n\psiout\right]^2+2\bar\sigma H=\bar \nu,
\end{align}

\noindent where the first normal derivative is the \emph{inner} normal derivative.

We note that $\psiin$ is independent of $\bar{b}_{\mathrm{ext}}$ and 1-homogenous in $\bar \xi$, while $\psiout$ is independent of $\bar \xi$ and 1-homogenous in the triple $(\bar{b}_{\mathrm{ext}}, \bar W, \bar \gamma)$.

Because we aim at constructing vortex rings of small cross-sectional radius, $\eps\ll1$, it is customary to rescale the elliptic problem \eqref{39}, \eqref{40}, \eqref{40a}, \eqref{circ cond} using blow-up variables. For this purpose, we consider the mapping $q$ defined by \begin{align}
q(r,z)  = \frac1{\bar\eps}(r -R,z) = (x_1,x_2)=x,
\end{align}
 which maps $\Omega$ onto some $\Omega_{\theta}$. We observe that $r = R(1+\eps x_1)>0$ enforces that $x_1 >-1/\eps$, and thus \begin{align}
 \HH_{1/\eps} =  q(\HH) = \{x\in \R^2:\: x_1 +1/\eps>0\}.
 \end{align} 
Moreover, we will consider \begin{align}
\bin = \pi R^2\bar \eps^2 \bar \xi,
\end{align}
as motivated by  \eqref{411} and the scaling of the circulation, which implies that $\bar{b}_{\mathrm{int}} R\approx \bin$ up to lower order terms. 
Finally, we take the opportunity to non-dimensionalize the problem. Setting
\begin{gather*}    
\bout=R\bar{b}_{\mathrm{ext}},\quad\beta = \frac{2 R^3\eps \bar \sigma}{\rout \bout^2} ,\quad 
W=\frac{R^2}{\bout} \bar W,\quad \gamma =\frac1\bout \bar \gamma ,\quad \nu = \frac{\eps^2 }{\rout \bout^2}\bar \nu,\label{substitutions2}
\end{gather*}
and writing $\vin  = 4\pi \bin^{-1}\psiin \circ q^{-1}$ and $\vout=\bout^{-1}\psiout\circ q^{-1}$, we are now considering the  overdetermined free boundary problem
   \begin{align}
-\div\left(\frac1{1+\eps x_1}\grad \vin \right) &=4 (1+\eps x_1) \quad\mbox{in }\Omega_{\theta},\label{32}\\
\vin&=0\quad\mbox{on }\de\Omega_\theta,\label{32'}\\
\vin&=0\quad\mbox{in }\HH_{1/{\eps}}\backslash\Omega_{\theta},\label{32''}\\
-\div\left(\frac1{1+\eps x_1}\grad \vout\right) &=0 \quad\mbox{in } \HH_{1/\eps}\setminus \de \Omega_{\theta},\label{33}\\
\vout  - \frac{W}2(1+\eps x_1)^2 & =\gamma \quad\mbox{on }\partial \Omega_{\theta},\label{34}\\
\vout& =0 \quad \mbox{on }\partial \HH_{1/{\eps}},\label{35}\\
\frac1{1+\eps x_1}|\grad\vout|&\to 0\quad\mbox{as }|x|\to \infty,\label{36}\\
\lim_{x_1\rightarrow-\frac{1}{\eps}}\frac1{1+\eps x_1}|\grad\vout|&<\infty,\label{36'}\\
\hspace{-2em} \rho \left( \frac{1}{1+\eps x_1}\partial_n \vin\right)^2 - \left[ \left( \frac{1}{1+\eps x_1}\partial_n \vout\right)\right]^2& + \beta \left(\kappa_{\theta} +\eps \frac{n\cdot e_1}{1+\eps x_1}\right)=\nu \quad   \mbox{on }\partial \Omega_{\theta},\:\:\label{38}
\end{align}
where $\kappa_{\theta}$ is the curvature of the one-dimensional boundary $\partial \Omega_{\theta}$ (see \eqref{413} for the derivation of that term). The factor of $4$ in \eqref{32} is simply introduced for notational convenience in the later part of the paper.   Furthermore, we have that
\begin{equation}
\label{52}
\int_{\partial \Omega_{\theta}} \frac1{1+\eps x_1}  \partial_n \vout\dd \Ha^{1} = -1,
\end{equation}
cf.~\eqref{circ cond}. Our rescaling of the interior stream function is valid only if the interior vorticity distribution is non-zero, $\bin\not=0$. If, however, $\bin=0$, we may simply choose $\vin=\psiin=0$.

The resulting problem thus depends only on two non-dimensionalized physical control quantities: $\rho$, which is the circulation-based inertia ratio, and $\beta$, which is (twice) the inverse of the  circulation based Weber number \eqref{501}. The speed of the ring $W$, the flux constant $\gamma$, and the Bernoulli constant $\nu$ will all depend on these quantities and on the ring geometry. Furthermore, the assumptions on the surface tension, $\beta$ is a $C^1$ function of $\eps$ and the conditions \eqref{404}, \eqref{405}, and \eqref{406} become
\begin{gather}
 \lim_{\eps\to 0} \frac1{\beta_{\eps} }\in [0,\infty)\setminus \left(8\rho + \frac1{2\pi^2}\right)^{-1}\N_{\ge 3},\label{322}\\
  \lim_{\eps\to0} {\eps}{\beta_{\eps}} = 0,\label{323}\\
 \eps |\partial_{\eps}\beta_{\eps}|\lesssim \beta_{\eps}.\label{324}
\end{gather}

%We remark that if none of the parameters is $0$, this system can be further non-dimensionalized and up to scaling, the solution only depends on the numbers $\eps$, $\rho$ and $\eps\sigma$.

%We want to allow for the case $a^2\rho^-=0$ (which corresponds to the hollow vortex ring) here though, so we will not make this reduction.

Our main result  can now be stated in the following precise form:

\begin{theorem}
\label{T2}
Let $\rho\in\R_{\ge0}$ denote the inertia ratio. Suppose that the inverse Weber number $\beta\in\R_{\ge0}$ is a $C^1$ function of the aspect ratio $\eps\in\R_{>0}$, that is either identically zero, $\beta=0$, or it is strictly positive, $ \beta_{\eps}>0$, with \eqref{322}, \eqref{323}, and \eqref{324}. Then there exists an $\eps_0\in\R_{>0}$ such that for any $\eps\in(0,\eps_0)$, there exists a solution $(\theta,\vin,\vout,W,\gamma,\nu)$ to \eqref{32}, \eqref{32'}, \eqref{32''}, \eqref{33}, \eqref{34}, \eqref{35}, \eqref{36}, \eqref{36'}, \eqref{38}, and \eqref{52}. The function $\theta$ is smooth and  $\vin$ and $\vout$  are both smooth up to the boundary $\de\Omega_\theta$. The shape function $\theta$ satisfies the geometric conditions in  \eqref{62}. 

Furthermore, the following holds:
\begin{enumerate}
\item The speed, flux and Bernoulli constants satisfy the following asymptotic expansions
\begin{align}\label{325}
  W&=\frac{1}{4\pi}\left(\log 8-\frac{1}{2}+\log\frac{1}{\eps}\right)+\rho\pi+\frac{\beta_{\eps}\pi}{2} +o(1),\\
\label{326}
 \gamma&=\frac{3}{8\pi}\log\left(\frac{8}{\eps}\right)-\frac{15}{16\pi}-\frac{\rho\pi}2-\frac{ \beta_{\eps}\pi}{4}+o(1),\\
 \nu & =    4\rho - \frac1{4\pi^2} +\beta_{\eps}+o(1),
\end{align}
as $\eps\to0$.  

\item The shape function $\theta$ goes to $0$ in every $H^k(\partial B)$-space as $\eps\rightarrow 0$. In fact, $\theta$ is a continuously Fr\'echet differentiable function of $\eps$ and it holds that $\de_\eps\theta\rightarrow 0$ at $\eps=0$ in every $H^k(\partial B)$ space.

\item There exists a constant $C_*$ such that for  small enough $\eps$, the solution  is unique  among all $(\theta,\vin,\vout,W,\gamma,\nu)$ such that $\theta$ is in the class \eqref{62} with $\norm{\theta}_{H^5}\leq\eps^{\ell}$ for some $\ell\in (1/2,1)$, and $W$ satisfying the bound 
\begin{equation}
    \label{321}
    |W||\log \eps|\left(\eps^2 +\|\theta\|_{H^k}^2\right)\le C_*.
\end{equation}
\end{enumerate}
\end{theorem}

Reversing the above rescaling, we recover Theorem \ref{Thm2}.

We will find our solution to this overdetermined boundary problem by applying the implicit function theorem to the jump condition \eqref{38}. Apparently, in the limiting case $(\theta,\eps)=(0,0)$, the cross-section domain $\Omega$ becomes the unit disk and the left-hand side of \eqref{38} reads 
\[
\lim_{\eps\to 0}\frac1{\beta_{\eps}} \left(\rho\left(\partial_n \vin)^2 -(\partial_n\vout)^2\right)\right) +1,
\]
which is indeed constant because  the limit potentials $\vin$ and $\vout$ defined now on $B$ and $\R^2$, respectively, are radially symmetric functions. We will give more insights into the strategy of the proof in the subsequent subsection. At this point, however, we already remark that in the case of positive surface tensions, the  linear operator associated with  the jump condition \eqref{38} has the Fourier symbol
\begin{align}
\lim_{\eps\to 0}\frac1{\beta_{\eps}}\left(8\rho+\frac1{2\pi^2}\right)(1-|k|) - 1+k^2,\label{Fourier sym}
\end{align}
which may actually degenerate unless assumption \eqref{322} is satisfied. The limiting problem for a hollow vortex ring was previously studied by Wegmann and Crowdy \cite{MR1794849}, who discovered bifurcations from a circular vortex sheet. These bifurcations bear some resemblance to the well-known Crapper waves, which are non-trivial, steady irrotational capillary water waves. After fixing the mass density, the bifurcations occur at precisely the same critical values that we need to exclude in \eqref{322}. It would be interesting to investigate such bifurcations in the context of our thin vortex rings as well. We plan to address this problem in future work.

\subsection{Strategy of the proof}\label{proof strat}

Our goal is to apply the implicit function theorem to the jump condition \eqref{38}. We thus have to establish that the geometric quantities and the solution given through \eqref{32}--\eqref{52} depend continuously differentiable on the pair $(\eps,\theta)$.  
With this, after pulling back  the jump  condition \eqref{38} to the unit circle, we have that
\begin{equation}
\label{41}
\rho \lambda_{\theta,\eps}^2 -  \mu_{\theta,\eps}^2 + \beta_{\eps} h_{\theta,\eps} =\const \quad\mbox{on }\partial B,
\end{equation}
where now
\begin{align*}
h_{\theta,\eps} & = \kappa_{\theta}\circ \chi_{\theta} +\eps \frac{(n_{\theta}\cdot e_1)\circ \chi_{\theta}}{1+\eps \chi_{\theta}\cdot e_1},\\
\lambda_{\theta,\eps} & = \frac1{1+\eps \chi_\theta\cdot e_1}\partial_n \vin\circ \chi_{\theta},\\
\mu_{\theta,\eps} & =\eps W (n_{\theta}\cdot e_1)\circ \chi_{\theta}- \frac1{1+\eps\chi_{\theta}\cdot e_1} \partial_{n}\vout \circ\chi_{\theta},
\end{align*}
where $\kappa_\theta$ is the mean curvature of $\Omega_\theta$ and $n_\theta$ its normal. The circulation condition \eqref{52} becomes
\begin{equation}
    \label{320}
    \int_{\partial B} m_{\theta} \mu_{\theta,\eps}\dd \Ha^1 = 1.
\end{equation}

Our goal is to find a unique shape function $\theta=\theta(\eps)$ such that the jump condition \eqref{41} is additionally satisfied. To this end, we use  the abstract notation ``$\const$'' to represent the Bernoulli constant $\nu$, indicating that the jump condition will be solved in a quotient space.
As previously stated, we will apply the implicit function theorem to solve this equation. In the base case $(\theta,\eps)=(0,0)$, the equation is trivially satisfied because, on the ball $\Omega_0=B_1(0)$, the curvature $h_{0,0}$ is constant. The solutions $\vin$ and $\vout$ to the interior and exterior problems, respectively, are radially symmetric, which ensures that their normal (and thus radial) derivatives are constant on the interface $\partial\Omega_0$. Furthermore, the coefficient $\eps W$ vanishes in this case.

Conceptually, the limiting problem for $\eps=0$ should correspond to a similar reformulation for a stationary 2D vortex.

We can not expect the implicit function theorem alone to prove the statement, as the limiting problem is clearly invariant under translation, and hence the derivative of \eqref{41} will degenerate in at least one direction.
To deal with this degeneracy we will introduce an additional parameter $S$, which, roughly speaking, corresponds to the additional degree of freedom that the speed $W$ gives us and which is not present in the limiting problem of a 2D vortex. This extra parameter will then allow us to perform a Lyapunov--Schmidt reduction. We will see that in terms of Fourier series the subspace that is missing in the image of the derivative is the span of $\cos\alpha=x_1$ (up to lower order terms).

 To solve the equation for positive $\eps$, we will show that all quantities $h_{\theta,\eps}$, $\lambda_{\theta,\eps}$ and $\mu_{\theta,\eps}$ are in a suitable sense continuously differentiable in $(\theta,\eps)$. This is comparably simple for the curvature $h_{\theta,\eps}$ and the interior potential $\lambda_{\theta,\eps}$. We treat these terms in Sections \ref{S:shape} and \ref{S:inner} below.
For the exterior problem, the analysis is more difficult for two reasons. The first one is that the (formal) limiting problem \begin{align}
\Delta \phi&=0\quad \text{in $\R^2\backslash \Omega_\theta$},\label{2D 1}\\
\phi&=\const \quad \text{ on $\de \Omega_\theta$},\label{2D 2}
\end{align}
does not have a good solution theory in $H^1$ anymore. Indeed solutions to this are of the form $p\log|x|+q\notin H^1$ if $\Omega_\theta=B$ and only unique under additional assumptions (such as fixing the prefactor of the logarithm), while for $\eps>0$, the solution for $\psiout$ does not require a selection principle.
Instead of directly working with the PDE for the limiting problem \eqref{2D 1}, \eqref{2D 2}, we will use a more indirect approach and reformulate this problem as \begin{align}
\int_{\de\Omega_\theta} \frac{1}{2\pi}\log|x-y| \de_n\phi(x)\dx&=\const \text{ on $\de\Omega_\theta$},\label{const cond}\\
 \int_{\de\Omega_\theta}  \de_n\phi(x)\dx&=-1,\label{circ cond 3}
\end{align}

\noindent which does indeed yield a unique solution $\phi=\frac{1}{2\pi}\log|x-y|*(\de_n\phi\mathcal{H}^1\mres\de\Omega_\theta)$ by classical potential theory, see e.g.\ \cite{armitage2012classical} for further reading.
This formulation is also available for $\eps>0$. If one lets $K(x,y)$ denote the fundamental solution of $-\div(\frac{1}{r}\nabla\cdot)$, then we can write \begin{align}
\int_{\Omega} K(x,y)\left[\frac{1}{r}\de_n\psiout\right]\dx=\frac{1}{2}\bar Wr^2+\bar\gamma.\label{sing layer}
\end{align}

\noindent See Section \ref{reform} for a detailed derivation.
There is also a natural analogue for \eqref{circ cond 3}, which is the circulation condition \eqref{circ cond}, that is, 
\begin{align}
\int_{\de\Omega} \frac{1}{r}\de_n \psiout\dx=-1.\label{circ cond 2}
\end{align}

\noindent The basic strategy is then to study the convergence of \eqref{const cond}, \eqref{circ cond 3} to \eqref{sing layer}, \eqref{circ cond 2} (pulled back to $\de B$).

The constants $W$ and $\gamma$ are still undetermined at this point. One of their degrees of freedom should be spent by fixing the circulation, the other by choosing the reduction parameter $S$.

The second difficulty is that $W$ and $\gamma$ apparently blow up logarithmically as $\eps\rightarrow 0$, as their asymptotics in  Theorem \ref{T2} show. As the main term in $K$ is a logarithm $\log|x-y|$ (see \eqref{43} and \eqref{SeriesK}), it will also naturally contain $|\log\eps|$-terms after pulling back too. Therefore one needs to be very careful to not pick up any terms with an asymptotic $\eps|\log\eps|$ or similar, which would not be continuously differentiable. In particular, for a general $W$ and $\gamma$, the function $\mu_{\theta,\eps}$ will in general not be differentiable in $\eps$.
The approach for extracting a continuously differentiable $\mu_{\theta,\eps}$, with a well-behaved dependence on the reduction parameter $S$, is to compute the asymptotic of the pullback of $K$, and then adjust $W$ and $\gamma$ to match the desired asymptotic of $\mu_{\theta,\eps}$.
The desired asymptotic of $\mu$ is \begin{align}
\mu=\tilde{\mu}_{\theta,0}+\eps S x_1+\text{ second order terms},\label{des asym}
\end{align}
where $\tilde{\mu}_{\theta,0}$ is the solution of the limiting problem given through $-\de_n \phi\circ \chi_\theta$ in \eqref{const cond}, \eqref{circ cond 3}.
The crucial step is that both $\gamma +\frac{1}{2}W(1+\eps(\chi_\theta)_1)^2$ and the image of this ansatz under $K$ lie in the span of $1$ and $y_1$ up to terms of order $\eps^2|\log\eps|$. This then allows to construct $\mu,W,\gamma$ so that \eqref{des asym} holds with higher order terms of vanishing derivative at $0$.

In Section \ref{S:IFT}, we finally solve the identity \eqref{41} for small values of $\eps$.
We will see that the derivative in $\theta$ at $0$ is the Fourier multiplier \begin{align}
\left(8\rho+\frac{1}{2\pi^2}\right)(1-|l|) +\beta_\eps(1-l^2) ,
\end{align}

\noindent which is invertible for $l\neq 1$ (the mode $l=1$ corresponds to translations and is handled by the reduction) under the  condition \eqref{322}.

\section{Preliminaries and notation}

\subsection{Notation}

Our paper is rather technical. To facilitate reading, we provide an overview of the notation:
\begin{itemize}
    \item We typically write $x$ for the space variable on $\Omega_{\theta}$ or $\partial \Omega_\theta$ and $y$ for the variable on $\de B$ or $B$.
    \item $\theta(\alpha)$ denotes the (signed) distance between $\Omega_\theta$ and $\de B$ at angle $\alpha$.
    \item $\Theta$ denotes a (fixed) extension of $\theta$ to $B_1(0)$.
    \item $\chi_\theta(y) = (1+\theta(y))y$ is the canonical map from $\de B_1(0)$ to  $\de \Omega_{\theta}$. Occasionally we will write only $\chi$.
    \item $\Chi_\theta(y) = (1+\Theta(y))y$ is an extension of $\chi_\theta$ which maps $B_1(0)$ to $\Omega_\theta$.
    \item The Jacobian of the map $\chi_\theta$ between the boundaries equals \begin{align}\label{20}
m_\theta:=\sqrt{(\de_\alpha\theta)^2+(\theta+1)^2}.
\end{align}
\item Let $X(\alpha) = (\cos \alpha , \sin \alpha)^T$ be the radial vector in polar coordinates as a function of the angle $\alpha\in \T\simeq \de B$.
Then
the pullback of the tangent $\tau_\theta$ of $\de\Omega_\theta$ is given by 
\[
\tau_{\theta}\circ \eta_\theta= \frac1{m_{\theta}} \left( (1+\theta) X^{\perp} +\partial_{\alpha}\theta X\right),
\]
and the pullback of the tangent $n_\theta$ of $\de\Omega_\theta$ is given by\begin{equation}\label{ExpNormal}
n_{\theta}\circ \chi_\theta=\frac1m_{\theta} \left((1+\theta) X - \partial_\alpha \theta X^{\perp}\right).
\end{equation}
We will occasionally just write $n$ instead of $n_\theta$.
\end{itemize}

\subsection{Some Preliminaries on fractional spaces and Fr\'echet differentiability in them} 
We would like to collect some general facts about $H^s$ spaces and Fr\'echet differentiability in them that we will frequently use. 

On the boundary $\de B$, we can write $L^2$-functions as Fourier series, since $\de B\simeq \T$. We then set \begin{align}
\norm{f}_{H^s}:= \norm{f}_{H^s(\de B)}:=\left(\sum_{m\in \Z} (1+|m|)^{2s}|\widehat{f}(m)|^2\right)^\frac{1}{2}.
\end{align}

\noindent This is the same as the fractional Sobolev space $W^{2,s}$ defined via the Gagliardo seminorm. For $s\notin \Z+\frac{1}{2}$ and $s> \frac{1}{2}$ there is a Sobolev embedding $H^s(\de B)\hookrightarrow C^{s-\frac{1}{2}}(\de B)$. 

There is a continuous linear trace extension operator from $H^s(\de B)$ to $H^{s+\frac{1}{2}}(B)$ for $s>0$, that is, for every $\theta \in H^s(\de B)$ there is a $\Theta\in H^{s+\frac{1}2}(B)$ such that \begin{align}
\norm{\Theta}_{H^{s+\frac{1}{2}}(B)}\lesssim \norm{\theta}_{H^s(\de B)}\quad\text{and}\quad \Theta|_{\de B}=\theta.
\end{align}
Here the norm on $H^{l}(B)$ is defined as usual via \begin{align}
\norm{f}_{H^l(B)}:=\norm{f}_{L^2(B)}+\norm{\frac{\nabla^{\lfloor l\rfloor}f(\tilde{y})-\nabla^{\lfloor l\rfloor}f(y)}{|y-\tilde{y}|^{1+l-\lfloor l\rfloor}}}_{L^2(B\times B)}.
\end{align}
There is a Sobolev embedding $H^s(B)\hookrightarrow C^{s-1}(B)$ for $s\notin \Z$. There is a bounded linear trace $H^s(B)\rightarrow H^{s-\frac{1}{2}}(\de B)$ for $s>\frac{1}{2}$.
We refer the reader e.g.\ to the monograph \cite{Hitchhiker} for further reading.

One important property is that if $u\in H^s(\Omega', \R^l)$ for some bounded $\Omega'\subset \R^d$ with sufficiently smooth boundary and $s>\frac{d}{2}$, then for every $F\in C^{\lceil s\rceil}$, the composition $F(u)$ is in $H^s$ as well and composition with $F$ is continuous in $u$, for a proof see e.g.\ \cite{BrezisMironescu}. 
In particular, this implies that these spaces are closed under products.

One important application is that: \begin{lemma}\label{FreLemma}\begin{itemize}
\item[a)] Let $U$ and $X\hookrightarrow Y$ be Banach spaces and let $V\subset U$ be open.
Suppose $G:V\rightarrow X$ is twice Gateux-differentiable as a map to $Y$. If its second derivative is (locally) bounded as a map from $U^2$ to $X$, then $G$ is (locally) continuously Fr\'echet differentiable.
\item[b)] Let $\Omega'\subset \R^d$ be a smooth and bounded domain. Let $j_1,j_2,j_3,l\in\N_0$ be given with $j_1=j_2(1+d+\dots+d^l)$ and let $F:\R^{j_3+j_1}\rightarrow \R^m$ be a smooth vector field.  If  $s>l+\frac{d}{2}$, then the map \begin{align}
(t,u)\rightarrow F(t,u, \nabla u, \dots ,\nabla^l u)\label{F map}
\end{align}
 is Fr\'echet smooth from $\R^{j_3}\times H^s(\Omega';\R^{j_2})$ to $H^{s-l}(\Omega';\R^m)$ and the Fr\'echet derivatives agree with the pointwise derivatives.
\end{itemize}
\end{lemma}

\begin{proof}
\textbf{a)} We choose $u \in V$ and pick $\delta u$ small enough so $u +\delta u$ belongs to an open ball in   $V$ centered at $u$. We want to show that the Gateux-derivative  $\dd G(u). \delta u$ is the Fr\'echet derivative.  Making use of the mean value theorem, we see that
\begin{align}
\mel \norm{G(u+\delta u)-G(u)-\dd G(u).\delta u}_{X}\\
& = \norm{\int_0^1 (\dd G(u+s\delta u)-\dd G(u)).\delta u\ds}_{X}\\
& = \norm{\int_0^1\int_0^1  \dd^2 G(u+st\delta u).(\delta u,\delta u)\dt \ds}_{X} \\
&\leq \norm{\delta u}_U^2\sup_{v\in V}\norm{\dd^2G(v)}_{U^2\rightarrow X},
\end{align}
which proves the claimed differentiability. Clearly, the derivative must be locally Lipschitz continuous if the second derivative is locally bounded. 

 \medskip

\textbf{b)} It is not restrictive to assume that $j_3=0$ since we can replace $u$ by $(u,t)\in\linebreak H^s(\Omega';\R^{j_2+j_3})$.

We first apply a) with $X=H^{s-l}(\Omega';\R^m)$ and $Y=C(\bar \Omega';\R^m)$: We observe that the map \eqref{F map} is Gateaux differentiable as a map to $C(\bar \Omega';\R^m)$ and the pointwise derivative and the Gateaux derivative agree. Indeed, thanks to the smoothness of $F$, it holds that
\begin{align}
\mel \norm{F(u+\delta u,\nabla(u+\delta u),\dots)-F(u,\dots)-\sum_{i=0}^l\de_iF(u,\dots). \nabla^i \delta u}_{C^0}\\
&\lesssim \norm{\delta u}_{C^l}\sup_{|x|\leq \norm{u}_{C^l}+\norm{\delta u}_{C^l}}\|\nabla^2 F(x)\|_{C^0}\\
&\lesssim_{F,u} \norm{\delta u}_{H^{s}},
\end{align}
where we used the Sobolev embedding $H^s\hookrightarrow C^l$ in the last step.

Furthermore, the second Gateaux derivative of $F$ is of the form \begin{align}\label{smoth func}
 \tilde{F}(u,\dots,\nabla^l u,\delta u_1,\dots, \nabla^l\delta u_1, \delta u_2,\dots, \nabla^l\delta u_2)\end{align}for some smooth $\tilde{F}$. This map is continuous as a map from $H^{s}(\Omega';\R^{3m})$ to $H^{s-l}(\Omega';\R^m)$ and in particular locally bounded and hence a) implies that the map is continuously Fr\'echet differentiable. 

To conclude smoothness, observe that every derivative of the map has the structure \eqref{smoth func} (with a different number of $\delta u$'s and a different $\tilde{F}$) and hence is continuously Fr\'echet differentiable too, which then implies smoothness.
\end{proof}

We remark that we are not going to make use of function spaces on any domains different from $B$ and $\de B$ in the sequel.

\section{Geometric lemmata}\label{S:shape}

We show some basic properties about the differentiability of various geometric properties. The two key results are the characterization of the tangent space of $\mathcal{M}$ near $0$ (Proposition \ref{tspace M}) and the proof of the differentiability of the (rescaled) curvature $h_{\theta,\eps}$ (Proposition \ref{DeriCurvature}). We start with the latter.

\subsection{Regularity of geometric quantities}

\begin{proposition}
\label{DeriCurvature}
The function $(\theta,\eps)\mapsto h_{\theta,\eps}\in H^{k-2}(\de B)$ is well-defined on a small open neighborhood of $(\theta,\eps)=(0,0)\in   H^{k}(\de B)\times \R$ and  Fr\'echet smooth in the joint variable. It holds that
\begin{align}
\scalar{\mathrm{D}_\theta\big|_{(\theta,\eps)=(0,0)} h_{\theta,\eps}}{\delta \theta}&=-\delta\theta-\de_\tau^2\delta\theta,\\
\de_\eps\big|_{(\theta,\eps)=(0,0)} h_{\theta,\eps}&=y_1.
\end{align}
\end{proposition}
Here and occasionally in the following, we allow for negative values of the cross-section parameter $\eps$ for mathematical convenience. 
\begin{proof}
The computation of the derivative is a straightforward exercise that can be carried out on an elementary level by making  use of the explicit formula 
\begin{align}\label{expr h}
 h_{\theta,\eps} &  = \frac{(1+\theta)^2 +2 (\de_\tau\theta)^2 -(1+\theta)\de_\tau^2\theta}{\left((1+\theta)^2 +(\de_\tau\theta)^2\right)^{\frac32}} + \eps \frac{(1+\theta)\cos \alpha +\de_\tau\theta\sin\alpha}{\left((1+\theta)^2 +(\de_\tau\theta)^2\right)^{\frac12}\left(1+\eps(1+\theta)\cos\alpha\right)}.
\end{align}
The justification of the smoothness in $H^{k-2}$ is a direct application of Lemma \ref{FreLemma}, except that the formula \eqref{expr h} is only locally smooth around $(0,0)$ in $\theta,\eps$, this is not an issue because we only care about small $\theta$ and can replace the definition with something smooth for large $\theta$.
\end{proof}

With the previous proposition, we already have all the information at hand we need to deal with the surface tension term in \eqref{41}. In the next lemma, we establish the regularity of two more geometric quantities that will be needed to analyze the pressure contributions in \eqref{41}.

\begin{lemma}\label{DeriNormal}
The functions $\theta\mapsto n_\theta\circ\chi_\theta\in H^{k-1}(\de B)$ and  $\theta\mapsto m_\theta\in H^{k-1}(\de B)$ are both Fr\'echet smooth near $\theta=0\in H^{k}(\de B)$, and it holds that
\begin{align*}
\la \mathrm{D}_{\theta}|_{\theta=0} m_{\theta},\delta \theta\ra  &= \delta \theta,\\
\la \mathrm{D}_{\theta}|_{\theta=0} (n_{\theta}\circ \chi_{\theta}),\delta \theta\ra & = -\partial_{\tau}\delta \theta\,  \tau_0,
\end{align*}
where $\tau_0=y^{\perp}$ is the tangent on $\de B$.
\end{lemma}
\begin{proof}
This is again a direct application of Lemma \ref{FreLemma}. The formulas for the derivative follow from the definitions in \eqref{20} and \eqref{ExpNormal}, respectively.
\end{proof}

 \subsection{The tangent space of $\mathcal{M}$}

\begin{proposition}\label{tspace M}
The manifold $V^k$ is smooth in a neighborhood of $\theta=0\in H^k(\partial B)$. In particular $\mathcal{M}$ is smooth around $(\theta,\eps)=(0,0)\in H^k(\partial B)\times \R$ with a  $C^{1,\frac{1}{\ell}-1}$-boundary.
The tangent space of $V^k$ at $\theta=0$ is given through \begin{align}
T_0V^k = \{\delta\theta\in H_{\sym}^k(\partial B):\:\scalar{\delta\theta}{1}=\scalar{\delta\theta}{x_1}=0\}.
\end{align}
\end{proposition}

In the statement of the proposition, the brackets $\la\cdot,\cdot \ra$ denote the $L^2(\partial B)$ inner product.

\begin{proof}
We first want to check that the geometric constraints in \eqref{62} can be written as the zero set of smooth functions $G_1,G_2$ of $\theta$, and to compute the derivative of these.

We first consider arbitrary $\theta\in H_{\sym}^k(\partial B)$, not necessarily in $V^k$.
We use an extension $\Theta$ of $\theta$ such that \begin{align}
\norm{\Theta}_{H^{k+\frac{1}{2}}(B)}\lesssim \norm{\theta}_{H^{k}(\de B)}\quad\text{and}\quad\Theta|_{\de B}=\theta.
\end{align} 
Then 
$\Chi_{\theta}(y) = (1+\Theta(y))y $
maps $B$ to $\Omega_\theta$. For sufficiently small $\theta$, this is a diffeomorphism, as it is a $C^1$-perturbation of the identity, since $H^{k+\frac{1}{2}}\hookrightarrow C^1$.

We can now write  the volume of the transformed domain as the function
\begin{align}
G_1(\theta)= |\Omega_\theta|=\int_{B}\det \mathrm{D}\Chi_\theta\dy.
\end{align}
This is the integral of a polynomial in $\Theta,\mathrm{D}\Theta$ and hence also smooth in $\Theta\in H^{k+\frac{1}{2}}$. As the linear map $\theta\rightarrow \Theta$ is smooth too, we see that this functional is smooth in $\theta\in H^k(\de B)$.
Using the fact that $\frac{\mathrm{d}}{\dt}\det (I+tA)\big|_{t=0}=\tr A$, we see that \begin{align}
\left.\frac{\mathrm{d}}{\dt}\right|_{t=0}G_1(t\theta)  =\int_{B} 2\Theta+\scalar{\mathrm{D}\Theta}{y}\dy=\int_B\div(x\Theta)\dy =\int_{\de B}\theta\dy, 
\end{align}
thanks to the divergence theorem. 

Similarly, we have for the $y_1$-moment
\begin{align}
G_2(\theta)=\int_{\Omega_\theta}y_1\dy=\int_B \Chi_\theta\cdot e_1 \det \mathrm{D}\Chi_\theta\dy,
\end{align}
which is smooth in $\theta$ for the same reason and it holds
\begin{align}
\left.\frac{\mathrm{d}}{\dt}\right|_{t=0} G_2(t\theta) =  \int_B \Theta(x)y_1+ y_1\div(y\Theta)\dy = \int_{\de B} \theta y_1\dy,
\end{align}
where the final step follows by an integration by parts.

Since the derivative of $(G_1,G_2)$ at $0$ is a diffeomorphism from $\spann\{1,y_1\}$ to $\R^2$, we use the inverse function theorem to see that $V^k$ is locally diffeomorphic to $\spann\{1,y_1\}^\perp\subset H_{\sym}^k(\partial B)$ and the diffeomorphism is smooth. This shows that $V^k$ is a smooth manifold locally around $\theta=0$.

Finally, due to the fibre construction, it is clear that $\M$ is a smooth manifold around $(\theta,\eps)=(0,0)\subset H^k(\partial B)\times \R$, with boundary regularity $C^{1,\frac1\ell-1}$. 
%
%The fact that this is a smooth manifold then follows from the fact that for every small $v\in \spann\{1,x_1\}$ there is a unique $\tilde{v}$ in the orthogonal complement $\spann\{1,x_1\}^\perp \subset H_{\sym}^k$ such that $G_1(v+\tilde{v})=G_2(v+\tilde{v})=0$ by the implicit function theorem (since the derivative of $G_1,G_2$ is invertible from $\spann\{1,x_1\}$ to $\R^2$), which gives a smooth diffeomorphism from $\spann\{1,x_1\}^\perp$ to $V^k$.
\end{proof}

As a trivial consequence, because  $V^k$ is a smooth manifold the geometric identities characterizing the tangent plane also hold at $\theta\neq 0$ up to higher order terms. For instance, the volume condition implies the following.

\begin{corollary}
For $\theta\in V^k$ and $\delta\theta\in T_\theta V^k$ it holds that
\begin{align}
\left|\int_{\de B} \theta\dy\right|&\lesssim \norm{\theta}_{H^k}^2\label{second order int},\\
\left|\la \mathrm{D}_\theta\int_{\de B} \theta\dy,\delta\theta\ra \right|&\lesssim\norm{\theta}_{H^k}\norm{\delta\theta}_{H^k}.\label{bd int}
\end{align}
\end{corollary}

Here we are using the natural identification of the tangent space with a subspace of $H^k$.

%\subsection{Some useful elementary estimates}
%\begin{lemma} We have the following estimates %\begin{align}
%\norm{x_1-(n_\theta\cdot e_1)\circ \chi_\theta}_{H^{k-1}}\lesssim \norm{\theta}_{H^k}\label{est n}
%\end{align}
%\end{lemma}

\section{The interior contribution} \label{S:inner}

In this section, we investigate the regularity properties of the interior potential 
\[
\lambda_{\theta,\eps} = \frac1{1+\eps \chi_{\theta}\cdot e_1} \partial_n \vin\circ \chi_{\theta}.
\]
We recall that the interior problem is given by
\begin{equation}
\label{16}
-\div\left(\frac1{1+\eps x_1} \grad \vin\right) = 4 (1+\eps x_1)\quad \mbox{in }\Omega_{\theta},\quad \vin=0\quad\mbox{on }\partial \Omega_{\theta},
\end{equation}
cf.~\eqref{32}, \eqref{32'}. The solution for $(\theta,\eps)=(0,0)$ and its normal derivative at the boundary are easily calculated,
\begin{align}
\vin(\theta=0,\eps=0)=1-|x|^2,\quad \de_n\vin(\eps=0,\theta=0)=-2.\label{phi0}
\end{align}

 Our goal is the following proposition:

 \begin{proposition}
          \label{C1}
     The mapping $(\theta,\eps)\mapsto \lambda_{\theta,\eps}$ is Fr\'echet smooth from $H^k(\partial B)$ to $H^{k-1}(\partial B)$ near  $(\theta,\eps)=(0,0)$. It holds that $\lambda_{0,0}=-2$ and
     \begin{align}
     \left.\mathrm{D}_{\eps}\right|_{(\theta,\eps)=(0,0)}\lambda_{\theta,\eps} & = -\frac12 y_1,\\
         \la\left.\mathrm{D}_{\theta}\right|_{(\theta,\eps)=(0,0)}\lambda_{\theta,\eps} ,\delta \theta\ra & = 2\mathcal{L} \delta \theta - 2\delta \theta,
     \end{align}
    where $\cL$ is the Dirichlet-to-Neumann operator associated with the Laplacian on the unit ball.
 \end{proposition}

 In particular, this also implies that $\lambda_{\theta,\eps}$ is continuously Fr\'echet-differentiable as a map from $\mathcal{M}$ to $H^{k-1}(\partial B)$.

Before proceeding, we recall a few properties of the Dirichlet-to-Neumann operator that will be relevant for our subsequent analysis.

 \begin{remark}\label{SeriesDirToNeumann}
The Dirichlet-to-Neumann operator $\cL$ is defined on $H^{\frac{1}{2}}(\de B)$ by taking some boundary datum $g\in H^{\frac{1}{2}}(\de B)$, solving the equation \begin{align}
\Delta f& =0\quad \text{in $B$},\\
 f&=g\quad
\text{on $\de B$}
\end{align}
and setting $\cL g=\de_n f|_{\de B}$. On the Fourier level, it takes a particularly simple form. Indeed, Fourier transforming the Laplace equation in angular direction, we obtain for any wave number $l\in \Z$ and radius $s\in(0,1)$ that
\[
\hat f_l''(s) + \frac1s \hat f_l'(s) = \frac{l^2}{s^2} \hat f_l(s),\quad \hat f_l(1)=\hat g_l.
\]
This ODE has the solution $\hat f_l(s) = s^{|l|}\hat g_l$, and thus 
\[
(\widehat{\cL g})_l = \hat f'_l(1) = |l|\hat g_l.
\]

\noindent From this Fourier representation, we draw three immediate conclusions:
\begin{enumerate}
\item The Dirichlet-to-Neumann operator is bounded from $H^s(\de B)$ to $H^{s-1}(\de B)$.
\item The operator maps constants to zero.
\item The operator maps $y_1$ to $y_1$.
\end{enumerate}

 The last observation implies that the derivative with respect to $\theta$ in the previous lemma is vanishing on $x_1$. This degeneracy reflects the translation invariance of the limiting elliptic problem in \eqref{16}. Moreover, the above observations yield
\begin{equation}
    \label{303}
    \la \mathcal{L} g,y_1\ra  = \la g, y_1\ra. 
\end{equation}
\end{remark}

The above proposition is an immediate consequence of the following study of
  smoothness in $\eps$ and $\theta$ for the normal derivative in the general interior problem.
 
\begin{lemma}\label{bulkDiffable}
The function $(\eps,\theta)\mapsto (\partial_n\vin)\circ \chi_{\theta}$ is well-defined and Fr\'echet-smooth on a small open neighborhood of $(\eps,\theta)=(0,0)\in \R\times H^{k}(\de B)$ and takes values in $H^{k-1}(\de B)$.  It holds that
\begin{align}
\left. \mathrm{D}_{ \eps}\right|_{(\eps,\theta)=(0,0)} (\partial_n \vin\circ \chi_{\theta} )&= -\frac{5}{2}   y_1,\\
 \la \mathrm{D}_\theta\big|_{(\eps,\theta)=(0,0)} 
(\partial_n \vin\circ \chi_{\theta}),\delta\theta\ra &   =  2 \cL \delta \theta -2 \delta \theta,
\end{align}
for any $\delta \theta\in H^k$.
\end{lemma}
 
We provide the short proof of Proposition \ref{C1}.

\begin{proof}[Proof of Proposition \ref{C1} assuming Lemma \ref{bulkDiffable}]
The smoothness follows directly from the preceding Lemma with the product rule and Lemma \ref{FreLemma} b).
The formulas for the derivative hold by the product rule because \begin{align}
\mathrm{D}_\eps\big|_{(\theta,\eps)=(0,0)}\frac{1}{1+\eps\chi_\theta\cdot e_1}&=-y_1\\
\la\mathrm{D}_\theta\big|_{(\theta,\eps)=(0,0)}\frac{1}{1+\eps\chi_\theta\cdot e_1},\de\theta\ra&=0.
\end{align}
\end{proof}

The differentiability of elliptic equations with respect to the domain is a classical topic see e.g.\ the book \cite{MR1215733} for further reading.

\begin{proof}[Proof of Lemma \ref{bulkDiffable}] 
We shall first formally derive the form of the derivatives and later justify why smoothness holds.

\medskip

\noindent\textbf{Step 1.} We begin with the formal derivation of the derivative with respect to $\eps$. Differentiating the elliptic problem \eqref{16} with respect to $\eps$ yields that 
\[
- \laplace \left.\mathrm{D}_{\eps}\right|_{(\eps,\theta)=(0,0)} \vin +  \partial_1 \phi_0  = 8 x_1,\quad \left.\mathrm{D}_{\eps}\right|_{(\eps,\theta)=(0,0)} \vin = 0\quad \mbox{on }\partial B,
\]
where $\phi_0=1-|x|^2$ is the solution to \eqref{16} at $(\theta,\eps)=(0,0)$  and thus, the problem for $\left.\mathrm{D}_{\eps}\right|_{(\eps,\theta)=(0,0)} \vin$ can be rewritten as
\[
-\laplace \left.\mathrm{D}_{\eps}\right|_{(\eps,\theta)=(0,0)} \vin =10  x_1\quad \mbox{in }B,\quad \left.\mathrm{D}_{\eps}\right|_{(\eps,\theta)=(0,0)} \vin = 0\quad \mbox{on }\partial B.
\]
The solution can be explicitly computed by rewriting the problem in polar coordinates and making the ansatz $\left.\mathrm{D}_{\eps}\right|_{(\eps,\theta)=(0,0)} \vin = f(s)\cos \alpha$, so that
\[
-s^2 f'' - s f' +f = 10  s^3.
\]
The Dirichlet boundary condition reads $f(1)=0$. The function $f(s) + \frac{5}{4}   s^3$ solves the corresponding homogeneous Cauchy--Euler equation, which is then solved via standard methods. It has a unique solution that is regular at $s=0$.  We eventually arrive at
\[
\left.\mathrm{D}_{\eps}\right|_{(\eps,\theta)=(0,0)} \vin = \frac{5}{4}     \left(s-s^3\right)\cos \alpha,
\]
and the Neumann boundary trace thus reads
\[
\partial_n \left.\mathrm{D}_{\eps}\right|_{(\eps,\theta)=(0,0)} \vin = -\frac{5}{2}   x_1.
\]
 
\medskip
 
\noindent\textbf{Step 2.} We turn now to the formal calculation of the derivative with respect to $\theta$. In the interior, we obtain that \begin{align}
-\laplace \la \left.\mathrm{D}_{\theta}\right|_{(\eps,\theta)=(0,0)} \vin,\delta \theta\ra = 0\quad\text{in $B$,}
\end{align}
because $\Delta \phi_0$ is constant. To differentiate the boundary condition, we write $\vin\circ \chi_{\theta}=0$ on $\de B$. Thus, differentiation with respect to $\theta$ yields
\[
 \la \left.\mathrm{D}_{\theta}\right|_{(\eps,\theta)=(0,0)} \vin,\delta \theta\ra            =- \nabla\phi_0\cdot \la \mathrm{D}_\theta|_{(\eps,\theta)=(0,0)}\chi_\theta,\delta\theta\ra =-\de_n\phi_0 \delta\theta
\]
and using the precise form of $\phi_0$, we conclude
\begin{equation}
\label{18}
-\laplace \left(\la \left.\mathrm{D}_{\theta}\right|_{(\eps,\theta)=(0,0)} \vin,\delta \theta\ra \right) = 0\quad\mbox{in }B,\quad \la \left.\mathrm{D}_{\theta}\right|_{(\eps,\theta)=(0,0)} \vin,\delta \theta \ra=   -2\delta \theta\quad\mbox{on }\partial B.
\end{equation}
We have to derive the derivative of the Neumann trace. We thus compute, using the explicit formula for the derivative of the normal found in Lemma \ref{DeriNormal}, 
\begin{align}
\la \left. \mathrm{D}_{\theta}\right|_{(\eps,\theta)=(0,0)} \left(\partial_n\vin\circ \chi_{\theta}\right),\delta \theta\ra & = \la \left.\mathrm{D}_{\theta}\right|_{(\eps,\theta)=(0,0)} \left((n_{\theta}\cdot \grad \vin)\circ \chi_{\theta}\right),\delta \theta\ra \\ 
 &=  n\cdot \la \left.\mathrm{D}_{\theta}\right|_{(\eps,\theta)=(0,0)} \left(\grad \vin\circ \chi_{\theta}\right),\delta \theta \ra -\partial_{\tau}\phi_0 \partial_{\tau}\delta \theta.
\end{align}
We have already seen that $\phi_0$ is a radial function and thus the second term on the right-hand side is zero. Using elementary differentiation rules, we then obtain
\[
\la \left.\mathrm{D}_{\theta}\right|_{(\eps,\theta)=(0,0)} \left(\partial_n\vin\circ \chi_{\theta}\right),\delta \theta \ra  = \la \partial_n \left.\mathrm{D}_{\theta}\right|_{(\eps,\theta)=(0,0)}   \vin ,\delta \theta \ra + \partial_n^2 \phi_0 \delta\theta.
\]
Using again the precise form of $\phi_0$ and the information in \eqref{18}, we see that 
\[
\la \left.\mathrm{D}_{\theta}\right|_{(\eps,\theta)=(0,0)} \left(\partial_n\vin\circ \chi_{\theta}\right),\delta \theta \ra = 2 \la \mathcal{L}-1, \delta \theta\ra .
\]
 \smallskip

\noindent  In the rest of the proof, we will prove the smoothness in $(\theta,\eps)$.

\medskip

\noindent\textbf{Step 3.} We shall again use an extension $\Theta\in H^{k+\frac{1}{2}}(B)$ of $\theta$ with $\Theta|_{\de B}=\theta$ and let again $\Chi_\theta(y)=(1+\Theta(y))y$ denote the corresponding extension of $\chi_\theta$ as described in the proof of Proposition \ref{tspace M}.

We first argue the differentiability of $\vin\circ \Chi_\theta$ in $(\Theta,\eps)$ as a map from $H^{k+\frac{1}{2}}(B)\times [-\delta,\delta]$ to $H^{k+\frac12}(B)$ for small $\delta$.

By pullback we obtain that the  function $\tilde{\phi}:=\vin(\theta,\eps)\circ \Chi_\theta$   fulfills the elliptic problem
 \begin{equation}\label{PDEonBall}
 \begin{aligned} 
 N(\Theta,\eps,\tilde\phi):=\div\left( \mathcal{C}(\Theta,\eps) \grad_y \tilde \phi\right)+ 4(1+\eps (\Chi_{\theta})_1)|\det\text{D}\Chi_\theta|&=0\quad \mbox{in }B,\\
 \tilde \phi  &=0\quad \mbox{on }\partial B,
 \end{aligned}
 \end{equation}
 where the coefficient matrix $\mathcal{C}$ is given by
 \begin{align*}
 \mathcal{C} & = \frac{|\det \mathrm{D}\Chi_\theta|}{1+\eps (\Chi_\theta)_1} (\mathrm{D} \Chi_\theta)^{-1} (\mathrm{D}\Chi_\theta)^{-T}.
% &  = \frac1{1+\eps \chi_1}\frac1{|\det D \chi|} \begin{pmatrix}
%  |\partial_2 \chi|^2 & -\partial_1\chi \cdot \partial_2\chi\\ -\partial_1\chi\cdot \partial_2\chi & |\partial_1\chi|^2.
%   \end{pmatrix}
 \end{align*}

\noindent We would like to apply the implicit function theorem in a neighborhood of $(0,0,\phi_0)$ to this in order to show that the solution $\tilde{\phi}$ is continuously Fr\'echet differentiable. We do this by considering the functional \begin{align}
N:H^{k+\frac{1}{2}}(B)\times [-\delta,\delta]\times H_0^{k+\frac{1}{2}}(B)\rightarrow H^{k-\frac{3}{2}}(B)
\end{align}
 defined as above, by the fact that all involved spaces are closed under composition with smooth functions as detailed in the preliminary section, this is well-defined. The boundary condition is included in the functional as we use the space $H_0^{k+\frac{1}{2}}(B)$ for the third variable.
 
 Clearly $(0,0,\phi_0)$ is a zero of this functional. The functional is Fr\'echet smooth around $(0,0,\phi_0)$ by Lemma  \ref{FreLemma} b), and the pointwise and Fr\'echet derivatives agree.
 
The derivative of $N$ with respect to $\tilde{\phi}$ at $(0,0,\phi_0)$ is $-\Delta$. The Dirichlet-Laplacian is invertible from $H^{k-\frac{3}{2}}(B)$ to $H_0^{k+\frac{1}{2}}(B)$, see for instance \cite[Thm.\ 5.1]{LionsMagenes}.

Hence the implicit function theorem shows the existence of the derivative of $\tilde{\phi}$, furthermore as $N$ is Fr\'echet smooth, $\tilde{\phi}$ is in fact Fr\'echet smooth too.

\medskip

\noindent\textbf{Step 4.} Because the map $\theta\rightarrow \Theta$ is linear and bounded, we see that $\vin\circ\Chi_\theta$ is also Fr\'echet differentiable in $(\theta,\eps)$. Furthermore, by the linearity of the trace and the product and chain rules, the map \begin{align}
\theta\mapsto(\nabla\vin)\circ\Chi_\theta|_{\de B}=(\nabla\Chi_\theta)^{-1}\nabla (\vin \circ\Chi_\theta)|_{\de B} 
\end{align}
is Fr\'echet smooth in $H^{k-1}(\de B)$, where we used the fact that $\nabla\Chi_\theta$ is invertible because $\Chi_\theta$ is a diffeomorphism (since it is a perturbation of the identity).

Finally, the normal is continuously differentiable in $H^{k-1}(\partial B)$ by Lemma \ref{DeriNormal} and hence the normal derivative is  Fr\'echet smooth by the product rule.
\end{proof}

\section{The homogenous equation}\label{S:outer}

This section is devoted to the construction and properties of $\mu_{\theta,\eps}(S)$. 
The construction requires quite a few technical preparations and is carried out in Subsection \ref{subsec mu}, which also contains the main result of the section,  Proposition \ref{mu diffable}.

\subsection{Reformulation as an integral equation}\label{reform}
We have noticed in our reformulation \eqref{41} of the jump condition that we only need to understand the normal derivative of the exterior homogeneous part $\vout$ along the boundary of the domain $\Omega_{\theta}$. It holds that
\[
\vout = \frac{W}2 (1+\eps x_1)^2 +\gamma\quad \mbox{in }\Omega_{\theta},
\]
and hence
\[
\frac1{1+\eps x_1}\partial_{n} \vout -\eps W n\cdot e_1 = -\left[\frac1{1+\eps x_1}\partial_n \vout\right].
\]
Recall that in this formulation, the normal $n$, including the one in the normal derivative, points in the direction outside of $\Omega_{\theta}$. 
Moreover, writing the elliptic equation in \eqref{33} globally in $\HH_{1/\eps}$, we obtain
\begin{equation}
\label{45}
-\div\left(\frac1{1+\eps x_1} \grad\vout\right) = \frac1{1+\eps x_1}\left[\partial_n \vout\right]\mathcal{H}^1\mres\partial \Omega_{\theta} \quad \mbox{in }\HH_{1/\eps}
\end{equation}
in the sense of distributions.

The Dirichlet condition in \eqref{34} can thus, using a single layer potential, be rewritten in the form
\begin{equation}
\label{44}
\int_{\partial \Omega_\theta} K_{0,\eps}(x, \tilde x) \frac1{1+\eps \tilde x_1} \left[(\partial_n \vout)( \tilde x)\right]\, \mathrm{d}\Ha^1( \tilde x) = \frac{W}2 (1+\eps x_1)^2 + \gamma \text{ on $\de \Omega_\theta$},
\end{equation}
for any $x\in \partial \Omega_{\theta}$, where $K_{0,\eps}$ denotes the fundamental solution of the elliptic  operator $-\div(\frac1{1+\eps x_1}\nabla\cdot )$ in $\HH_{1/\eps}$. (For a rigorous proof that this is compatible with the conditions at the axis and $\infty$, see, e.g.\ Lemma 2.1.4  in \cite{Meyer}.)

 We will pull back this equation onto $\de B$ in order to analyze the asymptotic of the boundary values. Then it reads as \begin{equation}\begin{aligned}
\label{eq B}
\mel\int_{\partial B} K_{\theta,\eps}(y,\tilde y) \frac1{1+\eps(\chi_\theta(\tilde y))_1} \left[(\partial_n \vout\circ \chi_\theta)(\tilde y)\right] \,\mathrm{d}\Ha^1(\tilde y) \\
&= \frac{W}2 \big(1+\eps(\chi_\theta(y))_1\big)^2 + \gamma \text{ on $\de B$},
\end{aligned}\end{equation}
where we set \begin{align}
K_{\theta,\eps}(y,\tilde y):= m_\theta (\tilde y)K_{0,\eps}(\chi_{\theta}(y),\chi_{\theta}(\tilde{y})).\label{def K}
\end{align}
The Jacobian $m_\theta$ was defined in \eqref{20}.

In order to derive the precise form of the kernel $K_{\theta,\eps}$, we go back to the original operator  $-\div(\frac{1}{r}\nabla\cdot)$ in \eqref{27} with the boundary and decay conditions \eqref{30}, \eqref{31}, \eqref{32b}.  Its fundamental solution $K_{\HH}$ takes the form
\begin{equation}
\label{43}
K_{\HH}((r,z),(\tilde r,\tilde z)) =\frac1{2\pi} \sqrt{r\tilde r} F\left(\frac{(r-\tilde r)^2 + (z-\tilde z)^2}{r\tilde r}\right),
\end{equation}
where the function $F:(0,\infty)\to \R$ is defined by
\begin{align*}
F(s) &= \int_0^{\pi} \frac{\cos t}{\sqrt{2(1-\cos t)+s}}\dt,
\end{align*}
see, for instance, \cite{FengSverak15} or Chapter 19 of the lecture notes \cite{SverakNotes}. For further reference, we notice the following expansion:

\begin{lemma}\label{LemmaF}
This function has an expansion \begin{align}\label{SeriesK}
F(s)=\log(\frac{8}{\sqrt{s}})-2+f_1(s)+f_2(s)\log(s)
\end{align}
for positive $s$ near $0$, where $f_1$ and $f_2$ are smooth functions with $f_1(0)=f_2(0)=0$. 
\end{lemma}

\begin{proof}
Via standard techniques, we rewrite $F(s)$ in terms of elliptic integrals,
\begin{align*}
F(s) &  = (2+s) \int_0^{\frac{\pi}2} \frac1{\sqrt{4\sin^2 (\tau) +s}} \dd\tau - \int_0^{\frac{\pi}2} \sqrt{4\sin^2(\tau) +s} \dd\tau\\
& = \frac{2+s}{\sqrt{4+s}} K\left(\sqrt{\frac{4}{4+s}}\right) -\sqrt{4+s} E\left(\sqrt{\frac{4}{4+s}}\right),
\end{align*}
where $E(k)$ and $K(k)$ denote, respectively, the complete elliptic integrals of the first and second kind with eccentricity $k$, see e.g.\ \cite{byrd2013handbook} for background reading. Both have well-known power series expansions near $k=1$, which are of the form
\[
\log \frac1{1-k} \left(a_0 + a_1 (1-k) + a_2 (1-k)^2 +\dots\right) + b_0 + b_1 (1-k) + b_2 (1-k)^2 +\dots
\]
with positive radii of convergence, see e.g.\  formulas 8.113 and 8.114 in \cite{gradshteyn2014table}. The leading order coefficients of $F(s)$ for $s$ small  are easily calculated, for instance
\[
F(s) = 2\int_0^{\frac{\pi}2} \frac{\cos \tau}{\sqrt{4\sin^2(\tau) +s}}\dd\tau +O(1) = \int_0^2 \frac1{\sqrt{\xi^2 +s}}\dd\xi +O(1) = \log \frac{1}{\sqrt{s}} + O(1).
\] 
We refer to Chapter 19 in \cite{SverakNotes} for some more details.
\end{proof}

 Changing now variables from $(r,z)$ on $\partial \Omega$ to $y\in \partial B$  yields
\begin{equation}
    \label{200}
K_{\theta,\eps}(y,\tilde y) = m_{\theta}(\tilde y)  K_{\HH}\left(\binom{1}{0} +\eps    y(1+\theta(y)) , \binom{1}{0} +\eps  \tilde y   (1+\theta(\tilde y)) \right).
\end{equation}
We shall occasionally  write $\K_{\theta,\eps}$ for the associated linear operator, 
\begin{equation}
\label{46}
\K_{\theta,\eps}f(y) = \int_{\partial B} K_{\theta,\eps}(y,\tilde y)f(\tilde y)\dd\Ha^1(\tilde y).
\end{equation}
With $\mu_{\theta,\eps}$, introduced  in Section \ref{proof strat}, 
\[
\mu_{\theta,\eps}(y) = \frac1{1+\eps (\chi_\theta)_1(y)} \left[(\partial_n \vout \circ \chi_{\theta} )(y)\right],
\]
the boundary equation  \eqref{44} now takes the short form 
\begin{equation}
\label{47}
 \K_{\theta,\eps} \mu_{\theta,\eps} = \frac{W}2(1+\eps (\chi_\theta)_1(y))^2 + \gamma.
\end{equation}

%
% Denoting by $K_{\eps}$ the kernel associated to the elliptic problem
%\begin{align*}
%-\div\left(\frac1{1+\eps x_1} \grad \varphi\right) &= f\quad \mbox{in }\HH_{1/\eps},\\
%\varphi &=0\quad \mbox{on }\partial \HH_{1/\eps},\\
%\frac1{1+\eps x_1} |\grad\varphi| &\to 0\quad\mbox{as }|x|\to \infty,
%\end{align*}
%the Dirichlet condition in \eqref{40} becomes
%\begin{equation}
%\label{42}
%\int_{\partial \Omega_{\theta}} K_{\eps}(x,\tilde x) \frac{1}{1+\eps \tilde x_1}\left[\partial_n \vout(
%\tilde x)\right]\, d\cH^1(\tilde x) = \frac{W}2 (1+\eps x_1)^2 +\gamma\quad \mbox{on }\partial \Omega_{\theta}.
%\end{equation}
We have to show that $\mu_{\eps,\theta}$  depends continuous  differentiably on $(\eps,\theta)$.
This is nontrivial as the kernel $K_{\theta, \eps}$ is (logarithmically) diverging as $\eps\to 0$, see Lemma \ref{L5} below, which is, in fact, a behavior that is inherited by the velocity $W$ and the flux constant $\gamma$. The limit $\eps\rightarrow 0$ of \eqref{47} is thus singular, so we can not expect $\mu_{\theta,\eps}$ to depend smoothly on $\eps$.
In order to make sure that at least a first derivative exists, we will need to carefully analyze the behavior of the integral kernel $K_{\theta,\eps}$ and we will have to construct appropriate boundary values.

We start our step-by-step analysis with a study of the renormalized limiting operator:
Considering   the expansion of $F$ in \eqref{SeriesK}, it is not difficult to observe formally that
\[
2\pi K_{\theta,\eps}(y,\tilde y) =  m_\theta(\tilde y)\left(\log\left(\frac1
\eps \right)+ \log(8)-2 +  m_{\theta}(\tilde y) \log\frac1{|\chi_{\theta}(y)- \chi_{\theta}(\tilde y)|}\right) +o(1),
\]
and it is thus natural to study   the kernel
\[
\widetilde K_{\theta,0}(y,\tilde y) =  \frac{1}{2\pi}{m_{\theta}(\tilde y)}\log\frac1{|\chi_{\theta}(y)-\chi_{\theta}(\tilde y)|}
\]
as an auxiliary object, which is simply the pullback of the Newtonian potential $-\frac{1}{2\pi}\log|x-\tilde{x}|$ to $\de B$.  We denote by $\widetilde \K_{\theta,0}$ the associated linear maps, cf.~\eqref{46}.

Picking (formally) $\eps=0$ in \eqref{47}, we are led to the auxiliary problem 
\[
\widetilde \K_{\theta,0}[\de_n\phi\circ\chi_{\theta}]=\const\quad \mbox{on }\de B,
\]
which corresponds to solving the Laplace equation in $\R^2\setminus \partial \Omega_\theta$ with constant Dirichlet boundary data,
    \[
    -\Delta \phi = 0\quad \mbox{in }\R^2\setminus \partial\Omega_{\theta},\quad \phi = \const\quad \mbox{on }\partial \Omega_{\theta}.
    \]
It is clear that $\phi$ is constant inside of $\Omega_{\theta}$, and one can select a unique solution by requiring that 
\[
\int_{\partial \Omega_{\theta}} [\partial_n\phi]\, \dd\Ha^1 = 1.
\]
We are thus concerned with the classical capacity potential problem for the set $\Omega_{\theta}$.

Going back to the notation introduced earlier, the normal derivative jump $\tilde \mu_{\theta,0}  = [\partial_n \phi\circ \chi_{\theta}]$ solves the problem
\begin{align}
    \widetilde \K_{\theta,0}\tilde{\mu}_{\theta,0}=\const, \quad \int_{\partial B} m_\theta \tilde{\mu}_{\theta,0}\dy=1.\label{def mu0}
\end{align}
As mentioned earlier, we will use the solution $\tilde{\mu}_{\theta,0}$  as the formal interpretation of $\mu_{\theta,\eps}$ at $\eps=0$. Differentiability properties will be studied below in Lemma \ref{Lem:muDiffable}.

%We first note the invertibility of this approximate map.

 The approximate map $\widetilde K_{\theta,0}$ is in fact smooth in $\theta$ and its derivative can be understood quite easily. 
In preparation for proving this, we first need to show that these kinds of kernels behave well on $H^k(\partial B)$-spaces.

\subsection{Mapping properties of perturbed logarithmic kernels}

Our first goal will be to show that $\mathcal{K}_{\theta,0}$ is smooth in $\theta$ as a map from $H^{k-1}(\de B)$ to $H^k(\de B)$. This requires us to develop some machinery about the general mapping properties of such kernels.

The key result of this subsection is Proposition \ref{Kernelsmooth}.\newline

We first want to understand the kernel at $\theta=0$, for this it will be convenient to understand the action of the Newtonian potential operator $\tilde\K_{0,0}$ in terms of Fourier series. For this we will occasionally identify $\R^2$ with $\mathbb{C}$, so that $e^{i\alpha}$ parametrizes $\de B$ with $\alpha\in \T$.
 
 \begin{lemma}\label{K inv}
The convolution kernel $\log|x-y|$ acts on $L^2(\de B)$ as the Fourier multiplier $e^{ik\alpha}\rightarrow \frac{-\pi}{|k|}e^{ik\alpha}$ for $k\in \Z_{\neq 0}$ and maps $1$ to $0$, where we used the identification with $\mathbb{C}$. In particular, for $y\in \de B$ it holds that \begin{align}
&\int_{\de B}\log|\tilde{y}-y|\dd\tilde{y}=0\label{63},\\
&\int_{\de B}x_1\log|\tilde{y}-y|\dd \tilde{y}=-\pi y_1\label{64},
\end{align}
and the associated linear map maps $H^{k-1}(\de B)$ to $H^{k}(\de B)$ boundedly.
\end{lemma}
\begin{proof}
We identify $\R^2$ with $\mathbb{C}$. On the Fourier level, the convolution can be turned into a multiplication via a simple change of variables ($\gamma=\alpha-\beta$),
\begin{align*}
-2\pi  \widehat{\left(\widetilde{\K}_{0,0} f\right)}(k) & = \frac1{2\pi}\int_0^{2\pi} e^{-ik\alpha} \int_0^{2\pi} \log |e^{i\alpha}-e^{i\beta}| f(\beta)\dd \beta\dd\alpha\\
& =  \frac{1}{2\pi} \int_0^{2\pi} e^{-ik\beta}f(\beta) \dd \beta\int_0^{2\pi} e^{-ik\gamma}\log |e^{i\gamma}-1|\dd \gamma\\
& =   \hat f(k) \int_0^{2\pi} e^{- i k\gamma}\log |e^{i\gamma}-1|\dd \gamma,
\end{align*}
for any wave number $k\in \Z$. We thus need to show that 
\begin{equation}\label{eq2}
    \int_0^{2\pi}e^{ik\alpha}\log|e^{i\alpha}-1|\dd\alpha =\begin{cases}-\frac{\pi}{|k|}&\mbox{if }k\not=0,\\
    0&\mbox{if }k=0.
    \end{cases}
\end{equation}

\noindent  It will be useful to use the series expansion of the logarithm, which requires an approximation to stay in the radius of convergence, therefore we will show \begin{equation}\label{eq3}
\int_0^{2\pi}e^{ik\alpha}\log|1-(1-\delta)e^{i\alpha}|\da\overset{\delta\searrow 0}{\longrightarrow} \begin{cases}-\frac{\pi}{|k|}&\mbox{if }k\not=0,\\
    0&\mbox{if }k=0.
    \end{cases}
\end{equation}

\noindent This implies \eqref{eq2}, as we can exchange the limit and the integral by dominated convergence. Now by making use of the identity $\log|z|=\Re \log z$ and of the Taylor series of the logarithm around $1$, we see that  the left hand side of \eqref{eq3} equals\begin{align}
&\int_0^{2\pi}e^{ik\alpha}\log|1-(1-\delta)e^{i\alpha}|\da\\
&=\int_0^{2\pi}e^{ik\alpha}\Re(\log1-(1-\delta)e^{i\alpha})\da\\
&=-\int_0^{2\pi}e^{ik\alpha}\Re\left(\sum_{j=1}^\infty \frac{(1-\delta)^j}{j}e^{ij\alpha}\right)\da\\
&=-\sum_{j=1}^\infty\frac{(1-\delta)^j}{2j}\int_0^{2\pi}e^{ik\alpha}(e^{ij\alpha}+e^{-ij\alpha})\,\mathrm{d}\alpha
\end{align}

\noindent Here we used the uniform convergence of the series to exchange the sum and the integral. Now in this sum, only the summand for $j=\pm k$ is nonzero by orthogonality, which implies \eqref{eq3}.

The special case \eqref{64} holds because $\cos \alpha=y_1$.

\end{proof}

In order to see what kind of terms we will need to study for $\theta\neq 0$, we first expand the kernel.

\begin{lemma}\label{L3}
For any $y,\tilde y\in \partial B$, it holds that
\begin{equation}
\label{48} 
\begin{aligned}
\mel
\log |y(1+\theta(y)) - \tilde y(1+\theta(\tilde y))|\\
&= \log |y-\tilde y| + \frac12 \log \left( 1+\theta(y)+\theta(\tilde y) + \theta(y)\theta(\tilde y)  +\frac{(\theta(y)-\theta(\tilde y))^2}{|y-\tilde y|^2}\right).
\end{aligned}
 \end{equation}
\end{lemma}

\begin{proof}
We start by observing the orthogonality condition $(y-\tilde y)\cdot (y+\tilde y)=0$ for any $y,\tilde y\in \partial B$, so that an application of the Pythagorean theorem yields 
\begin{align*}
\mel \left|y(1+\theta(y)) - \tilde y (1+\theta(\tilde y))\right|^2\\
 & = \left|  y-\tilde y  +  \frac12 (y-\tilde y)\left(\theta(y)+\theta(\tilde y))\right) +\frac12(y+\tilde y)(\theta(y)-\theta(\tilde y))\right|^2\\
 & = |y-\tilde y|^2 \left(1+\frac12(\theta(y)+\theta(\tilde y))\right)^2 +\frac14|y+\tilde y|^2 (\theta(y)-\theta(\tilde y))^2.
\end{align*}
We apply the parallelogram law $|y-\tilde y|^2 +|y+\tilde y|^2 = 4$ on the unit circle to rewrite this identity as
\begin{align*}
\mel \left|y(1+\theta(y)) - \tilde y (1+\theta(\tilde y))\right|^2\\
 & = |y-\tilde y|^2\left[ \left(1+\frac12(\theta(y)+\theta(\tilde y))\right)^2 -\frac14 (\theta(y)-\theta(\tilde y))^2\right]  + (\theta(y)-\theta(\tilde y))^2\\
 & = |y-\tilde y|^2 \bigl(1+\theta(y)+\theta(\tilde y) +\theta(y)\theta(\tilde y)\bigr) +(\theta(y)-\theta(\tilde y))^2.
\end{align*}
Therefore, applying the logarithm yields \eqref{48}.
 \end{proof}

This motivates that we should study kernels involving difference quotients of $\theta$.
 
 %The term correcting the Newtonian potential occuring in the expansion \eqref{48} makes it necessary to study the boundedness of generic smooth kernels that involve difference quotients. This will be the content of the following lemma.

\begin{lemma}\label{L4}
Let $k\ge 5$ and $\delta\in(0,1)$ be given, let $d \geq 1$ be an integer and let \begin{align} 
&r_1:B_{\delta}(0)^{\R^{4d}}\times [0,\delta]\times(\de B)^2\to \R\\
&r_2:B_{\delta}(0)^{\R^{3d}}\times [0,\delta]\times(\de B)^2\to \R
\end{align}

\noindent be smooth. Assume that $\tilde\theta\in H^k(\de B,\R^d)$ is such that $\norm{\tilde\theta}_{H^k(\de B, \R^d)}\leq \delta'$ with $\delta'\ll\delta$.
 Consider the kernels
\begin{align}
& R_1(y,\tilde y) := r_1\left(\tilde\theta(y),\tilde\theta(\tilde y), \frac{(\tilde\theta_1(y)-\tilde\theta_1(\tilde y))^2}{|y-\tilde y|^2}, \frac{(\tilde\theta_2(y)-\tilde\theta_2(\tilde y))^2}{|y-\tilde y|^2}, \dots ,\partial_{\tau}\tilde\theta(\tilde y),\eps,y,\tilde{y}\right),\quad\label{smoothKernel1}\\
& R_2(y,\tilde y) := r_2\bigl(\tilde\theta(y),\tilde\theta(\tilde y) ,\partial_{\tau}\tilde\theta(\tilde y),\eps,y,\tilde{y}\bigr)\log|y-\tilde{y}|,\label{smoothKernel2}
\end{align}
and denote by $\mathcal{R}_1$ and $\mathcal{R}_2$ the associated linear maps given by \begin{align} \mathcal{R}_{1/2} f(y)=\int_{\de B} R_{1/2}(y,\tilde{y}) f(\tilde{y})\,\mathrm{d}\tilde{y}.\end{align}
These maps are bounded from $H^{k-1}(\partial B)$ to $H^k(\partial B)$, and it holds
\begin{align}
&\|\mathcal{R}_1 f\|_{H^k} \lesssim_{\delta,\delta'}   \|f\|_{H^{k-1}}\norm{r_1}_{C^{k+2}},\\
&\|\mathcal{R}_2 f\|_{H^k} \lesssim_{\delta,\delta'}   \|f\|_{H^{k-1}}\norm{r_2}_{C^{k+2}},
\end{align}
for any $f\in H^{k-1}(\partial B)$, uniformly in $r_{1/2}$.
\end{lemma} 

Here comes our key result of this subsection.

\begin{proposition}\label{Kernelsmooth}
Let $\mathcal{R}_1$ and $\mathcal{R}_2$ be as above and $k\geq 5$, then $(\tilde\theta,\eps)\rightarrow \mathcal{R}_1$ and $(\tilde\theta,\eps) \rightarrow \mathcal{R}_2$ are Fr\'echet smooth as maps from $H^{k}(\de B)\times \R$ to $C(H^{k-1}(\de B),H^k(\de B))$, locally around $0$ and their Fr\'echet derivatives are given through the pointwise derivatives of the kernels.

%Moreover the operatornorm of the $l$-th derivative is $\lesssim \norm{r_{1/2}}_{C^{k+l}}$ locally around $0$.
\end{proposition}
\begin{proof}[Proof of the proposition using Lemma \ref{L4}]
We want to show this by applying Lemma \ref{FreLemma} a) to each derivative of $\mathcal{R}_{1/2}$. It holds that $C(H^{k-1},H^k)\hookrightarrow C(H^{k-1},L^2)$ and we want to use this inclusion for the application of the lemma. 

Because $r_1$ and $r_2$ are smooth and their arguments are continuous by the Sobolev embedding $H^{k}\hookrightarrow C^1$, we see that the kernels $R_{1/2}(y,\tilde{y})$ are $C_{\tilde\theta,\eps}^1(L_y^2)$, uniformly in $\tilde{y}$. 
Therefore, we see that $\mathcal{R}_{1/2}$ is Gateux-differentiable as map to $C(H^{k-1},L^2)$, since $H^{k-1}\hookrightarrow L^\infty$, and the derivative is the same as the pointwise one.\newline

The crucial observation is now that every derivative of $\mathcal{R}_{1/2}$ has the same structure \eqref{smoothKernel1} resp.\ \eqref{smoothKernel2} with a different smooth $r_{1/2}$ after changing the role of $\tilde\theta$. Indeed, for the derivative with respect to $\eps$ this is clear. The derivative with respect to $\theta$ splits into various partial derivatives of $r_{1/2}$ by the chain rule. For all partial derivatives, the structure is clear, except when the derivative is applied to a difference quotient, in which case we write \begin{align}
&\de_i r_{1/2}(\tilde\theta(y),\dots)\frac{\scalar{\tilde\theta_j(y)-\tilde\theta_j(\tilde{y})}{\delta\tilde\theta_j(y)-\delta\tilde\theta_j(\tilde{y})}}{|y-\tilde{y}|^2}\\
&=\de_i r_{1/2}(\tilde\theta(y),\dots)\\
&\quad\times\frac{|(\tilde\theta_j+\delta\tilde\theta_j)(y)-(\tilde\theta_j+\delta\tilde\theta_j)(\tilde{y})|^2-|\tilde\theta_j(y)-\tilde\theta_j(\tilde{y})|^2-|\delta\tilde\theta_j(y)-\delta\tilde\theta_j(\tilde{y})|^2}{2|y-\tilde{y}|^2}.\nonumber
\end{align}
This is of the structure \eqref{smoothKernel1} resp.\ \eqref{smoothKernel2} after replacing $d$ with $3d$ and $\theta$ with $(\tilde\theta,\tilde\delta\theta,\tilde\theta+\tilde\delta\theta)$. Since we can iterate this argument, it follows that indeed all derivatives have this structure.

In particular, via iteration we see that $\mathcal{R}_{1/2}$ have all Gateaux derivatives as $C(H^{k-1},L^2)$ valued maps.
All their derivatives are of the form \eqref{smoothKernel1} or \eqref{smoothKernel2}, respectively, and thus, thanks to Lemma \ref{L4},  these are  bounded in $C(H^{k-1},H^k)$. Therefore, Lemma \ref{FreLemma}  applied to every derivative of $\mathcal{R}_{1/2}$ yields the statement.
\end{proof}

\begin{proof}[Proof of Lemma \ref{L4}]
The proofs for the two kernels are very similar, we will therefore do them simultaneously.

First observe that by the Sobolev embedding $H^{k} \hookrightarrow C^1 $ and the condition $\norm{\tilde\theta}_{H^k} \ll \delta$, the expressions \eqref{smoothKernel1} and \eqref{smoothKernel2} are well-defined.
We further note that $\tilde\theta\in C^{k-1} $ by the Sobolev embedding and by the fundamental theorem of calculus $\frac{(\tilde\theta(y)-\tilde\theta(\tilde y))^2}{|y-\tilde y|^2}\in C^{k-2}(\partial B)^2$.  Hence, 
\begin{align}\label{rCk-2}
\left\{(y,\tilde y)\mapsto 
r_{1/2}\left(\tilde\theta(y),\tilde\theta(\tilde y), \frac{(\tilde\theta_1(y)-\tilde\theta_1(\tilde y))^2}{|y-\tilde y|^2},\dots ,\partial_{\tau}\tilde\theta(\tilde y),\eps,y,\tilde{y}\right)\right\}\in C^{k-2}((\partial B)^2),\qquad
\end{align}
It follows by   Schur's Lemma (see e.g.\ \cite[Appendix A]{Grafakos2})
that the map $\mathcal{R}_{1/2}$ is bounded from $L^2 $ to $L^2 $, and furthermore, it is continuous with respect to  $\tilde\theta\in H^k $ in that topology. In particular, we can assume that $\tilde\theta$ is smooth; the general statement then follows from smooth approximation and the lower semicontinuity of the $H^k$-norm with respect to $L^2$ convergence.

To estimate the $H^{k}$-norm it suffices to estimate $\norm{\partial_{\tau}^k \mathcal{R}_{1/2} f}_{L^2}$ as we already control the $L^2$-norm. For the derivatives, we note the following commutator rule, 
\begin{align*}
\partial_{\tau} \mathcal{R}_{1/2} f (y) - \mathcal{R}_{1/2}\partial_{\tau} f(y) = 
  \int_{\partial B}\left(\partial_{\tau}+\tilde \partial_{\tau}\right) R_{1/2}(y,\tilde y) f(\tilde y)\dd \Ha^1(\tilde y),
\end{align*}
where $\tilde \partial_{\tau}$ denotes the tangential derivative with respect to $\tilde y$, which follows via a simple integration   by parts. Setting $R_{1/2}^{(j)} = (\partial_{\tau} + \tilde \partial_{\tau})^j R_{1/2}$ and denoting by $\mathcal{R}_{1/2}^{(j)}$ the associated linear operator, we thus find via iteration that
\begin{equation}\label{53}
\begin{aligned}
 \|\partial_{\tau}^{k}( \mathcal{R}_{1/2} f)\|_{L^2}\lesssim \norm{\de_\tau(\mathcal{R}_{1/2}\de_\tau^{k-1} f) }_{L^2}+\sum_{j=1}^{k}\norm{\mathcal{R}_{1/2}^{j}\de_\tau^{k-j} f}_{L^2}.
\end{aligned}
\end{equation}
We distinguish the cases $j\leq k-2$ and $j=k-1,k$ in the following.

One crucial observation is that the derivative $\de_\tau+\tilde{\de}_\tau$ is $0$ on the difference $y-\tilde{y}$ and hence it does not act on the logarithm in $R_2$ but only on the prefactor $r_{2}$. Hence by \eqref{rCk-2}, we see that for $j\leq k-2$, the kernel $R_{1/2}^j$ is  an $L^\infty$ function, possibly multiplied by $\log|y-\tilde y|$. Hence the corresponding summands are controlled by Schur's Lemma.

% Note that the derivative $\de_\tau+\tilde{\de}_\tau$ is $0$ on the difference $y-\tilde{y}$. In particular it holds that \begin{align}
%\mathcal{R}_2^{\ell}=(\de_\tau+\tilde{\de}_{\tau})\left(r\left(\theta(y),\theta(\tilde y), \frac{(\theta(y)-\theta(\tilde y))^2}{|y-\tilde y|^2} ,\partial_{\tau}\theta(\tilde y),\eps,y,\tilde{y}\right)}\right)\log|x-y|
% \end{align}

The summands for $j=k-1,k$ need to be handled more carefully. %Note that the distances $|y- \tilde{y}|$ are invariant under the differential operator $\de_\tau+\tilde{\de}_\tau$, which implies that the difference quotients remain difference quotients under that differentiation and that the derivatives do not act on the logarithm.

The kernels $R_{1/2}^{k-1}$ and $R_{1/2}^{k}$ can be computed with the help of the product and chain rule, yielding for instance
\begin{align}
R_{1}^{j}(y,\tilde{y})=&\sum_{(J,N)\in I_j}C_{J,N}\mathrm{D}^{J}r_1
(\dots)%\left(\tilde\theta(y),\tilde\theta(\tilde y), \frac{(\tilde\theta(y)-\tilde\theta(\tilde y))^2}{|y-\tilde y|^2} ,\partial_{\tau}\tilde\theta(\tilde y),\eps,y,\tilde{y}\right)
\\
&\times\prod_{(l_1,\dots l_4,m_1,\dots m_4,i_1,\dots i_4)\in N}\left(\de_\tau^{m_{1}}\tilde\theta_{i_1}(y)\right)^{l_1}\times\left(\tilde{\de}_\tau^{m_2}\tilde\theta_{i_2}(\tilde{y})\right)^{l_2}\nonumber\\
&\times\left((\de_\tau+\tilde{\de}_\tau)^{m_{3}}\left(\frac{(\tilde\theta_{i_3}(y)-\tilde\theta_{i_3}(\tilde y))^2}{|y-\tilde y|^2}\right)\right)^{l_3}\times\left(\tilde{\de}_\tau^{m_4+1}\tilde\theta_{i_4}(\tilde{y})\right)^{l_4},\nonumber
\end{align}

\noindent where $C_{J,N}\in \R$ is a prefactor coming from the product rule and the indexing set $I_j$ is some complicated subset of $\{0,\dots,j\}^6\times (\{0,\dots, j\}^{8}\times \{1,\dots,d\}^4)$ such that for all indices in $N$ it holds that $ l_1m_{1}+l_2m_{2}+l_3m_{3}+l_4m_{4}\leq j$ (as there are only $j$ derivatives involved). Besides this property, the structure of $I_j$ is irrelevant for us. The expression for $R_2^j(y,\tilde y)$ is similar but simpler because the difference quotient terms are absent.

Using the triangle inequality, we estimate all summands separately. 
As long as all the derivatives in a summand are all in $L^\infty$, the corresponding summand of the kernel is $L^2$-bounded by the same argument as above. Because $H^1 \hookrightarrow C^0 $, this is precisely the case if $m_1,m_2\le k-1$ and  $m_3,m_4 \le k-2$. For the other summands, where this is not the case, there is one singular factor and there are at most $2$ derivatives on the other factors and as a consequence of $k\geq 5$ and $H^3\hookrightarrow C^2$, all other factors are $C^2$.
We are thus left with the study of the following types of expressions, whose $L^2$-boundedness in $y$ we have to show:
 %
% The only summands not covered by this are those where at least $k-2$ derivatives are on the same factor, which in particular implies that the remaining factors in that summand are $C^2$ since $k\geq 5$ and $H^{k-2}\hookrightarrow C^2$. \smallskip
 %
% 
\begin{align}
\text{I}& :=\int_{\de B} g(y,\tilde{y})\frac{(\de_\tau+\tilde{\de}_\tau)^l(\tilde\theta_i(y)-\tilde\theta_i(\tilde y))^2}{|y-\tilde y|^2} \,\mathrm{d}\tilde{y},\\
\text{II}&:= \int_{\de B} g(y,\tilde{y})\tilde{\de}_\tau^{l+1}\tilde\theta_i(\tilde y) \,\mathrm{d}\tilde{y} ,\\
\text{III}&:=\int_{\de B} g(y,\tilde{y})\tilde{\de}_\tau^{l+1}\tilde\theta_i(\tilde y)\log|y-\tilde{y}| \,\mathrm{d}\tilde{y},\\
\text{IV}&:= \int_{\de B} g(y,\tilde{y})\de_\tau^{k}\tilde\theta_i( y) \,\mathrm{d}\tilde{y},\\
\text{V}&:=\int_{\de B} g(y,\tilde{y})\de_\tau^{k}\tilde\theta_i( y)\log|y-\tilde{y}| \,\mathrm{d}\tilde{y}.
\end{align}

\noindent Here $g$ is a placeholder for a $C^2(\de B\times \de B)$ function, which is given through the product of the remaining factors in the above expression for $R_1^j(y,\tilde y)$ (or, analogously, $R_2^j(y,\tilde y)$) with $f(\tilde y)$ (if $j=k$) or $\partial_{\tau} f(\tilde y)$ (if $j=k-1$), and $l\in\{k-1,k\}$.

%of the argument $f\in H^{k-1}$ and the regular factors of the kernel, and whose $C^2$ norm is bounded in terms of $\theta$,$\norm{f}_{H^{k-1}}$ and $\norm{r_{1/2}}_{C^{k+2}}$ and $l\in \{k-2,k-1,k\}$.\medskip

The integrals $\text{IV}$ and $\text{V}$ are trivially bounded by pulling $\de_\tau^{k}\tilde\theta_i(y)\in L^2(\partial B)$ out of the integral.
 The integral $\text{II}$ is bounded, because after an integration by parts,
 \[
 \text{II} = -\int_{\partial B} \tilde \partial_{\tau} g(y,\tilde y)\partial_{\tau}^l \tilde \theta_i(\tilde y)\dd\tilde y,
 \]
which is uniformly controlled by $  \norm{\nabla g}_{L^\infty}\norm{\theta}_{H^k}$. 
The remaining two integrals require more attention.

Using the product rule the integral $\text{I}$ can be  decomposed into a sum of terms of the type \begin{align}
\int_{\de B}g(y,\tilde{y})\frac{(\de_\tau^{l-m}\tilde\theta_i(y)-\tilde{\de}_{\tau}^{l-m}\tilde\theta_i(\tilde y))(\de_\tau^{m}\tilde\theta_i(y)-\tilde{\de}_{\tau}^{m}\tilde\theta_i(\tilde y))}{|y-\tilde y|^2}\,\mathrm{d}\tilde{y},\label{int I}
\end{align}
for $m=0,\dots l$. If $m,l-m\leq k-2$ then $\de_\tau^{l-m}\tilde\theta$ and $\de_\tau^{m}\tilde\theta$ are Lipschitz and the kernel is pointwise bounded and hence also $L^2$-bounded. Of the remaining cases we only consider $m=k$, the others are symmetric or easier. 

We may split \begin{align}
g(y,\tilde{y})=g_1(y)+g_2(y,\tilde{y})|y-\tilde{y}| \quad\text{and}\quad g(y,\tilde{y})=\tilde{g}_1(\tilde{y})+\tilde{g}_2(y,\tilde{y})|y-\tilde{y}|\label{split g}
\end{align}
with $g_1,\tilde{g}_1,g_2,\tilde{g}_2$ bounded, which is possible by using a Taylor expansion around $y=\tilde{y}$. This leaves us with \begin{align}
&\int_{\de B}g(y,\tilde{y})\frac{(\tilde\theta_i(y)-\tilde\theta_i(\tilde y))(\de_\tau^{k}\tilde\theta_i(y)-\tilde{\de}_{\tau}^{k}\tilde\theta_i(\tilde y))}{|y-\tilde y|^2}\,\mathrm{d}\tilde{y}\\
&=-\underbrace{\int_{\de B} (\tilde{g}_1\tilde{\de}_\tau^k\tilde\theta_i)(\tilde{y})\frac{\tilde\theta_i(y)-\tilde\theta_i(\tilde{y})}{|y-\tilde{y}|^2}\,\mathrm{d}\tilde{y}}_{=:\text{I}_1}+\underbrace{g_1(y)\de_\tau^k\tilde\theta_i(y)\int_{\de B}\frac{\tilde\theta_i(y)-\tilde\theta_i(\tilde{y})}{|y-\tilde{y}|^2}\,\mathrm{d}\tilde{y}}_{=:\text{I}_2}\\
&\quad-\underbrace{\int_{\de B}\tilde{\de}_\tau^k\tilde\theta_i(\tilde{y})\tilde{g}_2(y,\tilde{y})\frac{\tilde\theta_i(y)-\tilde\theta_i(\tilde{y})}{|y-\tilde{y}|}\,\mathrm{d}\tilde{y}}_{=:\text{I}_3}+\underbrace{\de_\tau^k\tilde\theta_i(y)\int_{\de B}g_2(y,\tilde{y})\frac{\tilde\theta_i(y)-\tilde\theta_i(\tilde{y})}{|y-\tilde{y}|}\,\mathrm{d}\tilde{y}}_{=:\text{I}_4}.\nonumber
\end{align}

\noindent  The integrand in $\text{I}_4$ is uniformly bounded since $\tilde\theta_i$ is Lipschitz and hence the term is bounded in $L^2$. The third term $\text{I}_3$ is bounded uniformly in $y$ because  $\tilde\theta_i$ is Lipschitz and $\partial_{\tau}^k \tilde \theta_i \in L^2(\partial B)$. For the first term $\text{I}_1$ we use that singular integral kernels of the form $\frac{A(y)-A(\tilde{y})}{|y-\tilde{y}|^2}$ with $A$ Lipschitz are known as Calder\'on commutators and are $L^2$ bounded, see e.g.\ Example 4.3.8 in \cite{Grafakos2}. Hence $\text{I}_1$ is in $L^2$.
For the integral in $\text{I}_2$, we let $\tilde{y}'\in \de B$ be the reflection of $\tilde{y}$ on $y$, so that $|y-\tilde y|=|y-\tilde y'|$. Then it holds that \begin{align}
 \int_{\de B}\frac{\tilde\theta_i(y)-\tilde\theta_i(\tilde{y})}{|y-\tilde{y}|^2}\,\mathrm{d}\tilde{y} =\frac12  \int_{\de B}\frac{\tilde\theta_i(y)-\tilde\theta_i(\tilde{y})}{|y-\tilde{y}|^2}\,\mathrm{d}\tilde{y} + \frac12 \int_{\de B}\frac{\tilde\theta_i(y)-\tilde\theta_i(\tilde{y}')}{|y-\tilde{y}|^2}\,\mathrm{d}\tilde{y}
\end{align}
We use a Taylor expansion in each integrand to the effect that
\[
 \int_{\de B}\frac{\tilde\theta_i(y)-\tilde\theta_i(\tilde{y})}{|y-\tilde{y}|^2}\,\mathrm{d}\tilde{y} = \grad\tilde \theta_i(y) \cdot \int_{\partial B} \frac{y-\frac12(\tilde y +\tilde y')}{|y-\tilde y|^2}\dd \tilde y +O(\|\tilde \theta_i\|_{C^2}).
\]
Upon a rotation, we may assume that $y=e_1$. In this case, the center point $\frac12(\tilde y+\tilde y')$ is the projection of $\tilde y$ onto the $x_1$ axis and thus, rewriting the integral in polar coordinates, we find
\[
 \int_{\partial B} \frac{y-\frac12(\tilde y +\tilde y')}{|y-\tilde y|^2}\dd \tilde y  = e_1 \int_0^{2\pi} \frac{1-\cos \alpha}{2(1-\cos{\alpha})}\dd \alpha = \pi e_1.
\]
We conclude that the integral in $\text{I}_2$ is bounded uniformly in terms of the second derivative of $\|\tilde \theta_i\|_{C^2}$, and hence, the term $I_2$ is bounded in $L^2$.

We finally turn to the integral $\text{III}$, which we consider for $l=k$ only. We may    integrate by parts and obtain \begin{align}
\text{III}=-\int_{\de B} \tilde{\de}_\tau^{k}\tilde\theta_i(\tilde y)g(y,\tilde{y})\tilde{\de}_\tau\log|y-\tilde{y}| \,\mathrm{d}\tilde{y}-\int_{\de B} \tilde{\de}_\tau g(y,\tilde{y})\tilde{\de}_\tau^{k}\tilde\theta_i(\tilde y)\log|y-\tilde{y}| \,\mathrm{d}\tilde{y}.
\end{align}

\noindent The second summand is controlled by $ \norm{\nabla g}_{L^\infty}\norm{\theta}_{H^k}$ uniformly in $y$ by the Cauchy--Schwarz inequality as the logarithm is in $L^2$.
For the first summand we again use $\tilde{g}_1,\tilde{g}_2$ as above in \eqref{split g} and obtain that 
\begin{equation}\label{log int}
\begin{aligned}
&\int_{\de B} \tilde{\de}_\tau^{k}\tilde\theta_i(\tilde y)g(y,\tilde{y})\tilde{\de}_\tau\log|y-\tilde{y}| \,\mathrm{d}\tilde{y}\\
&=\int_{\de B} \tilde{\de}_\tau^{k}\tilde\theta_i(\tilde y)\tilde{g}_1(\tilde{y})\tilde{\de}_\tau\log|y-\tilde{y}| \,\mathrm{d}\tilde{y}+\int_{\de B} \tilde{\de}_\tau^{k}\tilde\theta_i(\tilde y)\tilde{g}_2(y,\tilde{y})|y-\tilde{y}|\tilde{\de}_\tau\log|y-\tilde{y}| \,\mathrm{d}\tilde{y}.
\end{aligned}
\end{equation}
 We rewrite the first term \begin{align}
&\int_{\de B} \tilde{\de}_\tau^{k}\tilde\theta_i(\tilde y)\tilde{g}_1(\tilde{y})\tilde{\de}_\tau\log|y-\tilde{y}| \,\mathrm{d}\tilde{y}=-\de_\tau \int_{\de B} \tilde{\de}_\tau^{k}\tilde\theta_i(\tilde y)\tilde{g}_1(\tilde{y})\log|y-\tilde{y}| \,\mathrm{d}\tilde{y},
\end{align}

\noindent and apply Lemma \ref{K inv} below to observe that this is in $L^2$.
The other term  is bounded uniformly in $\tilde{y}$ by the Cauchy--Schwarz inequality because  $|y-\tilde{y}|\tilde{\de}_\tau\log|y-\tilde{y}|$ is bounded.

To finish the proof it remains to estimate the terms \begin{align}
\norm{\de_\tau\mathcal{R}_{1}\de_\tau^{k-1} f}_{L^2}\quad \text{and}\quad \norm{\de_\tau\mathcal{R}_{2}\de_\tau^{k-1} f}_{L^2}.
\end{align}

\noindent In the case of $\mathcal{R}_1$ this is done via the Cauchy--Schwarz inequality and a uniform pointwise bound because the kernel $\de_\tau R_1$ is uniformly pointwise bounded as a consequence of \eqref{rCk-2}.
In the case of $\mathcal{R}_2$, we have
\begin{align}
\de_\tau\mathcal{R}_{2}\de_\tau^{k-1} f=\de_\tau\int\tilde\de_\tau^{k-1} f(\tilde{y}) r_2(\dots)\log|x-y|\,\mathrm{d}\tilde{y}.
\end{align}
Thanks to \eqref{rCk-2} the factor $r_2$ is in $C^{k-2}$ and one can treat this integral term in the same way as we have treated the first summand of $\text{III}$ in \eqref{log int}.
\end{proof}

\subsection{The limiting map $\widetilde{\mathcal{K}}_{\theta,0}$ and the approximate solution $\tilde{\mu}_{\theta,0}$} \label{SS1}
We first want to understand the behavior of the limiting kernel $\widetilde{\mathcal{K}}_{\theta,0}$, which is smooth (Lemma \ref{L7}), easy to understand in terms of Fourier series at $\theta=0$ (Lemma \ref{K inv}) and invertible modulo constants (Lemma \ref{L10}). %In particular, the approximate

\begin{lemma}\label{L7}
The linear map $\widetilde \K_{\theta,0}: H^{k-1}(\partial B)\to H^k(\partial B)$ is  Fr\'echet smooth in $\theta\in \linebreak H^k(\partial B)$, provided that $\|\theta\|_{H^k}$ is sufficiently small. Moreover, for all $f\in H^{k-1}(\de B)$, it holds that
\[
\la \left. \mathrm{D}_{\theta}\right|_{\theta=0} \widetilde \K_{\theta,0} f, \delta \theta\ra  =\tilde{\mathcal{K}}_{0,0} ( \delta\theta\,  f) -\frac{\delta \theta}{4\pi}\int_{\partial B} f\, \mathrm{d}\Ha^1 -\frac1{4\pi}\int_{\partial B} f\delta \theta\, \mathrm{d}\Ha^1.
\]
\end{lemma} 
For the convenience of the reader, we recall that in the statement of the lemma, 
$ \tilde{\mathcal{K}}_{0,0}$ is the operator associated with the Newtonian potential $  - \frac1{2\pi}\log |y-\tilde y|$.
We remark that the derivative $\left. \mathrm{D}_{\theta}\right|_{\theta=0} \widetilde \K_{\theta,0}$ has a nontrivial kernel which reflects the translation invariance of the limiting problem. Indeed, using \eqref{64} below we immediately see that the derivative vanishes for $f=\const $ and $\theta(x)=y_1$.

\begin{proof}
We use the expansion in Lemma \ref{L3} to note that $\widetilde K_{\theta,0}$ depends smoothly on $\theta$, that is
\begin{align}
\mel
\widetilde K_{\theta,0}(y,\tilde y)\\
&=-\frac{m_\theta(\tilde{y})}{2\pi}\biggl(\log|y-\tilde{y}|+\frac12\log\left(1+\theta(y)+\theta(\tilde y) +\theta(y)\theta(\tilde y)  +\frac{(\theta(y)-\theta(\tilde y))^2}{|y-\tilde y|^2}\right)\biggr).
\end{align}

\noindent  By the definition of $m_\theta$, everything except $\log|y-\tilde{y}|$ is smooth in $\theta$, locally around $0$, hence Proposition \ref{Kernelsmooth} is applicable and yields the smoothness. The formula for the first derivative at $\theta=0$ can be read off easily from the above expansion.
\end{proof}

The following Lemma follows directly from Lemma \ref{K inv}. 
\begin{lemma}\label{L6}
For any   $g\in H^k(\partial B)$, the problem  
\begin{align}\label{57}
 \tilde{\mathcal{K}}_{0,0}f - g &= \const\quad\text{ on $\de B$},\\
\int_{\partial B} f\, \mathrm{d}\Ha^1 &=0,
\end{align}
has a unique solution $f\in H^{k-1}(\partial B)$. Moreover, it holds that 
\[
\int_{\partial B} \tilde{\mathcal{K}}_{0,0}f\dd\Ha^1 = 0,
\]
and 
\begin{equation}
\label{59}
\|f\|_{H^{k-1}} \lesssim \|g\|_{H^k}\lesssim \|f\|_{H^{k-1}}.
\end{equation}
\end{lemma}

The invertibility of $\widetilde \K_{\theta,0}$ near $\widetilde \K_{0,0}$ follows now via a perturbation argument:

\begin{lemma}\label{L10}
For any sufficiently small $\theta\in H^k(\de B)$ and any $g\in H^k(\partial B)$, the problem 
\begin{align*}
    \widetilde{\mathcal{K}}_{\theta,0} f - g &=\const\quad\mbox{on }\partial B,\\
    \int_{\partial B}m_{\theta} f\, \mathrm{d}\Ha^1 & = 1,
\end{align*}
has a unique solution $f\in H^{k-1}(\partial B)$. It satisfies the estimate
\[
\|f\|_{H^{k-1}}\lesssim \|g\|_{H^k}+1.
\]
Moreover, for any fixed $g\in H^k(\partial B)$, the map $\theta\rightarrow f$ is Fr\'echet-smooth near $\theta=0$ as a function from $H^k(\partial B)$ to $H^{k-1}(\partial B)$.
\end{lemma}

\begin{proof}
We derive the statement of the lemma with the help of the implicit function theorem. For this purpose, we study the functional \begin{align*}\mathcal{J}: H^k(\partial B)\times H^{k-1}(\partial B) \to \left(H^{k}(\partial B)/\R \right) \times \R,\end{align*}
 given by
\begin{equation*}
\mathcal{J}(\theta,f) = \left(\widetilde\K_{\theta,0}f -g,\int_{\partial B} fm_{\theta}\, \mathrm{d}\Ha^1 -1\right).
\end{equation*}
It is well-defined and, by the virtue of Lemma \ref{L7} and Lemma \ref{DeriNormal},  Fr\'echet-smooth, provided that $\|\theta\|_{H^k}$ is small enough. Using \eqref{63} and the inverse of $\tilde{\mathcal{K}}_{0,0}$ defined via Lemma \ref{L6}, we furthermore notice that $\mathcal{J}(0,\widetilde{\K}_{0,0}^{-1}g + \frac{1}{2\pi})=(\const,0)$. Moreover, the differential
\[
\la \mathrm{D}_f \mathcal{J}(0,\frac{1}{2\pi}), \delta f\ra  = \left(\widetilde{\K}_{0,0}\delta f,\int_{\partial B}\delta f\, \mathrm{d}\Ha^1\right)
\]
is an isomorphism from $\delta f \in H^{k-1}(\partial B)$ to $H^k(\partial B)/\R\times \R$ thanks to Lemma \ref{L6}.

By the implicit function theorem, we thus find a unique solution $f_\theta$ to the identity $\mathcal{J}(\theta,f_{\theta})=(0,0)\in H^{k}(\partial B)/\R\times \R$, provided that $\|\theta\|_{H^k}$ is sufficiently small, and this solution $f_\theta$ is smooth in $\theta$.
\end{proof}

By choosing $g=0$ in the above statement, we derive the regularity of $\tilde \mu_{\theta,0}$ as defined in \eqref{def mu0}.

\begin{proposition}
\label{Lem:muDiffable}
The function $\tilde{\mu}_{\theta,0}$  is well-defined, at least for small enough $\theta$, and it is Fr\'echet smooth in $\theta$ as a map from $H^{k}(\de B)$ to $H^{k-1}(\de B)$. It holds $\tilde \mu_{0,0}=\frac1{2\pi}$, and its first derivative at $\theta=0$ is
\[
\la\left. \mathrm{D}_\theta\right|_{\theta=0}\tilde{\mu}_{\theta,0},\delta\theta\ra =-\frac{1}{2\pi}(\delta\theta-\frac{1}{2}\tilde{\mathcal{K}}_{0,0}^{-1}\delta\theta),
\]
provided that $\delta\theta\in H^k(\partial B)$ is a mean-zero function. 
%In particular, it holds that
%\[
%\la \left. D\right|_{\theta=0} (\K_{\theta,0}\tilde{\mu}_{\theta,0} ),\delta \theta\ra=0,
%\]
%for any such $\delta \theta$.
\end{proposition}

\begin{remark}\label{deri mu Fourier}
Written as a Fourier multiplier, this derivative at $0$ is \begin{align}
\cos(k\alpha)\rightarrow -\frac{1}{2\pi}(\cos(k\alpha)-|k|\cos(k\alpha)),
\end{align} by Lemma \ref{K inv}. In particular, the derivative is orthogonal to $x_1$, which reflects the translation invariance of the problem.
\end{remark}

\begin{proof}[Proof of Proposition \ref{Lem:muDiffable}]
Of course, choosing $g=0$ in Lemma \ref{L10}, we see that $\tilde{\mu}_{\theta,0}=f_{\theta}$ exists and is smooth in $\theta$.

In order to obtain the formula for the first derivative, we differentiate the identities defining $f_{\theta}$ and obtain
\begin{gather*}
\la \left. \mathrm{D}_{\theta} \right|_{\theta=0} \widetilde \K_{\theta,0}f_0,\delta \theta\ra + \widetilde{\K}_{0,0}\la \left. \mathrm{D}_{\theta} \right|_{\theta=0} f_{\theta}, \delta \theta\ra   \equiv 0,\quad\text{ modulo constants}\\
\int_{\partial B}\la  \left. \mathrm{D}_{\theta} \right|_{\theta=0} f_{\theta},\delta \theta\ra\, \mathrm{d}\Ha^1 + \int_{\partial B} f_0 \la  \left. \mathrm{D}_{\theta} \right|_{\theta=0} m_{\theta},\delta \theta\ra\, \mathrm{d}\Ha^1  =0.
\end{gather*}
Using the formulas for the derivatives of $\widetilde  \K_{\theta,0}$ and $m_{\theta}$ in the Lemmata \ref{L7} and \ref{DeriNormal} and recalling that $f_0=\frac{1}{2\pi}$, we find the identities
\[
\widetilde{\K}_{0,0}\left(\la \left.\mathrm{D}\right|_{\theta=0} f_\theta,\delta \theta\ra + \frac{1}{2\pi} \delta \theta\right) = \frac{1}{4\pi} \delta \theta +\const,\quad \int_{\partial B} \left(\la \left.\mathrm{D}\right|_{\theta=0} f_\theta,\delta \theta\ra + \frac{1}{2\pi} \delta \theta\right)\, \mathrm{d}\Ha^1=0.
\]
We know from Lemma \ref{K inv} that this problem has a unique solution, and the unknown constant is actually zero if $\delta\theta$ has zero mean. 

Furthermore, by the mean-freeness of $\delta\theta$, we have by Lemma \ref{L6} that the inverse $\widetilde{\K}_{0,0}^{-1}\delta\theta$ is well-defined and hence obtain by rearranging that \begin{align}
\la \left.\mathrm{D}\right|_{\theta=0} f_\theta,\delta \theta\ra=-\frac{1}{2\pi}\left(\delta\theta-\frac{1}{2}\widetilde{\K}_{0,0}^{-1}\delta\theta\right).
\end{align}
The proposition is proven.
\end{proof}

It will also be convenient for later to estimate the unknown constant in \eqref{def mu0}. 

\begin{lemma}\label{CapMaximal}
We have that $\widetilde \K_{\theta,0}\tilde{\mu}_{\theta,0}$ is smooth in $\theta\in H^k$ in a neighborhood of $0$ and it holds that \begin{align}
|\widetilde \K_{\theta,0}\tilde{\mu}_{\theta,0}|\lesssim \norm{\theta}_{H^k}^2\label{deri0}
\end{align}
for all small enough $\theta\in V^k$.
\end{lemma}

\begin{proof}
We already know that $\widetilde \K_{\theta,0}$ and $\tilde\mu_{\theta,0}$ are smooth, so smoothness follows from the product rule. To show \eqref{deri0}, it suffices to show that \begin{align}
\la\mathrm{D}_\theta\big|_{\theta=0}\widetilde  \K_{\theta,0},\delta\theta \ra \tilde{\mu}_{0,0}+\widetilde{\K}_{0,0}\la \mathrm{D}_\theta\big|_{\theta=0}\tilde{\mu}_{\theta,0},\delta\theta\ra=0,
\end{align}

\noindent for all $\delta\theta\in T_0V^k$. Note that all such $\delta\theta$ are orthogonal to constants by Proposition \ref{tspace M} Hence by Lemma \ref{L7} and \ref{Lem:muDiffable}, and because $\tilde \mu_{0,0}=\frac1{2\pi}$, this expression evaluates to \begin{align}
\widetilde{\K}_{0,0}\left(-\frac{1}{2\pi}\left(\delta\theta-\frac{1}{2}\widetilde{\K}_{0,0}^{-1}(\delta\theta)\right)\right)+\widetilde{\K}_{0,0}\left(\frac{1}{2\pi}\delta\theta\right)-\frac{\delta\theta}{4\pi}=0,
\end{align}
as desired.
\end{proof}

\subsection{The map $\K_{\theta,\eps}$}
In this subsection, we will analyze the asymptotic of $\K_{\theta,\eps}$ and show that this kernel is still invertible modulo constants (Lemma \ref{L15}).

\begin{lemma}\label{L5}
The kernel $K_{\theta,\eps}$ enjoys the asymptotic expansion \begin{align}
K_{\theta,\eps}(y,\tilde y)&= \frac{1}{2\pi}\left(\log\frac1{|y-\tilde y|}+\log \frac1{\eps} +\log(8)-2\right)\\
&\quad+ \eps \frac{y_1+\tilde y_1}{4\pi} \left(\log\frac1{|y-\tilde y|}+\log (8)-1+\log\frac1{\eps}\right)\nonumber\\
&\quad +\frac{1}{2\pi}\theta(\tilde y)\left(\log\frac1{|y-\tilde y|}+\log(8)-2+\log\frac1{\eps}\right)\nonumber\\
&\quad -\frac{1}{4\pi}\left(\theta(y)+\theta(\tilde y)\right)+R_{\theta,\eps}(y,\tilde y),\nonumber
\end{align}
and the linear operator  $\mathcal{R_{\theta,\eps}}$, associated to the remainder $R_{\theta,\eps}$ is an integral kernel mapping $H^{k-1}(\de B)$ to $H^{k}(\de B)$ with estimate
\[ 
\|\mathcal{R}_{\theta,\eps} f\|_{H^k} \lesssim  |\log\eps|(\norm{\theta}_{H^k}^2+\eps^2) \|f\|_{H^{k-1}},
\]
if $\norm{\theta}_{H^k}$ and $\eps$ are small enough.

Moreover $\mathcal{R_{\theta,\eps}}$ is continuously Fr\'echet differentiable in $(\theta,\eps)$ (away from $0$) as an operator from $H^{k-1}(\partial B)$ to $H^k(\partial B)$ with estimates
\begin{align*}
\|\mathrm{D}_{\eps}\mathcal{R}_{\theta,\eps} f\|_{H^k}  &\lesssim \left(\eps |\log \eps | + \frac1\eps \|\theta\|_{H^k}^2 \right) \|f\|_{H^{k-1}},\\
\| \la \mathrm{D}_{\theta}\mathcal{R}_{\theta,\eps} ,\delta \theta\ra f\|_{H^k} &\lesssim \left(\eps + \|\theta\|_{H^k}\right)|\log\eps| \|\delta \theta\|_{H^k}\|f\|_{H^{k-1}},
\end{align*}
if $\norm{\theta}_{H^k}$ and $\eps$ are small enough.
%operator norm $\lesssim \eps+\frac{\norm{\theta}_{H^k}^2}{\eps}+|\log\eps|\norm{\theta}_{H^k}$.
\end{lemma}

The linear maps associated with the explicit terms in this expansion are all also continuously Fr\'echet differentiable (except for the singularity of $\log\frac{1}{\eps}$ at $\eps=0$) with operator norms and derivatives of the obvious order as a straightforward application of Lemma \ref{L4} and Proposition \ref{Kernelsmooth} shows. This implies the following bounds, which we state without a proof.

\begin{lemma}\label{bd Ktheta}
The map $\mathcal{K}_{\theta,\eps}$ is bounded from $H^{k-1}(\de B)$ to $H^k(\de B)$ with estimate
\[
\|\mathcal{K}_{\theta,\eps} f\|_{H^k} \lesssim  |\log\eps|\|f\|_{H^{k-1}},
\]
for small enough $\theta$ and $\eps$ and it is continuously Fr\'echet differentiable in $(\theta,\eps)$ in these spaces away from $0$ with estimate
\begin{align*}
\|\mathrm{D}_{\eps}\mathcal{K}_{\theta,\eps} f\|_{H^k} &\lesssim \frac1{\eps} \|f\|_{H^{k-1}},\\
\|\la \mathrm{D}_{\theta}\mathcal{K}_{\theta,\eps},\delta \theta\ra f\|_{H^k} &\lesssim   |\log\eps| \|f\|_{H^{k-1}}\|\delta \theta\|_{H^k}.
\end{align*}
\end{lemma}

\begin{proof}[Proof of Lemma \ref{L5}]
For notational convenience, we introduce 
\begin{equation}\begin{aligned}
&s_1=|y(1+\theta(y))-\tilde{y}(1+\theta(\tilde{y}))|^2,\\
&s_2=\sqrt{(1+\eps y_1(1+\theta(y)))(1+\eps\tilde{y}_1(1+\theta(\tilde{y})))}.\label{s1s2}
\end{aligned}\end{equation}

\noindent In view of the definition  \eqref{200} of the kernel, the representation \eqref{43} and the expansion \eqref{SeriesK}, we  have that \begin{align}
K_{\theta,\eps}(y,\tilde y)& = m_{\theta}(\tilde y)\frac{s_2}{2\pi}F\left(\frac{\eps^2 s_1}{s_2^2}\right)\\
&= \frac{s_2 m_\theta(\tilde y)}{2\pi}\left(\log\left(\frac{1}{\sqrt{s_1}}\right)+\log(s_2)+\log(\frac{1}{\eps})+\log(8)-2+g\left(\eps^2 \frac{s_1}{s_2^2}\right)\right),
\end{align}

\noindent where $g(s):=F(s)-\log\left(\frac{8}{\sqrt{s}}\right)+2$ denotes the higher order terms of $F$. 
Taking a first-order Taylor expansion of $s_1,s_2$ and $m_\theta$ in $\eps$ and $\theta$ around $0$ and using Lemma \ref{L3} on the logarithm, one obtains the expansion above, where the remainders of this Taylor expansion and the contribution of $g$ make up the remainder term $R_{\theta,\eps}(y,\tilde y)$. 
It merely remains to argue the boundedness of the remainder.

We treat the contribution of $g$ and the other Taylor remainders differently.

After using the expansion in Lemma \ref{L3} of the logarithm, we see that all terms except $g(\eps^2\frac{s_1}{s_2^2})$ have the structure considered in Proposition \ref{Kernelsmooth} or are the product of $\log\eps$ with a term of this form. In particular they are Fr\'echet smooth in $(\theta,\eps)$ (resp.\ the product of $\log\eps$ with a smooth function)  by Proposition \ref{Kernelsmooth} and the Fr\'echet derivative and the pointwise derivative agree.

This implies that the pointwise Taylor expansion and the expansion with respect to the Fr\'echet derivatives of these terms agree and this part of the remainder has operator norm not larger than order $ |\log\eps|(\eps^2+\norm{\theta}_{H^k}^2)$ by Taylor's theorem. Furthermore, this remainder is differentiable with the corresponding bound by Taylor's theorem.

For $g$, we make use of the expansion of $F$ in Lemma \ref{LemmaF}.
This yields  \begin{equation}\begin{aligned}\label{remainderkernel}
\mel s_2m_\theta g\left(\eps^2\frac{s_1}{s_2^2}\right)= s_2 m_\theta f_1\left(\eps^2 \frac{s_1}{s_2^2}\right)+s_2m_\theta f_2\left(\eps^2\frac{s_1}{s_2^2}\right)(\log(s_1)+2\log\eps-2\log(s_2)).
\end{aligned}\end{equation}

\noindent After expanding $\log(s_1)$ in the second term with the help of Lemma \ref{L3}, we see that this is the sum of a kernel which has the form considered in Proposition \ref{Kernelsmooth} or $\log\eps$ times a kernel of this form. Hence this is also Fr\'echet smooth in $(\eps,\theta)$ (resp.\ $\log\eps$ times a smooth function) by Proposition \ref{Kernelsmooth}. As $f_1(0)=f_2(0)=0$, we see that the kernel in \eqref{remainderkernel} and its first derivative in $\eps$ are $0$ at $\eps=0$. Hence by the smoothness, the operator norm of the associated linear map is controlled by  $ \eps^2|\log\eps|$ and the norm of the derivative is bounded by $ \eps|\log\eps|$.\end{proof}

\begin{lemma}\label{L15}
The operator $\K_{\theta,\eps}$ is invertible modulo constants for $(\|\theta\|_{H^k},\eps)\in \mathcal{M}$ sufficiently small, in the sense that for any $g\in H^k(\partial B)$ and every $c\in \R $, there exists a unique function $f\in H^{k-1}(\partial B)$ such that
\begin{align}
\K_{\theta,\eps} f -g=\const,\quad \int_{\partial B} m_{\theta} f\dx=c.\label{int eq L15}
\end{align}
Moreover, there holds the continuity estimate
\begin{equation}
\|f\|_{H^{k-1}} \lesssim \|g\|_{H^k}+|c|.\label{est inv}
\end{equation}

\noindent The unknown constant in the first equation in \eqref{int eq L15} is at most of the order $|\log\eps|(\norm{g}_{H^k}+|c|)$.

Furthermore, for a fixed  $c$ and any smooth family $g=g_{\theta,\eps}\in H^k(\partial B)$, the function $f=f_{\theta,\eps}$ is continuously Fr\'echet differentable in $(\theta,\eps)$ for small $(\theta,\eps)\neq 0$ and it holds that 
\begin{align*}
    \|\mathrm{D}_{\eps} f\|_{H^{k-1}}&  \lesssim |\log \eps| (\|g\|_{H^{k}}+|c|)+\|\mathrm{D}_{\eps} g\|_{H^k},\\
    \|\la \mathrm{D}_{\theta}f, \delta \theta\ra\|_{H^{k-1}} & \lesssim \left(1+ |\log \eps| \|\theta\|_{H^k}\right)\|\delta\theta\|_{H^k} (\|g\|_{H^{k}}+|c| )+\|\la \mathrm{D}_{\theta} g,\delta \theta\ra \|_{H^k}.
\end{align*}
\end{lemma}

\begin{proof}It is enough to consider the case $c\not=0$, because the opposite case can be deduced via a compactness argument in the limit $c\to 0$. Then, upon rescaling, we may assume that $c=1$.

The statement of the lemma follows by a perturbation argument: We have seen in Lemma  \ref{L10} that $\widetilde\K_{\theta,0}$ is invertible. Moreover, by virtue of Lemma \ref{L5} it holds that
\begin{equation}\label{201}
\K_{\theta,\eps} f = \widetilde \K_{\theta,0} f + \mathcal{R} f+ \const ,
\end{equation}
where $\mathcal{R} f$ is a remainder term with kernel
\begin{align*}
R(y,\tilde y) &= \frac{m_{\theta}(\tilde y)}{2\pi}(s_2-1)\log \frac1{\sqrt{s_1}} + \frac{s_2m_{\theta}(\tilde y)}{2\pi}\left(\log s_2  +g\left(\frac{\eps^2 s_1}{s_2^2}\right)\right)\\
&\quad + \frac{1}{2\pi}\left(s_2m_{\theta}(\tilde y)-1\right)\left(\log \frac1\eps +\log 8-2\right),
\end{align*}
and $s_1 $ and $s_2 $ are given as in \eqref{s1s2}. Notice that all higher-order terms are absorbed into the constant in \eqref{201}. Arguing as in Lemma \ref{L5}, we see that the remainder satisfies the bound
\[
\|\mathcal{R} f\|_{H^k/\R} \lesssim |\log \eps|(\eps +\|\theta\|_{H^k}) \|f\|_{H^{k-1}}\lesssim \eps^{\frac{\ell}{2}}\norm{f}_{H^{k-1}},
\]

\noindent where we made use of the constraint $\norm{\theta}_{H^k}\lesssim \eps^\ell$ in the definition of $\mathcal{M}$ (cf.\ \eqref{def M}).
The integral constraint in our rescaled problem is the same as in Lemma \ref{L10}, and hence a Neumann series argument (with respect to the spaces $H^{k-1}(\partial B)$ and $H^k(\partial B)/\R $) translates the invertibility from $\widetilde \K_{\theta,0}$ to $\K_{\theta,\eps}$ and also shows differentiability away from $0$.  In addition, we have the estimates
\[
\|f\|_{H^{k-1}} \lesssim \|g\|_{H^k}+1.
\]
In order to derive estimates on the derivatives with respect to $\eps$, we notice that, due to the kernel expansion of Lemma \ref{L5}, it holds that 
\begin{align*}
\mathrm{D}_{\eps} K_{\theta,\eps}(y,\tilde y) =& -\frac1{2\pi\eps} m_{\theta}(\tilde y) + \frac{y_1+\tilde y_1}{4\pi} \left(\log \frac1{|y-\tilde y|} +\log (8)-2 +\log\frac1{\eps}\right) \\
&+\mathrm{D}_{\eps}R_{\theta,\eps}(y,\tilde y),
\end{align*}
where we have slightly redefined the remainder $R_{\theta,\eps}$ without changing its properties.  Differentiating the boundary problem \eqref{int eq L15} then gives
\[
\mathcal{K}_{\theta,\eps} \mathrm{D}_{\eps} f + \mathrm{D}_{\eps} \K_{\theta,\eps} f  -\mathrm{D}_{\eps} g= \const,
\]
and according to the 
above expansion of the kernel, the $O(1/\eps)$ term can be absorbed into the constant on the right-hand side since it is independent of $y$. Making use of the a priori estimates \eqref{est inv} and the bounds from Lemma \ref{L5} for the remainder then gives
\begin{align}
    \|\mathrm{D}_{\eps} f\|_{H^{k-1} }&\lesssim |\log\eps|\|f\|_{H^{k-1}} +  \left|\int_{\partial B} m_\theta\mathrm{D}_\eps f\dd \Ha^1\right|  +\|\mathrm{D}_{\eps} g\|_{H^k}  \\
    &\lesssim |\log \eps|\left(\|g\|_{H^k} +1\right)+\|\mathrm{D}_{\eps} g\|_{H^k},
\end{align}
where we used that $D_\eps\int_{\de B} m_\theta f\dd x=0$.

 The argumentation for the $\theta$-derivative is very similar. This time, the leading order $O(|\log\eps|)$ term can be absorbed into the constant.
 Using that \begin{align}
\mel\left|\int_{\partial B}m_\theta \la\mathrm{D}_\theta  f,\delta\theta\ra\dd\Ha^1\right|=\left|\int_{\de B}\la\mathrm{D}_\theta m_\theta,\delta\theta\ra f\dd\mathcal{H}^1\right|\\
&\lesssim \norm{f}_{H^{k-1}}\norm{\delta\theta}_{H^k}\lesssim (1+\norm{g}_{H^k})\norm{\delta\theta}_{H^k},
 \end{align}
 since $\mathrm{D}_\theta c=0$, we have that 
\begin{align*}
\mel \|\la \mathrm{D}_{\theta}f,\delta\theta\ra\|_{H^{k-1}}\\
&\lesssim \left|\int_{\partial B}m_\theta \la\mathrm{D}_\theta f,\delta\theta\ra\dd\Ha^1\right|  + \left(\eps  + \|\theta\|_{H^k} \right) |\log \eps| \|f\|_{H^{k-1}}\|\delta \theta\|_{H^k}+\|\la \mathrm{D}_{\theta} g,\delta \theta\ra \|_{H^k}  \\
&\lesssim \left(1+|\log \eps|\|\theta\|_{H^k}\right) (\|g\|_{H^k}+1 )\|\delta \theta\|_{H^k}+\|\la \mathrm{D}_{\theta} g,\delta \theta\ra \|_{H^k}.
\end{align*}
This finishes the proof.
\end{proof}

\subsection{The exact solution $\mu_{\theta,\eps}$}\label{subsec mu}
In this section, we finally conclude the existence of a continuously differentiable $\mu_{\theta,\eps}$ in Proposition \ref{mu diffable}, where $\mu_{\theta,\eps}$ is the solution of the boundary equation
 \begin{align}
\K_{\theta,\eps}\mu_{\theta,\eps}=\gamma+\frac{W}{2}(1+\eps(\chi_\theta)_1)^2,\label{des eq}
\end{align}
satisfying the integral constraint
  \begin{align}
\int_{\de B}m_\theta\mu_{\theta,\eps}\dx=1.
\end{align}
We recall that, at this point, $W$ and $\gamma$ are not defined yet.  Instead, identifying them appropriately will be part of the problem.

Anticipating that the leading order contribution of $\mu_{\theta,\eps}$ is given by the solution $\tilde \mu_{\theta,0}$ to the approximate problem studied in Subsection \ref{SS1} above, we make for any $S\in\R$ the ansatz
\begin{equation}
\label{74}
\mu_{\theta,\eps} (S) = \tilde{\mu}_{\theta,0} + \eps S (n_{\theta}\cdot e_1)\circ \chi_{\theta} +\mathcal{E}_{\theta,\eps}(S),
\end{equation}

\noindent where $\mathcal{E}_{\theta,\eps}(S)$ shall be an error term that is (more or less) of second order. The precise form of the first order in $\eps$ contribution is motivated by the corresponding term on the right-hand side of \eqref{des eq}.

We will construct $W$ and $\gamma$ as functions of $S$ by matching their asymptotics with that of the image of this ansatz under $\K_{\theta,\eps}$.
The speed $W$ is determined through the following lemma.

\begin{lemma}
    \label{L13}
    Let $(\theta,\eps)\in \mathcal{M}$ be small. Then, for any $S\in \R$, there exists a constant $W\in \R $ such that
    \begin{equation}
        \label{def 1st error}
        \begin{aligned}
       \mel \mathcal{K}_{\theta,\eps}\left(\tilde \mu_{\theta,0}  +\eps S (n_{\theta}\cdot e_1)\circ \chi_{\theta}\right)\\
      &  = \frac1{2\pi}\left(\log\frac1\eps +\log 8-2\right)  + \eps W \chi_{\theta}\cdot e_1 +\frac12 \eps^2 W(\chi_{\theta}\cdot e_1)^2 + \tilde{\mathcal{E}}_{\theta,\eps}(S) ,
        \end{aligned}
    \end{equation}
    where $\widetilde{\mathcal{E}}_{\theta,\eps}(S)$ is a smooth error term satisfying
\begin{align}
\|\widetilde{\mathcal{E}}_{\theta,\eps}(S)\|_{H^k}&\lesssim 
    |\log \eps|
    \left(\eps^2 +\|\theta\|^2_{H^k}\right) (1+|S|) ,\label{207}\\
   \|\la \mathrm{D}_{\theta}\widetilde{\mathcal{E}}_{\theta,\eps}(S),\delta \theta\ra \|_{H^k} &\lesssim |\log\eps| (\eps + \|\theta\|_{H^k})\|\delta \theta\|_{H^k}(1+|S|),\label{208}\\
    \|\mathrm{D}_{\eps} \widetilde{\mathcal{E}}_{\theta,\eps}(S)\|_{H^k}&\lesssim \left(\eps|\log \eps|  + \frac1{\eps}\|\theta\|_{H^k}^2\right)(1+|S|)\label{209}.
    \end{align}
\end{lemma}
Our proof actually reveals that 
\begin{align}
W=\frac{1}{4\pi}\left(\log 8-\frac{1}{2}+\log\frac{1}{\eps}\right)+\frac{1}{2}S\label{def W}
\end{align}
is an admissible choice.

We postpone the proof of this lemma for a moment.

Inserting the expansions  \eqref{74} and \eqref{def 1st error} into the boundary equation \eqref{des eq}, we obtain the relation  
\begin{align}\mathcal{K}_{\theta,\eps} \mathcal{E}_{\theta,\eps}(S)  + \widetilde{\mathcal{E}}_{\theta,\eps}(S) =\const.\label{def E}
\end{align} 
Moreover, as we have fixed the circulations in \eqref{320}  and  \eqref{def mu0} and because
\begin{align}
\int_{\partial B} m_\theta(n_{\theta}\cdot e_1)\circ \chi_{\theta}\dd\Ha^1 =\int_{\de\Omega_\theta}n_\theta\cdot e_1\dd \Ha^1=0 ,
\end{align}
we also have the integral condition
\begin{align}
\int_{\de B}m_\theta  \mathcal{E}_{\theta,\eps}(S)\dd \Ha^1=0.\label{def E 2}
\end{align}
Using this insight, we will obtain smallness and continuous differentiability of $\mathcal{E}_{\theta,\eps}$ via  Lemma \ref{L15}.

In fact, given $\widetilde{\mathcal{E}}_{\theta,\eps}$ from Lemma \ref{L13} and determining $\mathcal{E}_{\theta,\eps}$ via \eqref{def E} and \eqref{def E 2}, the formula in \eqref{74} can be actually   considered as a \emph{definition} of $\mu_{\theta,\eps}$. The desired boundary equation \eqref{des eq} will automatically hold then for some implicitly given $\gamma$, whose asymptotics we calculate below in Lemma \ref{asymp gamma}.

We stress that $\mu$, $W$, $\gamma$ and $\mathcal{E}_{\theta,\eps}$ are all affine functions in the reduction parameter $S$.

\begin{proposition}\label{mu diffable}
This construction does indeed yield a well-defined $\mu_{\theta,\eps}(S)\in H^k(\de B)$ for $(\theta,\eps)\in \mathcal{M}$  sufficiently small. The smallness condition is uniform in $S$. 

Moreover,  $\mu =\mu_{\theta,\eps}(S)$ is continuously Fr\'echet differentiable in the joint variable $(\theta,\eps,S)\in \mathcal{M}\times \R$ in a neighborhood of $(0,0)\times \R$, and it has   the same derivative in $\theta$ at $(\theta,\eps)=(0,0)$ as $\tilde{\mu}_{\theta,0}$.

Finally, the error functional satisfies the following estimates
\begin{align}
 \norm{\mathcal{E}_{\theta,\eps}(S)}_{H^{k-1}}&\lesssim (1+|S|)|\log\eps|(\eps^2+\norm{\theta}_{H^k}^2),\label{est error}\\
 \norm{\mathrm{D}_{ \eps}\mathcal{E}_{\theta,\eps}(S)}_{H^{k-1}}& \lesssim (1+|S|)|\log\eps|\left(\eps+\frac{\norm{\theta}_{H^k}^2}{\eps} \right),\label{est Derror_eps}\\
 \norm{\la \mathrm{D}_{\theta}\mathcal{E}_{\theta,\eps}(S),\delta \theta\ra }_{H^{k-1}}& \lesssim (1+|S|)|\log\eps|\left(\eps+ \norm{\theta}_{H^k} \right)\|\delta \theta\|_{H^k} .\label{est Derror_theta}
\end{align}
\end{proposition}

\begin{proof}[Proof of Proposition \ref{mu diffable} assuming Lemma \ref{L13}]
%The well-definedness is immediate from the construction and from Lemma \ref{L15}. The uniform

It follows from Lemma \ref{L13}, that the equation \eqref{def 1st error} does indeed hold. Notice that the   error $\mathcal{E}_{\theta,\eps}(S)$ is well-defined for small enough $(\theta,\eps)$, uniformly in $S$.
In particular, this shows the well-definedness of $\mu_{\theta,\eps}(S)$  and the uniformity of the smallness condition in $S$ follows from the fact that everything is affine in $S$.

Moreover, by Lemma \ref{L15} (applied to \eqref{def E} and \eqref{def E 2}) and \eqref{207}
  we see that
  \begin{align}
\norm{\mathcal{E}_{\theta,\eps}(S)}_{H^{k-1}}\lesssim \|\widetilde{\mathcal{E}}_{\theta,\eps}(S) \|_{H^k}\lesssim |\log\eps|(\eps^2+\norm{\theta}_{H^k}^2)(1+|S|),
\end{align}
which is \eqref{est error}.  To obtain the bounds on the derivatives, we argue similarly. For instance, again via Lemma \ref{L15} and \eqref{207} and \eqref{209}
\begin{align*}
    \|\mathrm{D}_{\eps} \mathcal{E}_{\theta,\eps}\|_{H^{k-1}} \lesssim |\log \eps|  \| \widetilde{\mathcal{E}}_{\theta,\eps}\|_{H^k}+\|\mathrm{D}_{\eps} \widetilde{\mathcal{E}}_{\theta,\eps}\|_{H^k}  \lesssim |\log \eps|\left(\eps +\frac1{\eps}\|\theta\|_{H^k}^2\right)(1+|S|).
\end{align*}
This proves \eqref{est Derror_eps}. Estimate \eqref{est Derror_theta} is derived analogously.

It is thus clear from the definition, Proposition \ref{Lem:muDiffable} and Lemma \ref{DeriNormal} that $\mu_{\theta,\eps}(S)$ is continuously Fr\'echet differentiable (in fact even smooth) except possibly at $(\theta,\eps)=(0,0)$, as all involved terms are continuously differentiable away from $(0,0)$. Differentiability in $(0,0)$ is a consequence of the error bounds \eqref{est Derror_eps} and \eqref{est Derror_theta} and our choice of $\ell$ in the definition of $\mathcal{M}$ in \eqref{def M}.
\end{proof}

We collect some trivial consequences.

\begin{corollary}
We have the following estimates for small enough $(\theta,\eps)\in \mathcal{M}$: \begin{align}
 \norm{\mathrm{D}_{\eps}\mu_{\theta,\eps}(0)}_{H^k}&\lesssim  |\log\eps|\left(\eps+\frac{\norm{\theta}_{H^k}^2}{\eps} \right)\label{cor est 1_eps},\\
 \norm{\la \mathrm{D}_{\theta}\mu_{\theta,\eps}(0),\delta \theta\ra }_{H^k}&\lesssim \|\delta \theta\|_{H^k}\label{cor est 1_theta},\\
 \norm{\mu_{\theta,\eps}(S)-\mu_{\theta,\eps}(0)-\eps S(n_\theta\cdot e_1)\circ\chi_\theta}_{H^k}&\lesssim |S||\log\eps|(\eps^2+\norm{\theta}^2)\label{cor est 2},\\
 \norm{\mu_{\theta,\eps}(S)-\mu_{\theta,\eps}(0)}&\lesssim \eps|S|,\label{cor est 3}\\
 \norm{\mathrm{D}_{ \eps}\bigl(\mu_{\theta,\eps}(S)-\mu_{\theta,\eps}(0)-\eps S(n_\theta\cdot e_1)\circ\chi_\theta\bigr)}_{H^k}&\lesssim |S||\log\eps|\left(\eps+\frac{\norm{\theta}_{H^k}^2}{\eps} \right),\label{cor est 4eps}\\
 \norm{\la \mathrm{D}_{\theta }\bigl(\mu_{\theta,\eps}(S)-\mu_{\theta,\eps}(0)-\eps S(n_\theta\cdot e_1)\circ\chi_\theta\bigr),\delta \theta\ra }_{H^k}&\lesssim |S||\log\eps|\left(\eps+ \norm{\theta}_{H^k}\right)\|\delta \theta\|_{H^k},\label{cor est 4theta}\\
 \norm{\mathrm{D}_{\eps}\bigl(\mu_{\theta,\eps}(S)-\mu_{\theta,\eps}(0)\bigr)}_{H^k}&\lesssim |S|,\label{cor est 5eps} \\
 \norm{\la \mathrm{D}_{\theta}\bigl(\mu_{\theta,\eps}(S)-\mu_{\theta,\eps}(0)\bigr),\delta \theta\ra }_{H^k}&\lesssim |S||\log\eps|(\eps+\norm{\theta}_{H^k})\|\delta \theta\|_{H^k}.\label{cor est 5theta}
\end{align}
Here, all derivatives are Fr\'echet derivatives.
\end{corollary}

\begin{proof}
All estimates are direct consequences of the construction of $\mu_{\theta,\eps}$, which is actually an affine function of $S$, its regularity, and the estimates on the error,
\eqref{cor est 1_theta} further uses Cor.\ \ref{Lem:muDiffable}.
The estimates on the differences  hold because \begin{align}
\mu_{\theta,\eps}(S)-\mu_{\theta,\eps}(0)-\eps S(n_\theta\cdot e_1)\circ\chi_\theta=\mathcal{E}_{\theta,\eps}(S) - \mathcal{E}_{\theta,\eps}(0),
\end{align}
and the right-hand side is linear in $S$. 
\end{proof}

We finally provide the proof of Lemma \ref{L13}.

\begin{proof}[Proof of Lemma \ref{L13}]
\textbf{Step 1. Expansion of $\K_{\theta,\eps} \tilde \mu_{\theta,0}$.} 
We start with the decomposition
\[
\mathcal{K}_{\theta,\eps}\tilde{\mu}_{\theta,0}=\widetilde{\mathcal{K}}_{\theta,0}\tilde{\mu}_{\theta,0}+(\mathcal{K}_{\theta,\eps}-\widetilde{\mathcal{K}}_{\theta,0})\tilde{\mu}_{\theta,0}.
\]
The first term on the right-hand side is a constant which is smooth in $\theta$ for small $\theta\in H^k(\partial B)$, and by Lemma \ref{CapMaximal} it holds that
\[
|\widetilde \K_{\theta,0} \tilde \mu_{\theta,0}| \lesssim \|\theta\|_{H^k}^2,\quad |\la \mathrm{D}_{\theta} (\widetilde \K_{\theta,0} \tilde \mu_{\theta,0}),\delta \theta\ra| \lesssim \|\theta\|_{H^k}\|\delta\theta\|_{H^k}.
\]
For the second term, we first expand $\tilde \mu_{\theta,0}$ with the help of Lemma \ref{Lem:muDiffable},
\[
\tilde \mu_{\theta,0}  = \frac1{2\pi}\left(1-\theta+\frac12\tilde{\mathcal{K}}_{0,0}^{-1}\theta\right) +O\left(\|\theta\|_{H^k}^2\right),
\]
where the error term is controlled in $H^{k-1}(\partial B)$, so that via Lemma \ref{bd Ktheta} and Lemma \ref{L7},
\begin{align*}
\mel     \|(\K_{\theta,\eps}-\widetilde\K_{\theta,0})\left(\tilde \mu_{\theta,0}  - \frac1{2\pi}\left(1-\theta+\frac12\tilde{\mathcal{K}}_{0,0}^{-1}\theta\right)\right) \|_{H^k}\\
&\lesssim |\log\eps| \|\tilde \mu_{\theta,0}  - \frac1{2\pi}\left(1-\theta+\frac12\tilde{\mathcal{K}}_{0,0}^{-1}\theta\right)\|_{H^{k-1}}\\
&\lesssim |\log \eps| \|\theta\|_{H^k}^2.
\end{align*}
Similarly, we find with the help of the chain rule that
\[
\left\|\left\la \mathrm{D}_{\theta} \left((\K_{\theta,\eps}-\widetilde\K_{\theta,0})\left(\tilde \mu_{\theta,0}  - \frac1{2\pi}\left(1-\theta+\frac12\tilde{\mathcal{K}}_{0,0}^{-1}\theta\right)\right)\right), \delta \theta\right\ra  \right\|_{H^k} \lesssim |\log \eps| \|\theta\|_{H^k}\|\delta \theta\|_{H^k}
\]
and 
\[
\|\mathrm{D}_{\eps}(\K_{\theta,\eps}-\widetilde\K_{\theta,0})\left(\tilde \mu_{\theta,0}  - \frac1{2\pi}\left(1-\theta+\frac12\tilde{\mathcal{K}}_{0,0}^{-1}\theta\right)\right) \|_{H^k}\lesssim \frac1{\eps}\|\theta\|_{H^k}^2.
\]

\noindent In order to identify the leading order terms in $\K_{\theta,\eps}\tilde\mu_{\theta,\eps}$, we are thus led to studying 
\begin{equation}
    \label{202}
(\K_{\theta,\eps} -\widetilde{\K}_{\theta,0})\frac1{2\pi}\left(1-\theta + \frac12\tilde{\mathcal{K}}_{0,0}^{-1}\theta\right).
\end{equation}
In view of  the expansion in Lemma \ref{L5} and the expansion of $K_{\theta,0}$ in the proof of Lemma \ref{L7}, we have that
\begin{align*}
 \mel    K_{\theta,\eps}(y,\tilde y) - \widetilde K_{\theta,0}(y,\tilde y) \\
 &=\frac1{2\pi} \left(\log\frac1{\eps} +\log 8-2\right)
  +\eps\frac{y_1+\tilde y_1}{4\pi} \left(\log\frac1{|y-\tilde y|} + \log 8-1+\log\frac1{\eps}\right)\\
    &\quad + \frac1{2\pi}\theta(\tilde y) \left(\log 8-2+\log\frac1{\eps}\right)  +\widetilde R_{\theta,\eps}(y,\tilde y),
\end{align*}
where $\widetilde{R}_{\theta,\eps}$ has the same bounds as $R_{\theta,\eps}$ in Lemma \ref{L5}. Up to error terms that are of the same order as those generated by the remainder $\widetilde{R}_{\theta,\eps}$, a short computation that uses the identities \eqref{63} and \eqref{64} of Lemma \ref{K inv} and the fact that $\widetilde{\mathcal{K}}_{0,0}\theta$ is a mean-zero function by Lemma \ref{L6} reveals that the leading order expansion in \eqref{202} is
\begin{align*}
    \frac1{2\pi} \left(\log \frac1{\eps} +\log 8-2\right) + \frac{\eps}{4\pi} \left(\log\frac1{\eps} +\log 8-\frac12\right) y_1 .
\end{align*}
Upon generating another error, we may replace $y_1$ by $\chi_{\theta}\cdot e_1$ in this formula. 

We thus conclude that
\begin{equation}
    \label{203}
    \K_{\theta,\eps}\tilde \mu_{\theta,0} =\frac1{2\pi}\left(\log\frac1{\eps} +\log 8-2\right) +\frac{\eps}{4\pi}\left(\log\frac1{\eps}+\log 8-\frac12\right)\chi_{\theta}\cdot e_1 +\widetilde{\mathcal{E}}_{\theta,\eps}^1,
\end{equation}
where $\widetilde{\mathcal{E}}_{\theta,\eps}^1$ has the properties
\begin{align}
    \|\widetilde{\mathcal{E}}_{\theta,\eps}^1\|_{H^k}&\lesssim |\log \eps|\left(\eps^2 +\|\theta\|^2_{H^k}\right),\label{204}\\
    \|\la \mathrm{D}_{\theta}\widetilde{\mathcal{E}}_{\theta,\eps}^1,\delta \theta\ra \|_{H^k} &\lesssim |\log\eps| (\eps + \|\theta\|_{H^k})\|\delta \theta\|_{H^k},\label{205}\\
    \|\mathrm{D}_{\eps} \widetilde{\mathcal{E}}_{\theta,\eps}^1\|_{H^k}&\lesssim \eps|\log \eps|  + \frac1{\eps}\|\theta\|_{H^k}^2.\label{206}
    \end{align}

\medskip

\noindent \textbf{Step 2. Expansion of $\eps\K_{\theta,\eps}(n_{\theta}\cdot e_1)\circ \chi_{\theta}.$}
We again use the expansion in Lemma \ref{L5} and see immediately that  the only relevant part of the kernel is 
\[
\frac{1}{2\pi}\left(\log\frac{1}{|y-\tilde{y}|}+\log(8)-2+\log\frac{1}{\eps}\right),
\]
because all other terms generate contributions that can be controlled as in \eqref{204}, \eqref{205}, and \eqref{206}.
Also, note that by Lemma \ref{bd Ktheta}, it holds that
\begin{align}
     \mel \norm{\int_{\de B}\frac{1}{2\pi}\left(\log\frac{1}{|y-\tilde{y}|}+\log(8)-2+\log\frac{1}{\eps}\right)(\tilde{y}_1-(n_\theta\cdot e_1)\circ\chi_\theta)\,\mathrm{d}\tilde{y}}_{H^k}\\
    &\lesssim |\log\eps|\norm{y_1-(n_\theta\cdot e_1)\circ\chi_\theta}_{H^{k-1}}\lesssim |\log\eps|\norm{\theta}_{H^k},
\end{align}

\noindent where the last estimate follows from $(n_\theta\cdot e_1)\circ\chi_\theta$ being in smooth in $\theta$ and being equal to $y_1$ at $\theta=0$. This term also has an $\eps$-derivative not larger than $ \eps^{-1} \norm{\theta}_{H^k}$ and a $\theta$-derivative not larger than $|\log\eps|$ by the same argument.

Hence it suffices to note that
\[
\int_{\partial B} \frac{1}{2\pi}\left(\log\frac{1}{|y-\tilde{y}|}+\log(8)-2+\log\frac{1}{\eps}\right)\tilde y_1\dd\Ha^1(\tilde y) = \frac12 y_1 ,
\]
by the virtue of Lemma \ref{K inv}. Again, we may substitute $y_1$ with $\chi_{\theta}\cdot e_1$ upon producing new error terms. 

In summary, we find that
\[
\eps \K_{\theta,\eps}(n_{\theta}\cdot e_1)\circ\chi_{\theta} = \frac{\eps}2\chi_{\theta}\cdot e_1 +\widetilde{\mathcal{E}}_{\theta,\eps}^2,
\]
with an error function controlled in the same manner as $\widetilde{\mathcal{E}}_{\theta,\eps}^1$ in \eqref{204}, \eqref{205} and \eqref{206}.

 \medskip

 \noindent\textbf{Step 3. Conclusion.} Combining the expansions of Step 1 and Step 2, we find that 
 \begin{align*}
\mel  \mathcal{K}_{\theta,\eps}\left(\tilde \mu_{\theta,0}  +\eps S (n_{\theta}\cdot e_1)\circ \chi_{\theta}\right)  - \eps W \chi_{\theta}\cdot e_1 -\frac12 \eps^2 W(\chi_{\theta}\cdot e_1)^2\\
& = \frac1{2\pi}\left(\log\frac1\eps +\log 8-2\right)\\
&\quad + \eps \left(\frac1{4\pi}\left(\log\frac1\eps  +\log 8 -\frac12\right) + \frac12 S - W\right)\chi_{\theta}\cdot e_1 \\
&\quad -\frac12\eps^2 W(\chi_{\theta}\cdot e_1)^2 + \tilde{\mathcal{E}}_{\theta,\eps}^1 + S \tilde{\mathcal{E}}_{\theta,\eps}^2.
 \end{align*}
  The second line is zero precisely if we choose $W=W(S)$ as in \eqref{def W}. This way, the first term in the final line becomes an error term. The estimates \eqref{207}, \eqref{208} and \eqref{209} on the error follow immediately from \eqref{204}, \eqref{205} and \eqref{206}.  
\end{proof}

It remains to compute $\gamma$:

\begin{lemma}\label{asymp gamma}
It holds that \begin{align}
\gamma=&\frac{1}{2\pi}\left(\log 8 + \log\frac{1}{\eps} -2\right) -\frac{W}{2}+O\left(|\log\eps|^2(1+|S|)(\eps^2+\norm{\theta}_{H^k}^2)\right).
\end{align}
\end{lemma}
\begin{proof}
It follows from Lemma \ref{L13} and the definition of $W$ that we have \begin{align}
\mathcal{K}_{\theta,\eps}\left(\tilde{\mu}_{\theta,0}+S\eps(n_\theta\cdot e_1)\circ\chi_\theta\right)=&\frac{1}{2\pi}\left(\log8+\log\frac{1}{\eps}-2\right)-\frac{W}2+\frac{1}{2}W(1+\eps(\chi_\theta)_1)^2\\
&+\tilde{\mathcal{E}}_{\theta,\eps}(S).\nonumber
\end{align}
By adding $\K_{\theta,\eps}\mathcal{E}(S)$ to both sides of the equation and using the definition of $\mu_{\theta,\eps}$ in \eqref{74}, we see that \begin{align} 
\K_{\theta,\eps}\mu_{\theta,\eps}=\frac{1}{2\pi}\left(\log8+\log\frac{1}{\eps}-2\right)-\frac{W}2+\frac{1}{2}W(1+\eps(\chi_\theta)_1)^2+\left(\tilde{\mathcal{E}}(S)+\K_{\theta,\eps}\mathcal{E}(S)\right),
\end{align}

\noindent and hence, by \eqref{des eq}, we must have that
\begin{align}
\gamma=\frac{1}{2\pi}\left(\log8+\log\frac{1}{\eps}-2\right)-\frac{W}2+\left(\tilde{\mathcal{E}}(S)+\K_{\theta,\eps}\mathcal{E}(S)\right),
\end{align}

\noindent because the term $\tilde{\mathcal{E}}(S)+\K_{\theta,\eps}\mathcal{E}(S)$ is a constant by definition \eqref{def E}.  By Lemma \ref{L15} and estimate \eqref{207}, this constant is controlled by
\[
|\tilde{\mathcal{E}}(S)+\K_{\theta,\eps}\mathcal{E}(S)| \lesssim |\log\eps| \|\tilde{\mathcal{E}}(S)\|_{H^k} \lesssim |\log\eps|^2 \left(\eps^2 +\|\theta\|_{H^k}^2\right)(1+|S|).
\]
\end{proof}

For the uniqueness of the solution, it will also be crucial that the values of $W$ and $\gamma$ determine $S$, if the circulation is given.

\begin{lemma}\label{bd det S}
Suppose $(\eps,\theta)\in \mathcal{M}$ are small and let $\tilde \gamma$ and $\tilde W$ be given such that there is some $\tilde \mu\in H^{k-1}$ with \begin{align}
\K_{\theta,\eps}\tilde\mu=\tilde\gamma+\frac{1}{2}\tilde{W}(1+\eps(\chi_\theta)_1)^2,\quad \int_{\de B}m_\theta \tilde\mu\dd\Ha^1=1,\label{mu eq}
\end{align}

\noindent then there is a unique $\tilde S$ such that $\tilde\mu=\mu_{\theta,\eps}(\tilde S)$. Furthermore $|\tilde S|\lesssim \tilde W+|\log\eps|$.
\end{lemma}
\begin{proof}
First, we note that the set of $\tilde W,\tilde\gamma$ for which there is a solution to \eqref{mu eq} is an affine subspace of $\R^2$. It clearly contains the set of all $W,\gamma$ given through the construction above, which is a 1-dimensional affine subspace as $W$ is non-constant in $S$ by definition \eqref{def W}. It also cannot be bigger since it would otherwise be the full $\R^2$ and would in particular contain $\R\times \{0\}$, which violates the uniqueness statement of Lemma \ref{L15}. Hence the set of $\tilde W,\tilde\gamma$ for which there is a solution to \eqref{mu eq} agrees with the set associated with $\mu_{\theta,\eps}(\tilde S)$.

The estimate on $\tilde S$ then follows from the asymptotic \eqref{def W}.
\end{proof}

\section{Proof of the main theorem}\label{S:IFT}
We finally turn to the proof of our main result, for which we construct solutions to the jump condition \eqref{41}, that is,
\[
\rho \lambda_{\theta,\eps}^2 -\mu_{\theta,\eps}(S)^2 + \beta_{\eps} h_{\theta,\eps} = \const.
\]
The proofs for both cases with and without surface tension are very similar, so will focus on the first case mostly and comment on the minor adaptations needed for the second case.

Let us consider the functionals
$\mathcal{F}:\mathcal{M}\times \R\rightarrow H_{\sym}^{k-2}(\de B)/\R$ (the manifold $\mathcal{M}$ was defined in \eqref{def M}), defined as
\begin{align}
\mathcal{F}(\theta,\eps,S)=\frac{1}{\beta_{\eps}}\left(\rho\lambda_{\theta,\eps}^2-\mu_{\theta,\eps}(S)^2\right)+ h_{\theta,\eps}.
\end{align}
if $\beta\eps >0$,
 and  
 \begin{align}
\mathcal{F}(\theta,\eps,S)=\left(\rho\lambda_{\theta,\eps}^2-\mu_{\theta,\eps}(S)^2\right)
\end{align}
if $\beta\eps =0.
$ 
\noindent As explained in Subsection \ref{proof strat}, it suffices to find $\theta$ such that $\mathcal{F}(\theta,\eps,S)=0$ (in $H^{k-2}_{\sym}(\partial B)/\R$) to show existence. It will be convenient to consider 
\[
\omega_{\eps} = \frac1{\beta_{\eps}}
\]
in the case of positive surface tensions. By the assumptions \eqref{405}, \eqref{406}, the function $\eps\mapsto \omega_{\eps}$ is $C^1$ and fulfills
\begin{equation}
\label{300}
\omega:= \lim_{\eps\to 0} \omega_{\eps} \in [0,\infty).
\end{equation}

 Below we will frequently suppress the arguments $\theta$ and $\eps$ in  $\mathcal{F}(S)$, $\lambda$, $h$ and $\mu(S)$.
We will first carry out the reduction and then apply the implicit function theorem.

\subsection{The reduction}

We shall split $\mathcal{F}=\mathcal{F}_1+\mathcal{F}_2$, where $\mathcal{F}_1$ is the projection of $\mathcal{F}$ onto the span of $y_1$ (with respect to the standard scalar product in $L^2(\de B)/\R$), that is,
\[
\mathcal{F}_1 = \frac{y_1}{\pi} \la \mathcal{F},\tilde y_1\ra ,
\]
and $\mathcal{F}_2$ is the projection onto the orthogonal complement of this. As explained in Subsection \ref{proof strat}, $\mathcal{F}_1$ is the part that we cannot control with the implicit function theorem due to the translation invariance of the limiting problem and which hence requires a Lyapunov--Schmidt reduction. 

In this first step, our goal is to ensure that we  can pick $S=S_{\theta,\eps}$ as a function of $\eps$ and $\theta$ uniquely such that $\mathcal{F}_1(S_{\theta,\eps})=0$ for small enough $(\eps,\theta)\in \mathcal{M}$, under the constraint 
\begin{align}
S_{\theta,\eps} |\log \eps|\left(\eps^2 + \|\theta\|^2_{H^k}\right) \le C_0,
\end{align}
for some sufficiently small universal constant $C_0$.
In the next subsection, we will show that the contribution to $\mathcal{F}_2$ through $S$ as a function of $(\theta,\eps)$ is continuously Fr\'echet differentiable in $(\eps,\theta)\in \mathcal{M}$ with derivative $0$ at $0$ in $H^{k-1}(\partial B)$.

%Instead of $\mathcal{F}_1$ we study $\scalar{\mathcal{F}}{x_1}$, which is $0$ iff $\mathcal{F}_1$ is. \smallskip

Because $\mu$ is affine in $S$ and the other terms are independent of $S$, the functional $\mathcal{F}_1$ is a quadratic polynomial in $S$, of which we first compute the asymptotics of the coefficients. 

We first note that it holds that $\mathcal{F}(0,0,0)=\const$,   because $\lambda_{0,0}=-2$ by \eqref{phi0}, $\mu_{0,0}(0)=\frac1{2\pi}$ by Proposition \ref{Lem:muDiffable} and \eqref{74} and $h_{0,0}=1$, and because $( \beta_{\eps})^{-1}$ is bounded in the case of positive surface tension by our assumption in \eqref{300}.   In particular, it must hold that $\mathcal{F}_1(0,0,0)=0$.

We now expand $\mathcal{F}_1$ into a first-order Taylor series. Using the smoothness $\lambda_{\theta,\eps}$ established in Proposition \ref{C1} and recalling Remark \ref{SeriesDirToNeumann}, we see that 
\begin{align*}
    \la \lambda^2,y_1\ra & = 4\la 1,y_1\ra +2\eps \la y_1,y_1\ra -8 \la \mathcal{L} \theta-\theta,y_1\ra +O\left(\eps^2 +\|\theta\|_{H^k}^2\right)\\
    & = 2\pi \eps +O\left(\eps^2 +\|\theta\|_{H^k}^2\right),
\end{align*}
because $\la 1,y_1\ra =0$. Similarly, using the smoothness of $h_{\theta,\eps}$ by   Proposition   \ref{DeriCurvature}, we observe that 
\begin{align}
    \la h,y_1\ra & = \la 1,y_1\ra +\eps \la y_1,y_1\ra -\la \theta+\partial_{\tau}^2 \theta,y_1\ra +O\left(\eps^2 +\|\theta\|_{H^k}^2\right)\\
    & = \eps \pi +O\left(\eps^2 +\|\theta\|_{H^k}^2\right),
\end{align}
because via two integrations by parts, it holds that $\la \partial_{\tau}^2 \theta,y_1\ra  = -\la \theta,y_1\ra$. 
 The term $\mu_{\theta,\eps}(0)$ is not smooth, but the second-order error in its expansion is of order $|\log\eps|(\eps^2+\norm{\theta}_{H^k}^2)$ by Proposition \ref{mu diffable}. 
Using this proposition and Proposition \ref{Lem:muDiffable} together with  Remark \ref{deri mu Fourier}, we find that
\begin{align}
    \la \mu(0)^2,y_1\ra & = \frac1{4\pi^2}\la 1,y_1\ra -\frac1{2\pi^2}\la \delta \theta-\frac12\tilde{\mathcal{K}}_{0,0}^{-1}\delta \theta,y_1\ra +O\left(|\log \eps|\left(\eps^2 +\|\theta\|_{H^k}^2\right)\right)\\
    &  = O\left(|\log \eps|\left(\eps^2 +\|\theta\|_{H^k}^2\right)\right),
\end{align}
if we expand along a geodesic (in the sense that we take the image of a line under the diffeomorphism between a neighborhood of $0$ in $\mathcal{M}$ and an open set of $\Span\{1,y_1\}^\perp\subset H_{\sym}^k(\de B)$ as described in Subsection \ref{tspace M}) so that  $\delta \theta\in T_0V^k$ satisfies $\|\delta \theta\|_{H^k}\lesssim \|\theta\|_{H^k}$.
Combining these expansions, in sum, we obtain that \begin{align}\label{const S term}
\scalar{\mathcal{F}_{\theta,\eps}(0)}{y_1}=\frac{2\rho\eps\pi}{\beta_{\eps}}+\pi\eps+O\left(|\log\eps|(\norm{\theta}^2+\eps^2)\right).
\end{align}
Moreover, similar calculations reveal that
\begin{align}
    \mathrm{D}_{\eps} \la \mathcal{F}_{\theta,\eps}(0),y_1\ra  &= \partial_{\eps} \left(\frac{2\rho \eps \pi}{\beta_{\eps}} +\pi\eps\right) + O\left(|\log\eps|\left(\eps + \frac{\|\theta\|_{H^k}^2}{\eps}\right)\right),\label{306}\\
    \la \mathrm{D}_{\theta}\la \mathcal{F}_{\theta,\eps}(0),x_1\ra,\delta\theta\ra & = O\left(|\log\eps|(\eps+\|\theta\|_{H^k})\|\delta \theta\|_{H^k}\right)\label{307}.
\end{align}

Similarly, we can compute the asymptotic of the terms depending on $S$, using the fact that only the contribution of $\mu$ depends on $S$. We write 
\begin{equation}\label{211}
\beta_{\eps}\scalar{\mathcal{F}(S)-\mathcal{F}(0)}{y_1}=2 \scalar{\mu(0)(\mu(0)-\mu(S))}{y_1}-\scalar{(\mu(0)-\mu(S))^2}{y_1}.
\end{equation}

\noindent By construction and \eqref{cor est 2}, we have that $\mu(S)-\mu(0)=\eps S(e_1\cdot n_\theta)\circ \chi_\theta$ up to a term which is at most of the order $  |S||\log\eps|(\eps^2+\norm{\theta}_{H^k}^2)$. Also we have \begin{equation}\norm{y_1-(e_1\cdot n_\theta)\circ \chi_\theta}_{H^{k-1}}\lesssim \norm{\theta}_{H^k},\end{equation} as this is smooth and vanishing  at $\theta=0$,  with a corresponding bound    on the derivative.
Hence, we have for the first term on the  right-hand side of \eqref{211} that
\begin{align*}
2 \scalar{\mu(0)(\mu(0)-\mu(S))}{y_1} &= -2\eps S\scalar{\mu(0)(n_{\theta}\cdot e_1)\circ \chi_{\theta}}{y_1} + O\left(|S||\log\eps|\left(\eps^2 +\|\theta\|_{H^k}^2\right)\right)\\
& =  -\eps S+ O\left(|S||\log\eps|\left(\eps^2 +\|\theta\|_{H^k}^2\right)\right),
\end{align*}

\noindent and the error term has a derivative one order lower.
For the second term, we may use a similar  computation via \eqref{cor est 2} %that by Proposition \ref{mu diffable} we have \begin{align}
%\mu(S)-\mu(0)=S\eps x_1+O(|S||\log\eps|(\eps^2+\norm{\theta}_{V^k}^2))
%\end{align}
 and noticing that $\scalar{y_1}{y_1^2}=0$, to see that
\[
\la (\mu(0)-\mu(S))^2,y_1\ra = O\left(S^2 \eps |\log \eps|(\eps^2+\|\theta\|^2_{H^k})\right).
\]
Together, this gives us that
\begin{equation}\begin{aligned}\label{higher term S}
  \beta_{\eps}\scalar{\mathcal{F}(S)-\mathcal{F}(0)}{y_1}
=&-\eps S+O\left(|S||\log\eps|\left(\eps^2+\norm{\theta}_{H^k}^2\right)\right) \\
&+O\left(S^2 \eps |\log\eps|\left(\eps^2+\norm{\theta}_{H^k}^2\right)\right). 
\end{aligned}\end{equation}
For the derivatives, we perform a similar calculation, using the assumption \eqref{324} on the derivative of $\beta\eps$, and find that
\begin{equation}\label{308}
\begin{aligned}
 \mel   \mathrm{D}_{\eps} \la \mathcal{F}{(S)}-\mathcal{F}(0),y_1\ra \\
 & = - S \partial_{\eps}\left(\frac{\eps}{ \beta_{\eps}}\right) + O\left(|S|\frac{|\log\eps|}{\beta_{\eps}} \left(\eps+\frac{\|\theta\|^2_{H^k}}{\eps}\right)\right) + O\left(S^2 \frac{|\log\eps|}{\beta_{\eps}} \left(\eps^2+\|\theta\|_{H^k}^2\right)\right),
 \end{aligned}
 \end{equation}
 and 
 \begin{equation}
     \label{309}
\begin{aligned}
\mel | \la \mathrm{D}_{\theta}\la \mathcal{F}(S)-\mathcal{F}(0),y_1\ra, \delta \theta\ra|\\
&  = O\left(|S|\frac{\eps+\norm{\theta}_{H^k}}{\beta_{\eps}} |\log\eps|\|\delta\theta\|_{H^k}\right) + O\left(S^2 \eps \frac{|\log\eps|}{\beta_{\eps}}(\eps +\|\theta\|_{H^k})\|\delta\theta\|_{H^k}\right).
\end{aligned}
 \end{equation}
All the error terms are either linear or quadratic in $S$.

Hence we see from \eqref{const S term} and \eqref{higher term S}, after dividing by $\eps$ that $\mathcal{F}_1(S)=0$ is a quadratic equation in $S$, namely 
\begin{equation}\label{305}
\begin{aligned}
 \mel   \frac{2\rho\pi}{\beta_{\eps}} + \pi + O\left(|\log\eps|\left(\eps+\frac{\|\theta\|_{H^k}^2}{\eps}\right)\right) -\left(1+O\left(|\log\eps|\left(\eps+\frac{\|\theta\|_{H^k}^2}{\eps}\right)\right)\right)\frac{S}{\beta_{\eps}}\\
 & = O\left(\eps\beta_{\eps}|\log\eps| \left(\eps+\frac{\|\theta\|_{H^k}^2}{\eps}\right)\right) \left(\frac{S}{\beta_{\eps}}\right)^2.
\end{aligned}
\end{equation}

\noindent It will be more convenient to consider this as an equation in $t=\frac{S}{\beta_{\eps}}$. We thus are concerned with a quadratic equation of the form
\noindent 
\[
a-bt =t^2 c,
\]
\noindent where
\begin{align}
    a & = \frac{2\rho\pi}{\beta_{\eps}} + \pi + O\left(|\log\eps|\left(\eps+\frac{\|\theta\|_{H^k}^2}{\eps}\right)\right),\\
    b& = 1+O\left(|\log\eps|\left(\eps+\frac{\|\theta\|_{H^k}^2}{\eps}\right)\right),\\
    c&= O\left(\eps\beta_{\eps}|\log\eps| \left(\eps+\frac{\|\theta\|_{H^k}^2}{\eps}\right)\right).
\end{align}
According to our assumption that $\norm{\theta}_{H^k}\lesssim \eps^\ell$ (see \eqref{def M}),
the equation is certainly solvable for $t\in \R$ under the assumption \eqref{404}, which was that   
\begin{equation}
    \label{312}
\lim_{\eps\to 0}  \eps\beta_{\eps} =0.
\end{equation}
 The two unique solutions are then (approximately) given by
\[
t=\frac{a}{b } +O(c)\quad\mbox{and}\quad t = -\frac{b}{c} - \frac{a}b + O(c),
\]
and thus, if $\theta$ and $\eps$ are small enough, $t\approx a/b$ is the unique  solution satisfying, say, $|ct|\le 1/2$ uniformly in $\eps\ll1$.

On the level of $S$, the argument shows that there exists a constant $C_0$ such that there is precisely one solution $S=S_{\theta,\eps}$ satisfying 
\[
|S_{\theta,\eps}| |\log \eps|\left(\eps^2 + \|\theta\|^2_{H^k}\right) \le C_0,
\]
for $(\theta,\eps)\in \mathcal{M}$ small. It has the leading order asymptotics 
\begin{equation}\label{311}
S_{\theta,\eps}= 2\rho \pi + \beta_{\eps}\pi + O\left(\beta_{\eps}|\log\eps|\left(\eps+\frac{\|\theta\|_{H^k}^2}{\eps}\right)\right),
\end{equation}
where we used the bound \eqref{323} on $\beta_\eps$,
and thus 
\begin{equation}
\label{bd S}
|S_{\theta,\eps}|\lesssim \beta_{\eps}.
\end{equation}

Because all error terms occurring in \eqref{305} are continuously differentiable in $(\theta,\eps)$, the (real-valued) implicit function theorem ensures that $S_{\theta,\eps}$ depends in a neighborhood of $(0,0)$ continuous differentiably on $(\theta,\eps)$. Differentiating the defining equation $\la \mathcal{F}(S),y_1\ra=0$ and using the assumption \eqref{324} (which says $|\de_\eps\beta_\eps|\lesssim \beta_\eps \eps^{-1}$) as well as the identities \eqref{306},  \eqref{211}, \eqref{308}, and \eqref{bd S} or, respectively, \eqref{307}, \eqref{211}, \eqref{309}, and \eqref{311} we find on the one hand that
\begin{align}
 \mel    \left(\mathrm{D}_{\eps}S_{\theta,\eps}  \right)\frac{\eps}{\beta_{\eps}}\left(1+O\left(|\log\eps|(1+\eps\beta_{\eps})\left(\eps+\frac{\|\theta\|_{H^k}^2}{\eps}\right)\right)\right) \\
 &= \de_\eps\left(\frac{2\rho\pi\eps +\eps\beta_{\eps}\pi}{\beta_{\eps}}\right) + O\left(|\log\eps|(1+\eps\beta_{\eps})\left(\eps+\frac{\|\theta\|_{H^k}^2}{\eps}\right)\right),
\end{align}
and thus,   
\begin{equation}
    \label{310}
    \frac{\eps}{\beta_{\eps}}|\mathrm{D}_{\eps}S_{\theta,\eps}| \lesssim 1,
\end{equation}
for $(\theta,\eps)\in \mathcal{M}$ sufficiently small, where we have used \eqref{312}.
On the other hand, it holds that
\begin{align}
    \mel \la \mathrm{D}_{\theta} S_{\theta,\eps},\delta \theta\ra \frac{\eps}{\beta_{\eps}}\left(1+O\left((1+\eps\beta_{\eps})|\log\eps|\left(\eps+\frac{\|\theta\|_{H^k}^2}{\eps}\right)\right)\right)\\
    & = O\left(\left(|\log\eps|  +\eps\beta_{\eps} |\log\eps|\right)(\eps +\|\theta\|_{H^k}) \|\delta \theta\|_{H^k}\right),
\end{align}
for $(\theta,\eps)\in \mathcal{M}$ sufficiently small, and thus,
\begin{equation}\label{315}
    \frac{\eps}{\beta_{\eps}}|\la \mathrm{D}_{\theta}S_{\theta,\eps},\delta \theta\ra | \lesssim \|\delta \theta\|_{H^k},
\end{equation}
assuming that \eqref{312} and \eqref{300}
are satisfied.

 The same calculation can also be made for the functional without surface tension, where, writing $S^{\sigma=0}_{\theta,\eps}$ instead of $S_{\theta,\eps}$ for clarity, one gets
\begin{equation}\label{bd alt S}
S^{\sigma=0}_{\theta,\eps} = 2\pi\rho + O\left(|\log\eps|\left(\eps + \frac{\|\theta\|_{H^k}^2}{\eps}\right)\right)
\end{equation}
for the unique solution satisfying
\[
S_{\theta,\eps}^{\sigma=0}|\log\eps|\left(\eps^2+\|\theta\|_{H^k}^2\right) \le C_*,
\]
for $(\theta,\eps)\in \mathcal{M}$  small. The derivatives are controlled as follows:
\begin{equation}\label{bd deri alt S}
\eps|\mathrm{D}_{\eps} S_{\theta,\eps}^{\sigma=0}| \lesssim 1,\quad \eps|\la \mathrm{D}_{\theta} S_{\theta,\eps}^{\sigma=0},\delta \theta\ra | = O\left(|\log \eps| \left(\eps +\frac{\|\theta\|_{H^k}^2}{\eps}\right)\|\delta \theta\|_{H^k}\right).
\end{equation}

\subsection{The contribution of $S$ to $\mathcal{F}_2$}
Next, we show that the contribution of $S_{\theta,\eps}$ to $\mathcal{F}_2$ is neglectable, i.e.\ that $\mathcal{F}_2(0,0,S_{0,0})=0$ in $H^{k-2}(\partial B)/\R$ (or $H^{k-1}(\partial B)/\R$ in the case $\beta_\eps=0$) and that $\mathcal{F}_2(\theta,\eps,S_{\theta,\eps})$ is continuously Fr\'echet differentiable near $0$ with $\mathrm{D}_{\theta,\eps}\big|_{\theta,\eps=0}\mathcal{F}_2(\theta,\eps,S_{\theta,\eps})=\mathrm{D}_{\theta,\eps}\big|_{\theta,\eps=0}\mathcal{F}_2(\theta,\eps,0)$.

%We will do this for any $S(\theta,\eps)$ satisfying weaker bounds than \eqref{bd S} and \eqref{bd deri S} (but not necessarily satisfying $\mathcal{F}_1(S)=0$) in order to obtain a stronger uniqueness statement. 

%The bounds that we will assume are \begin{align}
%|S|\lesssim& \sigma\eps+\eps^{-\frac{1}{2}+\delta'}\label{new bd S}\\
%|\mathrm{D}_{\theta,\eps} S|\lesssim & \sigma+\eps^{-\frac{3}{2}+\delta'} \label{new bd deri S}
%\end{align}

%for some $\delta'>0$. Clearly, this is weaker than \eqref{bd S} and \eqref{bd deri S}.

%For the case without surface tension, we will prove this for any $S_{\sigma=0}(\theta,\eps)$ which satisfies the bounds \begin{align}
%&|S_{\sigma=0}(\theta,\eps)|\lesssim \eps^{-\frac{1}{2}+\delta'}\label{new bd alt S}\\
%&|\mathrm{D}_{\theta,\eps}\lesssim \eps^{-1+\delta'}\label{new bd deri alt S}
%\end{align}

%for some $\delta'>0$ which is clearly weaker than \eqref{bd alt S} and \eqref{bd deri alt S} as $\norm{\theta}_{H^k}\lesssim \eps^\ell$.

 Let $J$ be the projection onto the orthogonal complement of $y_1$ (with respect to the usual $L^2$ scalar product).
Using that only $\mu$ depends on $S$, we may write \begin{align}
\mathcal{F}_2(S)-\mathcal{F}_2(0)=-\frac{1}{\beta_{\eps}}J \Big(2(\mu(S)-\mu(0))\mu(0)+(\mu(S)-\mu(0))^2\Big) .
\end{align}

\noindent Most of the first summand disappears since we modded out $y_1$. Indeed, by Proposition \ref{mu diffable} we know that \begin{align}
\mu(0)=\frac{1}{2\pi}+O(\eps+\norm{\theta}_{H^k}),\label{exp mu}
\end{align}

\noindent where the error term has a derivative $\lesssim 1$. On the other hand, by the same Proposition, we know that \begin{align}
\mu(S)-\mu(0)=\eps Sy_1+O(|S||\log\eps|(\eps^2+\norm{\theta}_{H^k}^2)).
\end{align}

%\noindent here the derivative of the error term in $(\theta,\eps)$ is $\lesssim |S|(\eps|\log\eps|+\frac{\norm{\theta}_{H^k}^2}{\eps}+\norm{\theta}_{H^k}|\log\eps|)$.

\noindent The product of the main terms here belongs to $\Span\{y_1\}$ and disappears thus under the projection operator $J$. In particular, we have \begin{align}\label{main term}
\norm{J\left((\mu(S)-\mu(0))\mu(0)\right)}_{H^{k-1}}\lesssim |S||\log\eps|(\eps^2+\norm{\theta}_{H^k}^2).
\end{align}

\noindent By \eqref{cor est 3} it holds that\begin{align}\label{est S}
\norm{\mu(S)-\mu(0)}_{H^{k-1}}\lesssim \eps|S|,
\end{align}

\noindent hence, using the bound \eqref{bd S} on $S_{\theta,\eps}$ we see that
\begin{align}
\norm{\mathcal{F}_2(S_{\theta,\eps})-\mathcal{F}_2(0)}_{H^{k-1}}&\lesssim \frac{1}{\beta_{\eps}}\bigg(\norm{J((\mu(S_{\theta,\eps})-\mu(0))\mu(0))}_{H^{k-1}}\\
&\qquad+\norm{\mu(S_{\theta,\eps})-\mu(0)}_{H^{k-1}}^2\bigg)\nonumber\\
&\lesssim \frac{|S_{\theta,\eps}|}{\beta_{\eps}}|\log\eps|(\eps^2+\norm{\theta}_{H^k}^2)+\frac{S_{\theta,\eps}^2\eps^2}{\beta_{\eps}}\\
&\lesssim |\log \eps |(\eps^2+\|\theta\|_{H^k}^2) + \beta_{\eps} \eps^2,
\end{align}
and thus, in particular,
\begin{equation}\label{313}
\frac1{\eps}\norm{\mathcal{F}_2(S_{\theta,\eps})-\mathcal{F}_2(0)}_{H^{k-1}} =o(1),
\end{equation}
for small $(\theta,\eps)\in \mathcal{M}$. The convergence of the last term is a consequence of \eqref{312}.

 In the case without surface tension, one has the same bound without the factor $\frac{1}{\beta_{\eps}}$, which still converges to $0$.

In either case, the observation implies that
\[
\mathcal{F}_2(\theta,\eps,S_{\theta,\eps})  = \mathcal{F}_2(\theta,\eps,0) +o(1).
\]
Next, we need to argue that $\mathcal{F}_2(S_{\theta,\eps})$ is still continuously Fr\'echet differentiable and, in fact, has the same derivative as $\mathcal{F}_2(0)$ at $0$. As all involved terms are continuously differentiable away from $0$, and $\mathcal{F}_2(0,\theta,\eps)$ is continuously Fr\'echet differentiable by   Propositions \ref{DeriCurvature}, \ref{bulkDiffable} and \ref{mu diffable}, it suffices to show that $\mathrm{D}_{\theta,\eps}(\mathcal{F}_2(S_{\theta,\eps})-\mathcal{F}_2(0))\rightarrow 0$ in $H^{k-1}$ as $(\theta,\eps)\to (0,0)$.

To account for the terms that disappear due to the projection, we rewrite \begin{align}
\mel\mathcal{F}_2(S)-\mathcal{F}_2(0)\\
&= -\frac{1}{\beta_{\eps}}J\left(2(\mu(S)-\mu(0))\mu(0) - \frac{\eps Sy_1}{2\pi} +(\mu(S)-\mu(0))^2\right)\\
& =-\frac1{\beta_{\eps}} J\left(2(\mu(S)-\mu(0)-\eps S y_1)\mu(0) +\eps S y_1 \left(\mu(0)-\frac1{2\pi}\right) +(\mu(S)-\mu(0))^2\right).
\end{align}
where the additional summand $\eps S y_1$ in the middle is annihilated by $J$.

Using the product rule, we then  compute the derivative and   estimate
\begin{align}
  \norm{\text{D}_{\eps}(\mathcal{F}_2(S_{\theta,\eps})-\mathcal{F}_2(0))}_{H^{k-1}} 
&\lesssim \left|\de_\eps\left(\frac{1}{\beta_{\eps}}\right)\right|\beta_{\eps} \|\F_2(S_{\theta,\eps})-\mathcal{F}_2(0)\|_{H^{k-1}}\\
&\quad +\frac{1}{ \beta_{\eps}}\norm{\mathrm{D}_{\eps}\big((\mu(S)-\mu(0)-\eps Sy_1)\mu(0)\big)}_{H^{k-1}}\nonumber\\
&\quad +\frac{1}{\beta_{\eps}}\norm{\mathrm{D}_{\eps}\left(\eps Sy_1\left(\mu(0)-\frac{1}{2\pi}\right)\right)}_{H^{k-1}}\\
&\quad +\frac{1}{ \beta_{\eps}}\norm{\mathrm{D}_{\eps}(\mu(S)-\mu(0))^2}_{H^{k-1}}\nonumber\\
&=:\mathrm{I}+\mathrm{II}+\mathrm{III}+\mathrm{IV},\nonumber
\end{align}

\noindent where $\mathrm{I}$ - $\mathrm{IV}$ stand for the terms in each line. 
(In the case without surface tension, the term $\mathrm{I}$ disappears, and the prefactor $\frac{1}{\beta_{\eps}}$ needs to be replaced with $1$.) Derivatives with respect to $\theta$ are estimated analogously. We omit the details.

Using the assumption in \eqref{324} and our last estimate \eqref{313}, we know that \begin{align}
\mathrm{I} \lesssim \frac1{\eps}\|\F_2(S_{\theta,\eps})-\F_2(0)\|_{H^{k-1}} \to 0,
\end{align}
as $(\theta,\eps)\to (0,0)$ in $\mathcal{M}$.

To estimate $\mathrm{II}$, we use the product rule and the fact that $H^{k-1}$ is closed under products to estimate 

\begin{align}
\mathrm{II}&\lesssim \frac{1}{ \beta_{\eps}}\biggl(\norm{\mu(S_{\theta,\eps})-\mu(0)-\eps S_{\theta,\eps} y_1}_{H^{k-1}}\norm{\mathrm{D}_{\eps}\mu(0)}_{H^{k-1}}\\
&\quad +\norm{\mu(0)}_{H^{k-1}}\norm{\mathrm{D}_{\eps}\bigl(\mu(S_{\theta,\eps})-\mu(0)-\eps S_{\theta,\eps} y_1\bigr)}_{H^{k-1}}\biggr).\label{sum II}
\end{align}

\noindent Invoking \eqref{cor est 2}, the first summand is controlled by $ \frac{1}{ \beta_{\eps}}|S_{\theta,\eps}||\log\eps|(\eps^{2}+\norm{\theta}_{H^k}^2)$, which is vanishing as a consequence of  the bound \eqref{bd S} on $S_{\theta,\eps}$.
For bounding the second summand, we use that $\mu$ is affine in $S$ and write
\begin{align}
\mel \mathrm{D}_{\eps}\bigr(\mu(S_{\theta,\eps})-\mu(0)-\eps S_{\theta,\eps} y_1\bigl)\\
&=S_{\theta,\eps}\mathrm{D}_{\eps}(\mu(1)-\mu(0)-\eps y_1)+(\mu(1)-\mu(0)-\eps y_1)\mathrm{D}_{\eps}S_{\theta,\eps}.
\end{align} 
Applying the bounds \eqref{bd S} and \eqref{310} on $S_{\theta,\eps}$ and    its derivative, and  the estimates \eqref{cor est 2} and \eqref{cor est 4eps}, we find that
\[
\|\mathrm{D}_{\eps}\bigr(\mu(S_{\theta,\eps})-\mu(0)-\eps S_{\theta,\eps} y_1\bigl)\|_{H^{k-1}} \lesssim \beta_{\eps} |\log\eps |\left(\eps + \frac{\|\theta\|^2_{H^k}}{\eps}\right),
\]
and thus, also the second summand vanishes as $(\theta,\eps)\to (0,0)$ in $\mathcal{M}$.

We proceed analogously with the third term. Using the product rule and \eqref{bd S} and \eqref{310}, we find
\[
\mathrm{III} \lesssim \|\mu(0) - \frac1{2\pi}\|_{H^{k-1}} +\eps \|\mathrm{D}_{\eps}\mu(0)\|_{H^k}.
\]
The right-hand side is asymptotically vanishing as a consequence of 
\eqref{cor est 1_eps} and the asymptotics $\mu(0) = \tilde \mu +\mathcal{E}(0)  = \frac1{2\pi} +o(1)$.

 Finally, to bound the term $\mathrm{IV}$, we use again that $\mu(S)$ is affine and estimate 
 \begin{align}
 \mathrm{IV}
&\lesssim \frac{1}{ \beta_{\eps}} \norm{\mu(S_{\theta,\eps})-\mu(0)}_{H^{k-1}} |\mathrm{D}_{\eps}S_{\theta,\eps}|\norm{\mu(1)-\mu(0)}_{H^{k-1}}\\
&\quad + \frac{1}{ \beta_{\eps}} \norm{\mu(S_{\theta,\eps})-\mu(0)}_{H^{k-1}}|S_{\theta,\eps}|\norm{\mathrm{D}_{\eps}(\mu(1)-\mu(0))}_{H^{k-1}}  .\nonumber
\end{align}
Making use of \eqref{bd S}, \eqref{310}, \eqref{cor est 2} and \eqref{cor est 4eps}, we find
\[
\mathrm{IV} \lesssim \eps |S_{\theta,\eps}|  \lesssim \eps\beta_{\eps},
\]
which is vanishing as $\eps \to 0$ by the virtue of \eqref{312}.

The case $\beta_\eps=0$ proceeds the same way, the term $\mathrm{I}$ drops out, and for the other terms, one can make the exact same calculations with the bounds \eqref{bd alt S} and \eqref{bd deri alt S} on $S^{\sigma=0}$.

\subsection{The implicit function theorem}
It remains to apply the implicit function theorem to the functional $(\theta,\eps)\mapsto \mathcal{F}_2(\theta,\eps,S_{\theta,\eps})$ defined in a small neighborhood of the origin $(0,0)$ in $\M$. It takes values in the space
\begin{align}
\mathcal{N}:=\{f\in H_{\sym}^{k-2}(\de B)/\R:\: \scalar{f}{y_1}=0\}
\end{align} 

\noindent (resp.\ $k-1$, if $\beta_\eps=0$), which is isomorphic to $T_0V^{k-2}$ (resp.\ $T_0V^{k-1}$) by Proposition \ref{tspace M} via the canonical identification.
To apply the implicit function theorem, we only need to show that the derivative (with respect to  $\theta$) of $\mathcal{F}_2(\theta,\eps,S_{\theta,\eps})$ is invertible from $T_0V^k$ to $\mathcal{N}$.
By the previous step, we only need to consider the derivative of $\mathcal{F}_2(\theta,\eps,0)$ at $(0,0)$, which in fact equals the derivative of $\mathcal{F}(\theta,\eps,0)$ at $(0,0)$, projected onto the orthogonal complement of $y_1$.

By the product rule and the Propositions \ref{DeriCurvature}, \ref{C1}, \ref{Lem:muDiffable}, and \ref{mu diffable}  this derivative, tested against $\delta\theta$, equals 
\[
\la \mathrm{D}_{\theta}\F(0,0,0),\delta \theta\ra  = -8 \omega \rho (\mathcal{L}\delta \theta -\delta \theta) +\frac{\omega }{2\pi^2}(\delta\theta -\frac12\widetilde{\mathcal{K}}_{0,0}^{-1}\delta \theta) -\delta \theta - \partial_{\tau}^2 \delta \theta,
\]
where $\omega$ was introduced in \eqref{300}.

 Using the Remarks \ref{SeriesDirToNeumann} and \ref{deri mu Fourier}, we may express this derivative as the Fourier multiplier
\[
\widehat{\la \dots \ra } (l) = \left(\omega\left(8\rho+\frac1{2\pi^2}\right)(1-|l|)-1+l^2\right)\widehat{\delta \theta}(l).
\]
Because our tangent vectors are symmetric, $\widehat{\delta \theta}(l)=\widehat{\delta \theta}(-l)$, see  Proposition \ref{tspace M}, it is enough to consider nonnegative wave numbers $l$. Furthermore, because they are mean-zero functions and orthogonal on $x_1=\cos(\alpha)$, it holds $\widehat{\delta\theta}(0)=\widehat{\delta\theta}(1)=0$, and we may hence restrict to $l\ge 2$.

 The derivative  is thus an invertible map from $T_0V^{k}$ to $\mathcal{N}$ if \begin{align}
\left|\omega\left(8\rho+\frac{1}{2\pi^2}\right)-(1+l)\right|\gtrsim \frac{l^2}{l-1}
\sim l,\end{align}

\noindent  for all $l\in\N_{\geq 2}$.

We remark that if $\omega=0$, this is always the case. In general,  invertibility is thus ensured precisely if
\[
\omega\left(8\rho+\frac1{2\pi^2}\right)\not\in \N_{\ge 3}.
\]

\noindent which follows from the assumption \eqref{322}. In the case without surface tension, one gets instead the Fourier multiplier
\begin{align}
\left(8\rho+\frac{1}{2\pi^2}\right)(1-l),
\end{align}
which is always invertible because its modulus is bounded below by $O(l)$.

 In either case, we obtain existence from the implicit function theorem.

\subsection{Asymptotic formulas}  

Next, we compute the asymptotics. First, note that (a posteriori!) we have that $\theta=\theta_{\eps}$ is differentiable with respect to $\eps$ and, in particular, $\norm{\theta_{\eps}}_{H^k}\lesssim \eps$.

By the definition of $W$ in \eqref{def W}, and the formulas for $S$ in \eqref{311} and \eqref{bd alt S} we have that \begin{align}
W&=\frac{1}{4\pi}\left(\log 8-\frac{1}{2}+\log\frac{1}{\eps}\right)+\frac{S}{2}\\
&=\frac{1}{4\pi}\left(\log 8-\frac{1}{2}+\log\frac{1}{\eps}\right)+\rho\pi+\frac{\beta_{\eps}\pi}{2} +o(1),
\end{align}

\noindent both for the case with and without surface tension.

Similarly, for the flux constant, we obtain via  Lemma \ref{asymp gamma} that \begin{align}
\gamma&=\frac{1}{2\pi}(\log8+\log\frac{1}{\eps}-2)-\frac{W}2 +o(1)\\
&=\frac{3}{8\pi}\log\left(\frac{8}{\eps}\right)-\frac{15}{16\pi}-\frac{\rho\pi}2-\frac{ \beta_{\eps}\pi}{4}+o(1).
\end{align}

\noindent Regarding the unknown constant in the equation $\F_2(S_{\theta_{\eps},\eps},\theta_{\eps},\eps) = \const$, which we denote by $\nu_{\eps}$ in what follows, we have that
\[
\nu_{\eps} = \F_2(0,\theta,\eps) +o(\eps)
\]
by the virtue of \eqref{313}. Averaging over $\partial B$, we may write
\[
\nu_{\eps} = \frac1{2\pi} \int_{\partial B} \left(\frac{\rho}{\beta_{\eps}} \lambda_{\theta_{\eps},\eps}^2 -\frac1{\beta_{\eps} }\mu_{\theta_{\eps},\eps}(0)^2 +h_{\theta_{\eps},\eps}\right)\dd \Ha^1 +o(\eps).
\]
Making now use of the asymptotics implied by Propositions \ref{DeriCurvature}, \ref{C1}, \ref{Lem:muDiffable}, and \ref{mu diffable} and invoking orthogonality properties, we deduce that
\[
\nu_{\eps} = \frac{1}{\beta_{\eps}}\left(4\rho -\frac1{4\pi^2}\right) + 1 +o(\eps).
\]
The error is differentiable and vanishing at $\eps=0$.

To compute the derivative of $\theta$ with respect to $\eps$, we first note
\[
0 = \left.\mathrm{D}_{\eps}\right|_{\eps=0} \left(\F_2(0,\eps,0)-\nu_{\eps}\right) +\la \left.\mathrm{D}_{\theta}\right|_{\theta=0} \F_2(\theta,0,0),\left.\partial_{\eps}\right|_{\eps=0}\theta_{\eps}\ra
\]
by our analysis in the previous subsection. In view of the above-mentioned propositions, the first term on the right-hand side is vanishing. Thanks to the invertibility of $\mathrm{D}_{\theta}\F(0,0,0)$, it follows that $\partial_{\eps} \theta_{\eps}$ is zero at $\eps=0$,
and thus
\[
\theta_{\eps} = o(\eps),
\]
by Taylor's theorem.

\subsection{Uniqueness and regularity}
For a given $S_{\eps,\theta}$ and $\eps$ such that the bounds \eqref{bd S}, \eqref{310} and \eqref{315} (resp.\  \eqref{bd alt S} and \eqref{bd deri alt S}) hold, the solution $\theta_{\eps}\in V^k$ is unique if $\eps$ is small enough by the implicit function theorem, provided that the extra condition $\norm{\theta}_{H^k}\leq\eps^\ell$ holds. 
%There is a slight subtlety here: In the uniqueness statement of the theorem,  uniqueness should hold for all $W$ (and therefore also for all $S$) fulfilling a bound, whereas here $S$ is a function of $\eps$.

We have shown in the first step that $S_{\eps,\theta}$ is uniquely determined through $\mathcal{F}_1(\theta,\eps,S_{\eps,\theta})=0$ for small $\eps$, as long as $S_{\theta,\eps} |\log \eps|\left(\eps^2 +\|\theta\|_{H^k}^2\right)=o(1)$ holds due to the bounds on the discarded second solution.

Now suppose that there is some other solution   $\tilde \theta,\tilde \varphi_{\mathrm{int}},\tilde  \varphi_{\mathrm{ext}}, \tilde  W,\tilde \gamma$ fulfilling \eqref{32}-\eqref{52} that is not the one we have constructed. Then $\tilde \varphi_{\mathrm{int}}$ is uniquely determined by $\tilde \theta$ as solutions to the Laplace equation are unique. Moreover, $\tilde  \varphi_{\mathrm{ext}}$ gives rise to some $\tilde \mu$, and by applying the fundamental solution $\mathcal{K}_{\tilde \theta,\eps}$ as in Subsection \ref{reform}, this $\tilde \mu$ also uniquely determines $\tilde  \varphi_{\mathrm{ext}}$. By Lemma \ref{bd det S}, this $\tilde \mu$ must be of the form $\mu(\tilde S)$ for some $\tilde S$ fulfilling $|\tilde S|\lesssim |\tilde W| + |\log\eps|$
and thus
\begin{align}
|\tilde S||\log \eps|\left(\eps^2 +\|\theta\|_{H^k}^2\right) \lesssim |\tilde W ||\log \eps|\left(\eps^2 +\|\theta\|_{H^k}^2\right) +o(1) =2C_*,
\end{align} 
by our assumption on $\tilde W$ in \eqref{321}, presuming $C_*$ is chosen sufficiently small.

Since it must hold that $\mathcal{F}_1(\tilde\theta,\eps,\tilde{S})=0$, we have $\tilde S=S(\tilde\theta,\eps)$ by the first step, where $S(\cdot,\cdot)$ is the same function for all solutions $(\tilde\theta,\tilde \varphi_{in},\tilde \varphi_{out},\tilde W,\tilde \gamma)$.

But this implies by the implicit function theorem that $\tilde{S}$ and $\tilde\theta$ must be the solution we constructed, contradicting non-uniqueness.

The smoothness of $\theta$ follows from the fact that $k\geq 5$ was arbitrary, and hence $\Omega_\theta$ is smooth. The smoothness of $\vin$ and $\vout$ follows from classical elliptic regularity, see e.g.\ \cite{gilbarg1977elliptic}.\hfill\qedsymbol

\begin{funding}	This work was produced while the first author was a doctoral student at the University of M\"unster. It  is funded by the Deutsche Forschungsgemeinschaft (DFG, German Research Foundation) under Germany's Excellence Strategy EXC 2044--390685587, Mathematics M\"unster: Dynamics -- Geometry  -- Structure, and through grant number 531098047. 
\end{funding}
\begin{ack} The authors thank Franck Sueur and Lukas Niebel for valuable discussions.
\end{ack}

	\bibliography{thinRings}
	\bibliographystyle{acm}

\end{document}